\documentclass[11pt,reqno,a4paper]{article}
\usepackage{mathrsfs}
\usepackage{amsmath}
\usepackage{amsthm}
\usepackage{amsfonts}
\usepackage{amssymb}
\usepackage[dvips]{graphicx}
\usepackage[dvips]{color}
\usepackage{latexsym}
\usepackage{enumerate}
\usepackage{xspace}

\newtheorem{thm}{Theorem}[section]
\newtheorem{cor}[thm]{Corollary}
\newtheorem{lemma}[thm]{Lemma}
\newtheorem{prop}[thm]{Proposition}

\newtheorem{rem}[thm]{Remark}
\numberwithin{equation}{section}

\def \bN {\Bbb N}

\def \and {\, \mbox{\rm and}\, }

\def \span {\,{\rm span}\,}

\begin{document}

\title{\bf Representer Theorems in Banach Spaces: Minimum Norm Interpolation, Regularized Learning and Semi-Discrete Inverse Problems}
\author{Rui Wang\thanks{School of Mathematics, Jilin University, Changchun 130012, P. R. China. E-mail address: {\it rwang11@jlu.edu.cn}. } \quad and \ Yuesheng Xu\thanks{
Department of Mathematics and Statistics, Old Dominion University, Norfolk, VA 23529, USA. This author is also a Professor Emeritus of Mathematics, Syracuse University, Syracuse, NY 13244, USA. E-mail address: {\it y1xu@odu.edu.} All correspondence should be sent to this author.} }

\date{}

\maketitle{}

\begin{abstract}
Constructing or learning a function from a finite number of sampled data points (measurements) is a fundamental problem in science and engineering. This is often formulated as a minimum norm interpolation problem, regularized learning problem or, in general, a semi-discrete inverse problem, in certain functional spaces. The choice of an appropriate space is crucial for solutions of these problems. When they are considered in an infinite dimensional Hilbert space, which gives a linear method for their solutions, the remarkable representer theorem reduces their solutions to finding a {\it finite} number of coefficients of elements in the space. Motivated by sparse representations of the reconstructed functions such as compressed sensing and sparse learning, much of the recent research interest has been directed to considering these problems in certain Banach spaces in order to obtain their sparse solutions, which is a feasible approach to overcome challenges coming from the big data nature of most practical applications. It is the goal of this paper to provide a systematic study of the representer theorems for these problems in Banach spaces. There are a few existing results for these problems in a Banach space, with all of them regarding implicit representer theorems. We aim at obtaining explicit representer theorems based on which convenient solution methods will then be developed. For the minimum norm interpolation, the explicit representer theorems enable us to express the infimum in terms of the norm of the linear combination of the interpolation functionals. For the purpose of developing efficient computational algorithms, we establish the fixed-point equation formulation of solutions of these problems. We reveal that unlike in a Hilbert space, in general, solutions of these problems in a Banach space may not be able to be reduced to {\it truly} finite dimensional problems (with certain infinite dimensional components hidden). We demonstrate how this obstacle can be removed, reducing the original problem to a truly finite dimensional one, in the special case when the Banach space is $\ell_1(\mathbb{N})$.
\end{abstract}

\textbf{Key words}: Representer theorem, minimum norm interpolation, regularized learning, sparse learning, semi-discrete inverse problem, Banach space.

\textbf{2010 Mathematics Subject Classification}: Primary: 68Q32, 41A05, 46A22, 45Q05, 65D10; Secondary: 46B45.
\section{Introduction}

A core issue in data science is to learn or construct a function from a finite number of sampled data points. This may be modeled as an interpolation problem, an optimization problem or, a semi-discrete inverse problem. Learning such a function is an ill-posed problem in the sense that a small error in the sampled data will result in a large error in the reconstructed function. Because sampled data inevitably contain noise, the ill-posedness of problems of this type is unavoidable. It is well-recognized that minimum norm interpolation and the regularization method are effective approaches to treat the ill-posedness. A typical regularization problem consists of two terms, a fidelity and a regularization. The fidelity term is designed to measure the data fidelity while the regularization term is taken as a constraint on the space from which the target function is to be chosen. The goal of this paper is to systematically study the solution representation of the three types of problems: minimum norm interpolation, regularized learning and regularized semi-discrete inverse problems in Banach spaces. Regularized learning and regularized semi-discrete inverse problems are originated from different sources. The semi-discrete inverse problem often refers to a physical problem described by a physical law expressed via certain integral equation, which relates its solution with a finite number of measurements. While a regularized learning problem has a fidelity term describing certain learning network (not necessarily a physical law). However, these two types of problems have the same general mathematical formulation in the sense that both of them have a fidelity term involving a finite number of measurements and a regularization term specifying prior solution information. For this reason, we will not distinguish them in this paper since as far as the solution representation is concerned, there is little distinct between them. For simplicity we may use the term the ``regularization problem" to refer to both of these problems, when necessary.

In the regularization problem, classical regularization methods aim at finding a target function in a reproducing kernel Hilbert space (RKHS) on which point-evaluation functionals are continuous, from a finite number of point-evaluation functionals. The point-evaluation functionals on an RKHS can be represented by the reproducing kernel \cite{Aron}. The success of these regularization methods lies on the celebrated representer theorem \cite{AMP, CO, CS, KW70, SHS, Wen}, which states that a solution of the regularization problem is a linear combination of the kernel function with one of its variable evaluated at the given data points. The earliest form of the representer theorem in a Hilbert space may be traced back to \cite{deBoor}.

To motivate the research problem studied in this paper, we review several classical results of representer theorems in Hilbert spaces. To our best knowledge, the first repersenter theorem for minimum norm interpolation in a Hilbert space was established by de Boor and Lynch in 1966 \cite{deBoor} in the context of spline interpolation. We give below a brief review of this inspiring result. This practical example of minimum norm interpolation is for geometric shape design \cite{deBoor77}. Suppose that we are given values $y_j$, $j=0, 1,\dots, m+1,$ of a real-valued function $s$, defined on $[0,1]$, at a finite number of points $0=x_0<x_1<\cdots<x_{m+1}=1$. We wish to construct $s$ so that its graph has the minimum concavity. Mathematically, this problem can be described by finding a function $s\in C^2[0,1]\cap C^4([0,1]\setminus \Omega)$, with $\Omega:=\{x_j: j=1,2,\dots,m\}$ such that
\begin{equation}\label{cubic}
\|s''\|_{L_2[0,1]}=\min\left\{\|f''\|_{L_2[0,1]}:  f\in C^2[0,1]\cap C^4([0,1]\setminus \Omega), f(x_j)=y_j, j=0, 1,\dots, m+1\right\}.
\end{equation}
The minimization problem \eqref{cubic} seeks its solution in the infinite dimensional space $C^2[0,1]\cap C^4([0,1]\setminus \Omega)$. However, it was proved in \cite{deBoor} that the solution of this problem lies in a finite dimensional subspace. In fact, $s$ is a cubic spline with knots $\Omega$ satisfying the interpolation conditions
$$
s(x_j)=y_j, \ \ \mbox{for}\ \ j=0,1,\dots, m+1.
$$
A general form of this problem was investigated in \cite{deBoor}.
About the same time, Kimeldorf and Wahba \cite{KW70} considered the univariate L-spline smoothing problem in a Hilbert space and obtained a representer theorem for its solution. The representer theorem of Kimeldorf and Wahba has been found applicable to the solution of the semi-discrete inverse problem \cite{KLW, Wen}. A multivariate version of the L-spline smoothing problem was investigated by De Figueiredo and Chen in \cite{GChen}, giving a representer theorem for its solution. In the context of machine learning, the representer theorem for the solution of the regularized empirical risk minimization on an RKHS was established by Sch\"olkopf, Herbrich, and Smola in \cite{SHS}. Argyriou, Micchelli and Pontil in \cite{AMP} gave necessary and sufficient conditions to ensure a general regularized empirical risk minimization problem in an RKHS has a representer theorem. Moreover, we obtained in our recent work \cite{WX, WX1} representer theorems for solutions of the regularized semi-discrete inverse problems in the functional reproducing kernel Hilbert spaces naturally introduced by the inverse problems.

The representer theorem is useful in both theory and computation. From a theoretical point of view, representer theorems for the regularization problem reveal exactly in what sub-class its solution lies. According to the representer theorem, a solution of the problem is a linear combination of the reproducing kernel with one of its variable evaluated at given data points. This leads to the study of universality of a kernel \cite{MXZ06}, which gives necessary and sufficient conditions on a kernel so that a linear combination of the kernel with one of its variable evaluated at given data points can arbitrarily approximate a given continuous function. Moreover, motivated by representing a solution of the regularized learning in a multiscale manner, refinement of a reproducing kernel was studied in \cite{XZ2007, XZ2009, ZXZ2012}. From a practical standpoint, representer theorems for either minimum interpolation or regularization problems are useful, because they dramatically reduce an infinite dimensional problem to a finite dimensional one whose solution can be obtained by solving either a linear system or a finite dimensional optimization problem. In other words, representer theorems provide a theoretical basis for the reduction of the general machine learning problem or the regularized semi-discrete inverse problem in Hilbert spaces to discrete problems that can be solved by implementable computer algorithms.

Compared to Hilbert spaces, Banach spaces with more choices of norms enjoy more geometric structures, which will be beneficial for developing effective methods for solving the function reconstruction problem. Several recent research directions point to consideration of minimum norm interpolation, regularized learning or semi-discrete inverse problems in Banach spaces. Compress sensing \cite{CRT, D} motivates us to study minimum norm interpolation or the regularized learning problem in Banach spaces. Image restoration using TV norms for regularization \cite{CDOS, ROF} leads to searching an optimization solution in a Banach space. In the area of inverse problems, semi-discrete inverse problems were considered in Banach spaces \cite{SRS}. Regularized learning in Banach spaces was originated in \cite{MP}. Since then, regularized learning in a Banach space and a desired representer theorem of its solutions have received considerable attention in the literature. The minimum norm interpolation or its related regularization problem in a Banach space is also motivated from a theoretical point of view: the functional extension problem in such a space. Recently, extension of a given function on a finite set in  $\mathbb{R}^n$ to a function on the entired $\mathbb{R}^n$ was studied in a series of papers \cite{Fefferman05, FeffermanIII, FeffermanI, FeffermanII}. Specifically, Fefferman considered in \cite{Fefferman05} the extension of a function defined on a finite subset $E$ of $\mathbb{R}^n$ to a $C^m$ function $F: \mathbb{R}^n\rightarrow\mathbb{R}$ with the $C^m$ norm of the smallest possible order of magnitude. A sufficient condition ensuring the existence of the desired extension was proved in that paper. In \cite{FeffermanIII, FeffermanI, FeffermanII}, Fefferman and Klartag exhibited algorithms for constructing such an extension function $F$ and for computing the order of magnitude of its $C^m$ norm.

The notion of reproducing kernel Banach space (RKBS) was originally introduced in \cite{ZXZ} and further developed in \cite{SZ, SZH, ZZ}. In the framework of a semi-inner-product RKBS, the representer theorem of the solutions of the regularized learning problem was derived from the dual elements and the semi-inner products \cite{ZXZ,ZZ}. In \cite{XY}, an alternative definition of RKBS was provided by the dual bilinear form. In this paper, for a reflexive and smooth RKBS, the representer theorem of the solutions of the regularized learning problem was also obtained using the G\^{a}teaux derivative of the norm function and the reproducing kernel. The above RKBSs, in which the representer theorem was well established for the regularized learning problem, are all reflexive and smooth. In fact, the reflexivity guarantees the existence of solutions of the regularized learning problem and the smoothness allows us to use the G\^{a}teaux derivative of the norm function to describe the representer theorem. In the special case of a semi-inner-product RKBS, the G\^{a}teaux derivative can be represented by the semi-inner-product. In addition, the reproducing kernel provides a closed-form function representation for the point-evaluation functionals. The representer theorem was generalized to a non-reflexive and non-smooth Banach space which has a pre-dual space \cite{HLTY, Unser2016, Unser}. Having the pre-dual space guarantees that the Banach space has the weak$^{*}$ topology, which together with the continuity of the loss function and the regularizer, also leads to the existence of the solutions. Due to lack of the G\^{a}teaux derivative, other tools need to be used to describe the representer theorem. The representer theorem was obtained in \cite{HLTY} by employing the subdifferential of the norm function for a lower semi-continuous loss function and the quadratic regularizer. The representer theorem was established in \cite{Unser} by the duality mapping for a class of inverse problems with a convex and continuous loss function and regularizer. Moreover, representer theorems for deep kernel learning and deep neural networks were obtained in \cite{BRG19} and \cite{Unser2019}, respectively.

It is the main purpose of this paper to understand the solution representation of the minimum norm interpolation and the regularization problem in a Banach space. In the literature there are a few existing representer theorems for regularized learning problems in a Banach space. However, all of them are in an {\it implicit} form, which makes them not direct for solution representations. We will first bridge different approaches used in the literature for establishing representer theorems for a solution of the minimum norm interpolation problem in a Banach space and its related regularization problem, to deepen the understanding of the underlying mathematical ideas. As such, we will provide novel {\it explicit} representer theorems in a general setting, potentially useful for direct solution representations. We aim at expressing the {\it simplicity,  beauty, generality and unity} of the representer theorem and commit to developing solution representations of these problems suitable for further designing their numerical algorithms.

The minimum norm interpolation problem is closely related to the regularized learning problem. Relations between them were investigated in \cite{MP}. We first establish solution representations for the minimum norm interpolation problem. We then convert the resulting representer theorems to the regularization problem through the relation between the two problems. The essence of the representer theorem refers to that the original optimization problem in an infinite dimensional space can be reduced to one possibly in a finite dimensional space. This profits from the fact that the number of data points used to learn a function is finite. A crucial issue about the representer theorem concerns how to characterize the relation between the solutions of the original infinite dimensional optimization problem and the finite dimensional optimization problem. To address it, we characterize the minimum norm interpolation problem through two different approaches. Firstly, the minimum norm interpolation problem is interpreted as a problem of best approximation. Due to the Hahn-Banach theorem, the latter can be characterized by the functionals which peak at the best approximation point. The set of such functionals is defined by the value of the duality mapping at the best approximation point. Accordingly, the duality mapping become a suitable tool for the representer theorem of the solutions of the minimum norm interpolation problem and then the regularization problem. Secondly, as a classical optimization problem with constraints, the minimum norm interpolation problem can be solved by the Lagrange multiplier method. Due to lack of the smoothness of general Banach spaces (not necessarily smooth), the subdifferential of the norm function needs to be used here. In a special case that the Banach space is smooth, the duality mapping and the subdifferential of the norm function are both reduced to the G\^{a}teaux derivative of the norm function. The representer theorem for this case has a simple form which is described by the G\^{a}teaux derivative. In summary, the fact that the number of data used to learn a function is finite leads to the desired representer theorems and the mathematical tools, such as the duality mapping, the subdifferential and the G\^{a}teaux derivative of the norm function help us describe the representer theorems.

It is desirable to develop solution representations of the minimum norm interpolation problem and the regularization problem convenient for algorithmic development. Inspired by the success of the fixed-point approach used in solving several types of finite dimensional problems such as machine learning \cite{AMPSX, LMX, Li-Song-Xu2018, Li-Song-Xu2019, PSW}, image processing \cite{CHZ, LMSX, LSXZ, LSXX, MSX}, medical imaging \cite{KLSX, LZKSVSLFX, ZLKSZX} and solutions of semi-discrete inverse problems \cite{FJL, JL}, we develop solution representations for the minimum norm interpolation problem and the regularization problem by using a fixed-point formulation via the proximity operator of the functions appearing in the objective function or constraints. This formulation has great potential for convenience of designing iterative algorithms for solving these problems. Difficulty of extending the existing work which is either in a finite dimensional space or in a Hilbert space to the current setting lies in the infinite dimensional component of the Banach space. In particular, we reformulate solutions of the minimum norm interpolation problem and the regularization problem in the special Banach space $\ell_1(\mathbb{N})$ as fixed-points of a nonlinear map defined on a finite dimensional space by making use of special properties of the pre-dual space of $\ell_1(\mathbb{N})$, leading to implementable iterative algorithms for solving the problem. Extension of this approach to either the minimum norm interpolation problem or the regularization problem in a general Banach space will be a future research topic. We remark that a solution method for the minimum norm interpolation problem in $\ell_1(\mathbb{N})$ was proposed in \cite{CX} by reformulating it as a linear programming problem. However, solving the resulting linear programming problem requires an exponential computational cost, and thus the method is not feasible for practical computation in the context of big data analytic. The fixed-point equation approach presented in this paper overcomes this difficulty.

In passing, we would like to point it out that although the representer theorem reduces an infinite dimensional problem to a finite dimensional one, in general, often certain infinite dimensional component is hidden in the resulting finite dimensional problem. We will single out these hidden infinite dimensional component and for certain special cases of practical importance, we will show how this obstacle can be removed to obtain a truly finite dimensional one.

This paper is organized in nine sections. In section 2, we describe the minimum norm interpolation problem in a Banach space and present a sufficient condition to ensure the existence of its solutions. We characterize in section 3 a solution of the minimum norm interpolation problem by two different approaches in which either of the duality mapping or the subdifferentail of the norm function are used to describe the representer theorem for the problem. We first establish implicit representer theorems and identify the relation of our results with those existing in the literature. We then enact explicit representer theorems by using duality arguments. We also consider several special cases of practical importance and provide special results for these cases. In section 4, we develop approaches to determine the coefficients involved in the representer theorems when the Banach space has a pre-dual space and the linear functionals appearing in the minimum norm interpolation problem belong to the pre-dial space. These approaches of determining the coefficients lead to solution methods for solving the minimum norm interpolation problem. In section 5, we present the infimum of the minimum norm interpolation in a Banach space. In section 6, we propose fixed-point equations for the minimum norm interpolation problem in a Banach space. In particular, when the Banach space $\mathcal{B}$ is the special space $\ell_1(\mathbb{N})$, we design implementable fixed-point equations for finding a solution of this problem. This fixed-point formulation will serve as a basis for further development of efficient convergence guaranteed algorithms. We describe in section 7 the regularization problem in a Banach space and discuss a sufficient condition that ensures the existence of its solutions. We also elaborate an intrinsic relation between the regularization problem and a related minimum norm interpolation problem. In section 8, we establish both implicit and explicit representer theorems for the regularization problem. We also deliver special results for several cases of practical importance. Moreover, the second portion of
section 8 is devoted to the presentation of solutions of regularization problems in Banach spaces. We present the representer theorems based solution representations and as well as the fixed-point formulation for the regularization problems. Finally, we make conclusion remarks in section 9, discussing the main contribution of this paper.

\section{Minimum Norm Interpolation in a Banach Space}
Minimum norm interpolation aims at finding an element, in a suitable space, having the smallest norm and interpolating a given set of sampled data. In this section, we describe the minimum norm interpolation problem in a Banach space and present a sufficient condition which ensures the existence of its solutions.

We first describe the minimum norm interpolation problem in a Banach space. Let $\mathcal{B}$ denote a real Banach space with norm $\|\cdot\|_{\mathcal{B}}$. By $\mathcal{B}^*$ we denote the dual space of $\mathcal{B}$, the space of all continuous linear functionals on $\mathcal{B}$ with the norm
$$
\|\nu\|_{\mathcal{B}^*}:=\sup_{f\in\mathcal{B},f\neq0}
\frac{|\nu(f)|}{\|f\|_{\mathcal{B}}},
\ \ \mbox{for all}\ \ \nu\in\mathcal{B}^*.
$$
The dual bilinear form $\langle\cdot,\cdot\rangle_{\mathcal{B}}$ on $\mathcal{B}^*\times\mathcal{B}$ is defined as
$$
\langle\nu,f\rangle_{\mathcal{B}}:=\nu(f),
\ \mbox{for all}\  \nu\in\mathcal{B}^*
\ \mbox{and all}\ f\in\mathcal{B}.
$$
For each $m\in\mathbb{N}$, let $\mathbb{N}_m:=\{1,2,\dots,m\}$. Suppose that $\nu_j \in \mathcal{B}^*$, $j\in\mathbb{N}_m,$ are a finite number of linearly independent elements. Associated with these functionals, we introduce an operator $\mathcal{L}:\mathcal{B}\rightarrow\mathbb{R}^m$ by
\begin{equation}\label{functional-operator}
\mathcal{L}(f):=[\langle\nu_j,f\rangle_{\mathcal{B}}:j\in\mathbb{N}_m],
\ \ \mbox {for all}\ \ f\in\mathcal{B}.
\end{equation}
According to the continuity of the linear functionals $\nu_j$, $j\in\mathbb{N}_m,$ on $\mathcal{B}$, we have for each $f\in\mathcal{B}$ that
$$
\|\mathcal{L}(f)\|_{\mathbb{R}^m}
=\left(\sum_{j\in\mathbb{N}_m}|\langle\nu_j,f\rangle_{\mathcal{B}}|^2\right)^{1/2}
\leq\left(\sum_{j\in\mathbb{N}_m}\|\nu_j\|_{\mathcal{B}^*}^2\right)^{1/2}
\|f\|_{\mathcal{B}},
$$
which yields that
$$
\|\mathcal{L}\|\leq\left(\sum_{j\in\mathbb{N}_m}\|\nu_j\|_{\mathcal{B}^*}^2\right)^{1/2}.
$$
For a given vector $\mathbf{y}:=[y_j: j\in\mathbb{N}_m]\in\mathbb{R}^m$, we set
\begin{equation}\label{My}
\mathcal{M}_{\mathbf{y}}:=\{f\in\mathcal{B}:\mathcal{L}(f)=\mathbf{y}\}.
\end{equation}
In particular, when $\mathbf{y}$ is the zero vector, we write $\mathcal{M}_0$.
The minimum norm interpolation problem with given data $\{(\nu_j, y_j):j\in\mathbb{N}_m\}$ has the form
\begin{equation}\label{mni}
\inf\{\|f\|_{\mathcal{B}}:f\in\mathcal{M}_{\mathbf{y}}\}.
\end{equation}

We now consider the existence of a solution of the minimum norm interpolation problem \eqref{mni}. The linear independence of the functionals $\nu_j$, $j\in\mathbb{N}_m,$ ensures that $\mathcal{M}_{\mathbf{y}}$ is nonempty for any given $\mathbf{y}\in\mathbb{R}^m$. By employing standard arguments in convex analysis \cite{ET,Zalinescu}, we establish a sufficient condition that ensures the existence of a solution of the problem. To this end, we recall some notions in Banach spaces. Since the natural map is the isometrically imbedding map from $\mathcal{B}$ into $\mathcal{B}^{**}$, there holds
\begin{equation}\label{natural-map}
\langle\nu,f\rangle_{\mathcal{B}}=\langle f,\nu\rangle_{\mathcal{B}^*},
\ \ \mbox{for all}\ \ f\in\mathcal{B}\ \ \mbox{and all}\ \ \nu\in\mathcal{B}^*.
\end{equation}
The weak$^*$ topology of the dual space $\mathcal{B}^{*}$ is the smallest topology for $\mathcal{B}^{*}$ such that, for each $f\in\mathcal{B}$, the linear functional $\nu\rightarrow \langle\nu,f\rangle_\mathcal{B}$ on $\mathcal{B}^*$ is continuous with respect to the topology. A sequence $\nu_n$, $n\in\mathbb{N},$ in $\mathcal{B}^{*}$ is said to converge weakly$^{*}$ to $\nu\in\mathcal{B}^{*}$ if
$$
\displaystyle{\lim_{n\rightarrow+\infty}\langle\nu_n,f\rangle_{\mathcal{B}}
=\langle\nu,f\rangle_{\mathcal{B}}},\ \ \mbox{for all}\ \ f\in\mathcal{B}.
$$
A normed space $\mathcal{B}_{*}$ is called a pre-dual space of a Banach space $\mathcal{B}$ if $(\mathcal{B}_{*})^{*}=\mathcal{B}.$ It follows from equation \eqref{natural-map} with $\mathcal{B}$ being replaced by $\mathcal{B}_{*}$ that
\begin{equation}\label{natural-map-predual}
\langle\nu,f\rangle_{\mathcal{B}}=\langle f,\nu\rangle_{\mathcal{B}_*},
\ \ \mbox{for all}\ \ f\in\mathcal{B}\ \ \mbox{and all}\ \ \nu\in\mathcal{B}_*.
\end{equation}
The pre-dual space $\mathcal{B}_{*}$ guarantees that the Banach space $\mathcal{B}$ enjoys the weak$^{*}$ topology. The Banach-Alaoglu theorem \cite{Meg} ensures that if a Banach space $\mathcal{B}$ has a pre-dual space $\mathcal{B}_{*}$, then any bounded sequence $f_n$, $n\in\mathbb{N},$ in $\mathcal{B}$ has a weak$^{*}$ accumulation point $f\in\mathcal{B}$. That is, there exists a subsequence $f_{n_k}$, $k\in\mathbb{N}$, such that
$$
\lim_{k\rightarrow+\infty}\langle\nu,f_{n_k}\rangle_{\mathcal{B}}
=\langle\nu,f\rangle_{\mathcal{B}},
\ \ \mbox{for all}\ \ \nu\in\mathcal{B}_*.
$$
A Banach space $\mathcal{B}$ is said to be reflexive if $(\mathcal{B}^{*})^{*}=\mathcal{B}.$
It is clear that a reflexive Banach space $\mathcal{B}$ always takes the dual space $\mathcal{B}^{*}$ as a pre-dual space $\mathcal{B}_{*}$. However, a  Banach space $\mathcal{B}$ having a pre-dual space may not be reflexive. For example, the Banach space $\ell_1(\mathbb{N})$ of all real sequences
$$
\mathbf{x}:=(x_j: j\in \mathbb{N}),\ \ \mbox{with}\ \ \|\mathbf{x}\|_1:=\sum_{j\in\mathbb{N}}|x_j|<+\infty,
$$
has $c_0$ as its pre-dual space, where $c_0$ denotes the space of all real sequences
$\mathbf{u}:=(u_j: j\in \mathbb{N})$ such that $\displaystyle{\lim_{j\rightarrow+\infty}u_j=0}$, with
$$
\|\mathbf{u}\|_\infty := \sup\{|u_j|: j\in\mathbb{N}\}<+\infty.
$$
Clearly, the Banach space $\ell_1(\mathbb{N})$ has a pre-dual space but it is not reflexive.

We now turn to considering the existence of a solution of the minimum norm interpolation problem \eqref{mni} in a Banach space $\mathcal{B}$ having the pre-dual space $\mathcal{B}_*$. In this case, the linear functionals appearing in \eqref{mni} need to be restricted to the pre-dual space $\mathcal{B}_*$.

\begin{prop}\label{existence-mni}
If the Banach space $\mathcal{B}$ has the pre-dual space $\mathcal{B}_{*}$ and $\nu_j\in\mathcal{B}_{*}$, $j\in\mathbb{N}_m,$ are linearly independent, then for any $\mathbf{y}\in\mathbb{R}^m$ the minimum norm interpolation problem \eqref{mni} has at least one solution.
\end{prop}
\begin{proof}
Since for any $\mathbf{y}\in\mathbb{R}^m$, the set $\mathcal{M}_{\mathbf{y}}$ is nonempty, there exists a sequence $f_n$, $n\in\mathbb{N},$ in $\mathcal{M}_{\mathbf{y}}$ satisfying
\begin{equation}\label{minimizing-mni}
\lim_{n\rightarrow+\infty}\|f_n\|_{\mathcal{B}}
=\inf\{\|f\|_{\mathcal{B}}:f\in\mathcal{M}_{\mathbf{y}}\}.
\end{equation}
This ensures that the sequence $f_n$, $n\in\mathbb{N}$, is bounded. By the Banach-Alaoglu theorem, there exists a subsequence $f_{n_k}$, $k\in\mathbb{N}$, which weakly$^{*}$ converges to $\hat{f}\in\mathcal{B}$. It suffices to prove that the weak$^{*}$ accumulation point $\hat{f}$ is a solution of the minimum norm interpolation problem \eqref{mni}.

We first verify that  $\hat{f}$ satisfies the interpolation condition. Since $\nu_j\in\mathcal{B}_{*}$, $j\in\mathbb{N}_m,$ the linear functionals $\nu_j,j\in\mathbb{N}_m,$ are weakly$^{*}$ continuous. This leads to
$$
\langle\nu_j,\hat{f}\rangle_{\mathcal{B}}
=\lim_{k\rightarrow+\infty}\langle\nu_j,f_{n_k}\rangle_{\mathcal{B}},\ \ \mbox{for all}\ \ j\in\mathbb{N}_m.
$$
By the fact that $f_n\in\mathcal{M}_{\mathbf{y}}$, $n\in\mathbb{N}$, we get $\mathcal{L}(\hat{f})=\mathbf{y}$. That is, the interpolation condition holds. Note that the norm $\|\cdot\|_{\mathcal{B}}$ is weakly$^{*}$ lower semi-continuous on $\mathcal{B}$. According to the weak$^{*}$ convergence of $f_{n_k}$, $k\in\mathbb{N},$ there holds
$$
\|\hat{f}\|_{\mathcal{B}}\leq\liminf_{k\rightarrow+\infty}\|f_{n_k}\|_{\mathcal{B}},
$$
which together with \eqref{minimizing-mni} yields that $\hat{f}$ is a solution of \eqref{mni}.
\end{proof}

In the special case that $\mathcal{B}=\ell_1(\mathbb{N})$, the existence of a solution of the minimum norm interpolation problem \eqref{mni} was given in \cite{CX} by an elementary argument. As a consequence of the above proposition, we also have the existence of a solution of \eqref{mni} when $\mathcal{B}$ is reflexive.

\begin{cor}\label{existence-mni-reflexive}
If the Banach space $\mathcal{B}$ is reflexive and $\nu_j\in\mathcal{B}^{*}$, $j\in\mathbb{N}_m,$ are linearly independent, then for any $\mathbf{y}\in\mathbb{R}^m$ the minimum norm interpolation problem \eqref{mni} has at least one solution.
\end{cor}
\begin{proof}
If $\mathcal{B}$ is reflexive then the dual space $\mathcal{B}^{*}$ is its pre-dual space. In this special case, the existence of a solution of the problem \eqref{mni} follows directly from Proposition \ref{existence-mni}.
\end{proof}

We note that solutions of the problem \eqref{mni} may not be unique unless the Banach space $\mathcal{B}$ is strictly convex.

\section{Representer Theorems for Minimum Norm Interpolation}

In this section, we establish several representer theorems for a solution of the minimum norm interpolation problem (\ref{mni}). The resulting representer theorems show that even though the minimum norm interpolation problem with a finite number of data points is a minimization problem in an {\it infinite dimensional} space, it can be transferred to one possibly in a {\it finite dimensional} space. We then identify the relation of the representer theorems obtained here with those existing in the literature.

The representer theorems for the minimum norm interpolation problem established in the literature are often stated with restricted assumptions on the Banach space \cite{XY, ZXZ}. We realize that most of the assumptions are used to ensure the existence and uniqueness of the solution of the minimum norm interpolation problem. For example, as we have established in the last section, if $\mathcal{B}$ is a Banach space having a pre-dual space or being reflexive, then the minimum norm interpolation problem has at least a solution. If $\mathcal{B}$ is strictly convex, then there exists at most a solution of the minimum norm interpolation problem. The smoothness of the Banach space allows us to describe the representer theorem by using the G\^{a}teaux derivative of the norm function. We shall clarify in this section that the validity of the representer theorem does not depend on the these properties of the Banach space.

\subsection{Implicit Representer Theorems}

We first present implicit representer theorems for a solution of the minimum norm interpolation problem (\ref{mni}).

We treat the minimum norm interpolation problem via two different approaches: a functional analytic approach and a convex analytic approach. In the functional analytic approach, we convert the minimum norm interpolation problem \eqref{mni} as a best approximation problem and then use the duality theory to characterize the best approximation from a linear translate of a subspace. The duality theory was used extensively in the literature \cite{Bra, CDW, FD, DUX, DUWX, MSSW, SWX90, SWX90-2, SWX91, UX, Xu} to characterize a best approximation from a convex set or a subspace in Banach spaces.
For a nonempty subset $\mathcal{M}$ of $\mathcal{B}$, we define the distance from $f\in\mathcal{B}$ to $\mathcal{M}$ by
$$
\mathrm{d}(f,\mathcal{M}):=\inf\{\|f-h\|_{\mathcal{B}}:h\in\mathcal{M}\}.
$$
An element $f_0\in\mathcal{M}$ is said to be a best approximation to $f$ from $\mathcal{M}$ if
$$
\|f-f_0\|_{\mathcal{B}}=\mathrm{d}(f,\mathcal{M}).
$$
A subset $\mathcal{M}$ of $\mathcal{B}$ is called a convex set if
$$
tf+(1-t)g\in\mathcal{M}, \ \ \mbox{for all}\ \ f,g\in\mathcal{M}
\ \ \mbox{and all}\ \ t\in[0,1].
$$
It is easy to see that for any $\mathbf{y}\in\mathbb{R}^m\backslash\{0\}$, $\mathcal{M}_{\mathbf{y}}$ is a closed convex subset and $\mathcal{M}_0$ is a closed subspace of $\mathcal{B}$. In fact, $\mathcal{M}_{\mathbf{y}}$ is a translate of $\mathcal{M}_0$. Moreover,  there holds that
\begin{equation}\label{relation-My¡ªM0}
\mathcal{M}_0+f=\mathcal{M}_{\mathbf{y}},
\ \ \mbox{for each}\ \ f\in\mathcal{M}_{\mathbf{y}}.
\end{equation}
An obvious relation between subsets $\mathcal{M}_{\mathbf{y}}$ and $\mathcal{M}_0$ is given in the following lemma.

\begin{lemma}\label{relation}
If $\mathbf{y}\in\mathbb{R}^m\backslash\{0\}$, then for any $f\in\mathcal{M}_{\mathbf{y}}$ and any $g\in\mathcal{M}_0$, there holds
$$
\mathrm{d}(f,\mathcal{M}_0)=\mathrm{d}(g,\mathcal{M}_{\mathbf{y}}).
$$
\end{lemma}
\begin{proof}
Suppose that $f\in\mathcal{M}_{\mathbf{y}}$ and $g\in\mathcal{M}_0$. It follows from  \eqref{relation-My¡ªM0} that there holds
\begin{eqnarray*}
\mathrm{d}(g,\mathcal{M}_{\mathbf{y}})&=&\inf\{\|g-h\|:h\in\mathcal{M}_{\mathbf{y}}\}\\
&=&\inf\{\|g-(f+\tilde{h})\|:\tilde{h}\in\mathcal{M}_0\}\\
&=&\inf\{\|f-(g-\tilde{h})\|:\tilde{h}\in\mathcal{M}_0\}.
\end{eqnarray*}
Since $\mathcal{M}_0$ is a closed subspace of $\mathcal{B}$, we get that
$$
\mathrm{d}(g,\mathcal{M}_{\mathbf{y}})
=\inf\{\|f-\tilde{h}\|:\tilde{h}\in\mathcal{M}_0\}
=\mathrm{d}(f,\mathcal{M}_0),
$$
which completes the proof.
\end{proof}

Lemma \ref{relation} allows us to develop a characterization of solutions of the minimum norm interpolation problem \eqref{mni} in terms of best approximation from a subspace.

\begin{prop}\label{chatacterization-best-app}
Suppose that $\mathcal{B}$ is a Banach space with the dual space $\mathcal{B}^*$ and $\nu_j\in\mathcal{B}^*$, $j\in\mathbb{N}_m$, are linearly independent. If $\mathbf{y}\in\mathbb{R}^m\backslash\{0\}$, and $\mathcal{M}_{\mathbf{y}}$ and $\mathcal{M}_0$ are defined by \eqref{My}, then $\hat{f}\in\mathcal{B}$ is a solution of the minimum norm interpolation problem \eqref{mni} with $\mathbf{y}$ if and only if $\hat{f}\in\mathcal{M}_{\mathbf{y}}$ and $0$ is a best approximation to $\hat{f}$ from $\mathcal{M}_0$.
\end{prop}
\begin{proof}
An element $\hat{f}\in\mathcal{B}$ is a solution of \eqref{mni} with $\mathbf{y}$ if and only if $\hat{f}\in\mathcal{M}_{\mathbf{y}}$ and
$$
\|\hat{f}\|_{\mathcal{B}}=\inf\{\|f\|_{\mathcal{B}}:f\in\mathcal{M}_{\mathbf{y}}\}
=\mathrm{d}(0,\mathcal{M}_{\mathbf{y}}).
$$
By Lemma \ref{relation}, we get that $\hat{f}$ is a solution of \eqref{mni} if and only if
$$
\hat{f}\in\mathcal{M}_{\mathbf{y}}\ \ \mbox{and}\ \ \|\hat{f}\|_{\mathcal{B}}=\mathrm{d}(\hat{f},\mathcal{M}_0).
$$
The latter is equivalent to the fact that $\hat{f}\in\mathcal{M}_{\mathbf{y}}$ and $0$ is a best approximation to $\hat{f}$ from $\mathcal{M}_0$.
\end{proof}

Proposition \ref{chatacterization-best-app} makes it possible for us to characterize a solution of the minimum norm interpolation problem \eqref{mni} by using the duality approach \cite{Bra}. In other words, Proposition \ref{chatacterization-best-app} enables us to characterize a solution of the problem \eqref{mni} via identifying a best approximation from $\mathcal{M}_0$ with continuous linear functionals by the well-known Hahn-Banach theorem.

We next review a corollary of the Hahn-Banach theorem in the context of best approximation.
To this end, we recall the definition of annihilators of subsets in Banach spaces. Let $\mathcal{M}$ and $\mathcal{M}'$ be subsets of $\mathcal{B}$ and $\mathcal{B}^*$, respectively. According to \cite{Meg}, the annihilator, in $\mathcal{B}^*$, of $\mathcal{M}$ is defined by
$$
\mathcal{M}^{\perp}:=\{\nu\in\mathcal{B}^*: \langle\nu,f\rangle_{\mathcal{B}}=0,
\ \mbox{for all}\  f\in\mathcal{M}\}.
$$
The annihilator, in $\mathcal{B}$, of $\mathcal{M}'$ is defined by
$$
^{\perp}\mathcal{M}':=\{f\in\mathcal{B}: \langle\nu,f\rangle_{\mathcal{B}}=0,
\ \mbox{for all}\  \nu\in\mathcal{M}'\}.
$$
The following result, whose proof may be found in \cite{Bra}, characterizes the best approximation with continuous linear functionals. For special results regarding the $L_1$ approximation, the readers are referred to \cite{Pinkus}.

\begin{lemma}\label{bestapp}
Let $\mathcal{M}$ be a closed subspace of a Banach space $\mathcal{B}$ and $f\notin\mathcal{M}$. Then $f_0\in\mathcal{M}$ is a best approximation to $f$ from $\mathcal{M}$ if and only if there is a continuous linear functional $\nu\in\mathcal{M}^{\perp}$ such that
$$
\|\nu\|_{\mathcal{B}^*}=1\ \  \mbox{and}\ \  \langle\nu,f-f_0\rangle_{\mathcal{B}}=\|f-f_0\|_{\mathcal{B}}.
$$
\end{lemma}

Combining Proposition \ref{chatacterization-best-app} with Lemma \ref{bestapp}, we characterize below a solution of the minimum norm interpolation problem \eqref{mni} in terms of continuous linear functionals.

\begin{prop}\label{chatacterization-functional}
Suppose that $\mathcal{B}$ is a Banach space with the dual space $\mathcal{B}^*$ and $\nu_j\in\mathcal{B}^*$, $j\in\mathbb{N}_m$, are linearly independent. Let $\mathbf{y}\in\mathbb{R}^m\backslash\{0\}$, and $\mathcal{M}_{\mathbf{y}}$ and $\mathcal{M}_0$ be defined by \eqref{My}. Then $\hat{f}\in\mathcal{B}$ is a solution of the minimum norm interpolation problem \eqref{mni} with $\mathbf{y}$ if and only if $\hat{f}\in\mathcal{M}_{\mathbf{y}}$ and there is a continuous linear functional $\nu\in\mathcal{M}_{0}^{\perp}$ such that
\begin{equation}\label{peak-functional}
\|\nu\|_{\mathcal{B}^*}=1\ \ \mbox{and}\ \ \langle\nu,\hat{f}\rangle_{\mathcal{B}}=\|\hat{f}\|_{\mathcal{B}}.
\end{equation}
\end{prop}
\begin{proof}
Proposition \ref{chatacterization-best-app} ensures that $\hat{f}$ is a solution of the minimum norm interpolation problem \eqref{mni} with $\mathbf{y}$ if and only if $\hat{f}\in\mathcal{M}_{\mathbf{y}}$ and $0$ is a best approximation to $\hat{f}$ from $\mathcal{M}_0$. By applying Lemma \ref{bestapp} with $\mathcal{M}:=\mathcal{M}_{0}$, $f_0:=0$ and $f:=\hat f$, we conclude that the latter is equivalent to that $\hat{f}\in\mathcal{M}_{\mathbf{y}}$ and there is a linear functional $\nu\in\mathcal{M}_{0}^{\perp}$ satisfying equations \eqref{peak-functional}.
\end{proof}

We next identify the subspace $\mathcal{M}_0^{\perp}$ of $\mathcal{B}^*$ with the linear span of the finite number of linear continuous functionals
\begin{equation}\label{Vm}
\mathcal{V}_m:=\mathop{\{\nu_j: j\in\mathbb{N}_m\}}.
\end{equation}
For this purpose, we recall Proposition 2.6.6 of \cite{Meg} which states that for a subset $\mathcal{M}'$ in $\mathcal{B}^*$, the set $(^{\perp}\mathcal{M}')^{\perp}$ coincides with the closed linear span of $\mathcal{M}'$ in the weak$^*$ topology of $\mathcal{B}^*$.
For a subset $\mathcal{M}'$ of $\mathcal{B}^*$, we denote by $\overline{\mathcal{M}'}^{w*}$ the closure of $\mathcal{M}'$ in the weak$^*$ topology of $\mathcal{B}^*$. The following lemma provides a specific representation of the annihilator $\mathcal{M}_{0}^{\perp}$ of $\mathcal{M}_{0}$.

\begin{lemma}\label{finite}
Suppose that $\mathcal{B}$ is a Banach space with the dual space $\mathcal{B}^*$ and $\nu_j\in\mathcal{B}^*$, $j\in\mathbb{N}_m$, are linearly independent. If $\mathcal{M}_0$ is defined by \eqref{My} with $\mathbf{y}=0$ and $\mathcal{V}_m$ defined by \eqref{Vm}, then
\begin{equation}\label{M0perp}
\mathcal{M}_0^{\perp}=\span\mathcal{V}_m.
\end{equation}
\end{lemma}
\begin{proof}
We prove this lemma by appealing to Proposition 2.6.6 of \cite{Meg}. The definition of the annihilator leads to
$$
\mathcal{M}_{0}={}^{\perp}\mathcal{V}_m.
$$
Applying Proposition 2.6.6 of \cite{Meg} with $\mathcal{M}':=\mathcal{V}_m$ yields that
\begin{equation}\label{M0perp1}
\mathcal{M}_{0}^{\perp}=(^{\perp}\mathcal{V}_m)^{\perp}
=\overline{\span\mathcal{V}_m}^{w*}.
\end{equation}
Since the linear span of $\mathcal{V}_m$ is a finite dimensional subspace of $\mathcal{B}^*$, there holds
$$
\overline{\span\mathcal{V}_m}^{w*}
=\span\mathcal{V}_m.
$$
Substituting the equation above into the right hand side of equation \eqref{M0perp1} leads to the desired result \eqref{M0perp} of this lemma.
\end{proof}

The above lemma shows that subspace $\mathcal{M}_0^{\perp}$ is of finite dimension. This is a consequence of the fact that the number of continuous linear functionals appearing in $\mathcal{M}_0$ is finite. Lemma \ref{finite} together with Proposition \ref{chatacterization-functional} leads to a solution representation of the minimum norm interpolation problem \eqref{mni}.

\begin{prop}\label{representer-theorem-functional}
Suppose that $\mathcal{B}$ is a Banach space with the dual space $\mathcal{B}^*$ and $\nu_j\in\mathcal{B}^*$, $j\in\mathbb{N}_m$, are linearly independent. Let $\mathbf{y}\in\mathbb{R}^m$ and $\mathcal{M}_{\mathbf{y}}$ be defined by \eqref{My}. Then $\hat{f}\in\mathcal{B}$ is a solution of the minimum norm interpolation problem \eqref{mni} with $\mathbf{y}$ if and only if $\hat{f}\in\mathcal{M}_{\mathbf{y}}$ and there exist $c_j\in\mathbb{R}$, $j\in\mathbb{N}_m,$ such that the linear functional $\nu:=\sum_{j\in\mathbb{N}_m}c_j\nu_j$ satisfying equations \eqref{peak-functional}.
\end{prop}

\begin{proof}
We first consider the case that $\mathbf{y}:=[y_j:j\in\mathbb{N}_m]=0.$ Note that the minimum norm interpolation problem \eqref{mni} with $\mathbf{y}=0$ has a unique solution $\hat{f}=0.$ On one hand, it is clear that the trivial solution $\hat{f}=0$ belongs to $\mathcal{M}_{0}$. Moreover, equations \eqref{peak-functional} also hold by choosing $c_j\in\mathbb{R},$ $j\in\mathbb{N}_m$ such that the norm of the functional $\nu:=\sum_{j\in\mathbb{N}_m}c_j\nu_j$ equals to $1$. On the other hand, if $\hat{f}\in\mathcal{M}_{\mathbf{y}}$ and there exist $c_j\in\mathbb{R}$, $j\in\mathbb{N}_m,$ such that the functional $\nu:=\sum_{j\in\mathbb{N}_m}c_j\nu_j$ satisfying $\langle\nu,\hat{f}\rangle_{\mathcal{B}}=\|\hat{f}\|_{\mathcal{B}}$, then
$$
\|\hat{f}\|_{\mathcal{B}}
=\sum_{j\in\mathbb{N}_m}c_j\langle\nu_j,\hat{f}\rangle_{\mathcal{B}}
=\sum_{j\in\mathbb{N}_m}c_jy_j=0,
$$
which further implies $\hat{f}=0$. That is, we get the desired conclusion for $\mathbf{y}=0$ .

If $\mathbf{y}\neq0$, Proposition \ref{chatacterization-functional} ensures that $\hat{f}\in\mathcal{B}$ is a solution of the minimum norm interpolation problem \eqref{mni} with $\mathbf{y}$ if and only if $\hat{f}\in\mathcal{M}_{\mathbf{y}}$ and there is a continuous linear functional $\nu\in\mathcal{M}_{0}^{\perp}$ satisfying equations \eqref{peak-functional}.
By Lemma \ref{finite}, there exist $c_j\in\mathbb{R}$, $j\in\mathbb{N}_m,$ such that this continuous linear functional $\nu$ has the form $\nu=\sum_{j\in\mathbb{N}_m}c_j\nu_j$. This establishes the desired result of this proposition.
\end{proof}

We now turn to establishing the representer theorem for a solution of the minimum norm interpolation problem \eqref{mni} by using a convex analytic approach. Specifically, as a special convex programming problem with constraints, the minimum norm interpolation problem can be solved by the Lagrange multiplier method. We recall the notion of the subdifferential of a convex function on a Banach space. A convex function $\phi:\mathcal{B}\rightarrow\mathbb{R}\cup\{+\infty\}$ is said to be subdifferentiable at $f\in\mathcal{B}$ if there exists a functional $\nu\in\mathcal{B}^*$ such that
\begin{equation}\label{subdifferentiable}
\phi(g)-\phi(f)\geq\langle\nu,g-f\rangle_{\mathcal{B}},
 \ \ \mbox{for all}\ \ g\in\mathcal{B}.
\end{equation}
By $\partial\phi(f)$ we denote the set of all the functionals in $\mathcal{B}^*$ satisfying inequality \eqref{subdifferentiable} and it is called the subdifferential of $\phi$ at $f$. In other words,
\begin{equation}\label{subdifferential}
\partial\phi(f):=\{\nu\in\mathcal{B}^*:
\phi(g)-\phi(f)\geq\langle\nu,g-f\rangle_{\mathcal{B}},
 \ \ \mbox{for all}\ \ g\in\mathcal{B}\}.
\end{equation}

In convex programming, the Lagrange multiplier method provides a simple necessary and sufficient condition for solutions of optimization problems with constraints \cite{Zalinescu}.

\begin{lemma}\label{lagrange-multiplier}
If $\phi$ and $\psi_j$, $j\in\mathbb{N}_m,$ are all convex functions from $\mathcal{B}$ to $\mathbb{R}$, then $\hat f\in\mathcal{B}$ is a solution of the optimization problem
$$
\inf\{\phi(f): f\in\mathcal{B}, \psi_j(f)=0, j\in\mathbb{N}_m\}
$$
if and only if $\psi_j(\hat f)=0$, $j\in\mathbb{N}_m,$ and there exist $\lambda_j\in\mathbb{R}$, $j\in\mathbb{N}_m$, such that
\begin{equation}\label{lagrange}
0\in\partial\left(\phi+\sum_{j\in\mathbb{N}_m}\lambda_j\psi_j\right)(\hat f).
\end{equation}
Moreover, if $\psi_j$, $j\in\mathbb{N}_m,$ are continuous at $\hat f$, then \eqref{lagrange} is equivalent to
$$
0\in\partial\phi(\hat f)+\sum_{j\in\mathbb{N}_m}\lambda_j\partial\psi_j(\hat f).
$$
\end{lemma}

By Lemma \ref{lagrange-multiplier} we get an alternative form of the representer theorem for a solution of the minimum norm interpolation problem \eqref{mni}.

\begin{thm}\label{representer-theorem-subdifferential}
Suppose that $\mathcal{B}$ is a Banach space with the dual space $\mathcal{B}^*$ and $\nu_j\in\mathcal{B}^*$, $j\in\mathbb{N}_m$, are linearly independent.
Let $\mathbf{y}\in\mathbb{R}^m$ and $\mathcal{M}_{\mathbf{y}}$ be defined by \eqref{My}. Then $\hat{f}\in\mathcal{B}$ is a solution of the minimum norm interpolation problem \eqref{mni} with $\mathbf{y}$ if and only if $\hat{f}\in\mathcal{M}_{\mathbf{y}}$ and there exist $c_j\in\mathbb{R},\ j\in\mathbb{N}_m,$ such that
\begin{equation*}\label{rt-subdifferential}
\sum_{j\in\mathbb{N}_m}c_j\nu_j\in\partial\|\cdot\|_{\mathcal{B}}(\hat{f}).
\end{equation*}
\end{thm}
\begin{proof}
According to Lemma \ref{lagrange-multiplier} with $\phi:=\|\cdot\|_{\mathcal{B}}$ and
$$
\psi_j:=\langle\nu_j,\cdot\rangle_{\mathcal{B}}-y_j,\ \ j\in\mathbb{N}_m,
$$
we have that $\hat{f}$ is a solution of \eqref{mni} if and only if
$$
\langle\nu_j,\hat{f}\rangle_{\mathcal{B}}=y_j, \ \ j\in\mathbb{N}_m,
$$
and there exist $\lambda_j\in\mathbb{R}$, $j\in\mathbb{N}_m,$ such that
\begin{equation}\label{subdiff1}
0\in\partial\left(\|\cdot\|_{\mathcal{B}}
+\sum_{j\in\mathbb{N}_m}\lambda_j\left(\langle\nu_j,
\cdot\rangle_{\mathcal{B}}-y_j\right)\right)(\hat{f}).
\end{equation}
Since $\nu_j$, $j\in\mathbb{N}_m,$ are continuous on $\mathcal{B}$, Lemma \ref{lagrange-multiplier} ensures that equation \eqref{subdiff1} is equivalent to
\begin{equation}\label{subdiff2}
0\in\partial\|\cdot\|_{\mathcal{B}}(\hat{f})
+\sum_{j\in\mathbb{N}_m}\lambda_j\partial(\langle\nu_j,
\cdot\rangle_{\mathcal{B}}-y_j)(\hat{f}).
\end{equation}
It follows from the linearity of $\langle\nu_j,\cdot\rangle_{\mathcal{B}}$, $j\in\mathbb{N}_m,$ that
\begin{equation}\label{subdiff3}
\partial\langle\nu_j,\cdot\rangle_{\mathcal{B}}(\hat{f})=\nu_j, \ \ j\in\mathbb{N}_m.
\end{equation}
Substituting \eqref{subdiff3} into \eqref{subdiff2} leads to
$$
-\sum_{j\in\mathbb{N}_m}\lambda_j\nu_j\in\partial\|\cdot\|_{\mathcal{B}}(\hat{f}).
$$
Choosing $c_j:=-\lambda_j$, for $j\in\mathbb{N}_m,$  completes the proof of this theorem.
\end{proof}

Both Proposition \ref{representer-theorem-functional} and Theorem \ref{representer-theorem-subdifferential} reveal that the minimum norm interpolation problem with a finite number of data points, as a minimization problem in an infinite dimensional space, can be transferred to one in a finite dimensional space about finitely many coefficients $c_j$, $j\in\mathbb{N}_m$. Although the two representer theorems presented in Proposition \ref{representer-theorem-functional} and Theorem \ref{representer-theorem-subdifferential} are derived from two different viewpoints and described by different mathematical tools, they are intimately connected to each other. The bridge of these two viewpoints is the known result \cite{Cio} that the subdifferential of the norm $\|\cdot\|_{\mathcal{B}}$ at $\hat{f}$ coincides with the set of all the functionals satisfying equations \eqref{peak-functional}. That is, for each $f\in\mathcal{B}\backslash\{0\}$, there holds
\begin{equation}\label{relation-subdifferential-functionals}
\partial\|\cdot\|_{\mathcal{B}}(f)=\{\nu\in\mathcal{B}^*: \|\nu\|_{\mathcal{B}^*}=1, \langle\nu,f\rangle_{\mathcal{B}}=\|f\|_{\mathcal{B}}\}.
\end{equation}

Another well-known notion related to the functionals satisfying equations \eqref{peak-functional} is the peak functional. We recall it below. Let $\mathcal{B}$ be a real Banach space. For each $\nu\in\mathcal{B}^{*}$, there holds
\begin{equation}\label{functional-continuity}
\langle\nu,f\rangle_{\mathcal{B}}
\leq\|\nu\|_{\mathcal{B}^*}\|f\|_{\mathcal{B}},
\ \ \mbox{for all}\ \ f\in\mathcal{B}.
\end{equation}
For a fixed element $f\in\mathcal{B}$, we are particularly interested in identifying functionals $\nu\in\mathcal{B}^{*}\backslash\{0\}$ that
allow $\langle\nu,f\rangle_{\mathcal{B}}$ assuming the upper bound in the inequality \eqref{functional-continuity}.
We say that a functional $\nu\in\mathcal{B}^{*}$ peaks at an element $f\in\mathcal{B}$ if
$$
\langle\nu,f\rangle_{\mathcal{B}}=\|\nu\|_{\mathcal{B}^*}\|f\|_{\mathcal{B}}.
$$
In other words, peak functionals are those whose applications to $f$ are equal to the products of their norms with the norm of $f$. Their extremal property matters.
This gives rise to the notion of the duality mapping on a Banach space \cite{Cio}. Specifically, the duality mapping $\mathcal{J}$ from $\mathcal{B}$ to the collection of all subsets in $\mathcal{B}^*$ is defined for all $f\in\mathcal{B}$ by
$$
\mathcal{J}(f):=\{\nu\in\mathcal{B}^*: \|\nu\|_{\mathcal{B}^*}=\|f\|_{\mathcal{B}}, \langle\nu,f\rangle_{\mathcal{B}}=\|\nu\|_{\mathcal{B}^*}
\|f\|_{\mathcal{B}}\}.
$$
A solution of the minimum norm interpolation problem \eqref{mni} may be described by these functionals.

In the following proposition, we summarize various solution representations of the minimum norm interpolation problem \eqref{mni}, developed by using various notions and tools.

\begin{prop}\label{representer-theorem-summarize}
Suppose that $\mathcal{B}$ is a Banach space with the dual space $\mathcal{B}^*$ and $\nu_j\in\mathcal{B}^*$, $j\in\mathbb{N}_m$, are linearly independent. Let $\mathbf{y}\in\mathbb{R}^m$, $\mathcal{M}_{\mathbf{y}}$ be defined by \eqref{My} and $\hat{f}\in\mathcal{B}$. Then the following statements are equivalent:

(i) $\hat{f}$ is a solution of the minimum norm interpolation problem \eqref{mni} with $\mathbf{y}$.

(ii) $\hat{f}\in\mathcal{M}_{\mathbf{y}}$ and there exist $c_j\in\mathbb{R}$, $j\in\mathbb{N}_m,$ such that the linear functional $\nu:=\sum_{j\in\mathbb{N}_m}c_j\nu_j$ satisfies
\begin{equation}\label{rt-functional}
\|\nu\|_{\mathcal{B}^*}=1\ \ \mbox{and}
\ \ \langle\nu,\hat{f}\rangle_{\mathcal{B}}
=\|\hat{f}\|_{\mathcal{B}}.
\end{equation}

(iii) $\hat{f}\in\mathcal{M}_{\mathbf{y}}$ and there exist $c_j\in\mathbb{R}$, $j\in\mathbb{N}_m,$ such that the linear functional $\nu:=\sum_{j\in\mathbb{N}_m}c_j\nu_j$ peaks at $\hat{f}$, that is,
\begin{equation}\label{rt-peak-functional}
\langle\nu,\hat{f}\rangle_{\mathcal{B}}
=\left\|\nu\right\|_{\mathcal{B}^*}\|\hat{f}\|_{\mathcal{B}}.
\end{equation}

(iv) $\hat{f}\in\mathcal{M}_{\mathbf{y}}$ and there exist $c_j\in\mathbb{R}$, $j\in\mathbb{N}_m,$ such that
\begin{equation}\label{rt-duality}
\sum_{j\in\mathbb{N}_m}c_j\nu_j\in\mathcal{J}(\hat{f}).
\end{equation}

(v) $\hat{f}\in\mathcal{M}_{\mathbf{y}}$ and there exist $c_j\in\mathbb{R}$, $j\in\mathbb{N}_m,$ such that
\begin{equation*}\label{rt-subdifferential1}
\sum_{j\in\mathbb{N}_m}c_j\nu_j\in\partial\|\cdot\|_{\mathcal{B}}(\hat{f}).
\end{equation*}
\end{prop}

\begin{proof}
The equivalence of statements (i) and (ii) and that of (i) and (v) have been proved in Proposition \ref{representer-theorem-functional} and Theorem  \ref{representer-theorem-subdifferential}, respectively. It remains to verify the equivalence of statements (ii), (iii) and (iv).

We first show the equivalence of statements (ii) and (iii). On one hand, if statement (ii) holds, there exist $c_j\in\mathbb{R}$, $j\in\mathbb{N}_m,$ such that the linear functional $\nu:=\sum_{j\in\mathbb{N}_m}c_j\nu_j$ satisfies \eqref{rt-functional}. It is clear that this functional $\nu$ peaks at $\hat{f}$. That is, statement (iii) holds. On the other hand,
if statement (iii) holds, there exist $c_j\in\mathbb{R}$, $j\in\mathbb{N}_m,$ such that the nonzero linear functional $\nu:=\sum_{j\in\mathbb{N}_m}c_j\nu_j$ satisfies
\eqref{rt-peak-functional}. By setting
$$
\tilde{\mathbf{c}}:=\frac{\mathbf{c}}{\|\nu\|_{\mathcal{B}^{*}}}
\ \ \mbox{and}\ \
\tilde{\nu}:=\sum_{j\in\mathbb{N}_m}\tilde{c}_j\nu_j,
$$
we get that $\tilde{\nu}$ satisfies \eqref{rt-functional} and thus statement (ii) holds.

We next prove the equivalence of statements (ii) and (iv). Note that if $\hat{f}=0$, statements (ii) and (iv) both hold without any assumptions. Specifically, equations \eqref{rt-functional} hold by choosing $c_j\in\mathbb{R},$ $j\in\mathbb{N}_m,$ such that the norm of $\nu:=\sum_{j\in\mathbb{N}_m}c_j\nu_j$ equals to $1$. On the other hand, equation \eqref{rt-duality} can be obtained by choosing $c_j=0$, $j\in\mathbb{N}_m$. Hence, it remains to prove the equivalence for the case that $\hat{f}\neq0$. For each $\mathbf{c}:=[c_j:j\in\mathbb{N}_m]\in\mathbb{R}^m$, we scale it by setting $\tilde{\mathbf{c}}:=\|\hat{f}\|_{\mathcal{B}}\mathbf{c}$. Set
$$
\nu:=\sum_{j\in\mathbb{N}_m}c_j\nu_j\ \ \mbox{and}\ \ \tilde{\nu}:=\sum_{j\in\mathbb{N}_m}\tilde{c}_j\nu_j.
$$
Clearly, $\nu$ satisfies \eqref{rt-functional} if and only if $\tilde{\nu}$ satisfies that
$$
\|\tilde{\nu}\|_{\mathcal{B}^*}=\|\hat{f}\|_{\mathcal{B}}
\ \ \mbox{and}\ \
\langle\tilde{\nu},\hat{f}\rangle_{\mathcal{B}}
=\|\tilde{\nu}\|_{\mathcal{B}^*}\|\hat{f}\|_{\mathcal{B}},
$$
which is equivalent to that $\tilde{\nu}$ satisfies \eqref{rt-duality}.
\end{proof}

Proposition \ref{representer-theorem-summarize} gives four characterizations of a solution of the minimum norm interpolation \eqref{mni}, which will serve as a basis for further developing the representer theorems.
We remark that the equivalence of statements (i) and (iii) has been established in \cite{MP}.

We next consider the special case when $\mathcal{B}$ is a smooth Banach space. In this case, the representer theorem can enjoy a nice simple form. To this end, we recall the notion of the G\^{a}teaux derivative of the norm function. The norm $\|\cdot\|_{\mathcal{B}}$ is said to be G\^{a}teaux differentiable at $f\in\mathcal{B}\setminus\{0\}$ if for all $h\in\mathcal{B}$, the limits
\begin{equation}\label{GateauxDiff-limit}
\lim_{t\rightarrow 0}\frac{\|f+th\|_{\mathcal{B}}-\|f\|_{\mathcal{B}}}{t}
\end{equation}
exist. If the norm $\|\cdot\|_{\mathcal{B}}$ is G\^{a}teaux differentiable at $f\in\mathcal{B}\setminus\{0\}$, then there exists a continuous linear functional, denoted by $\mathcal{G}(f)$, in $\mathcal{B}^*$ such that
\begin{equation}\label{GateauxDiff}
\langle\mathcal{G}(f),h\rangle_{\mathcal{B}}=\lim_{t\rightarrow 0}\frac{\|f+th\|_{\mathcal{B}}-\|f\|_{\mathcal{B}}}{t},
\ \ \mbox{for all}\ \ h\in\mathcal{B}.
\end{equation}
We call $\mathcal{G}(f)$ the G\^{a}teaux derivative of the norm $\|\cdot\|_{\mathcal{B}}$ at $f\in\mathcal{B}$. It follows from \eqref{GateauxDiff} that
\begin{equation}\label{relation-norm-GateauxDiff1}
|\langle\mathcal{G}(f),h\rangle_{\mathcal{B}}|
\leq\|h\|_{\mathcal{B}}
\end{equation}
and
\begin{equation}\label{relation-norm-GateauxDiff}
\langle\mathcal{G}(f),f\rangle_{\mathcal{B}}
=\|f\|_{\mathcal{B}}.
\end{equation}
According to inequality \eqref{relation-norm-GateauxDiff1} and equation \eqref{relation-norm-GateauxDiff}, we have that
\begin{equation}\label{norm-GateauxDiff}
\|\mathcal{G}(f)\|_{\mathcal{B}^*}=1.
\end{equation}
Since for $f=0$, the limit defined as in \eqref{GateauxDiff-limit} does not exist, the G\^{a}teaux derivative of the norm $\|\cdot\|_{\mathcal{B}}$ can not be defined at $f=0$. To simplify the presentation, we define the G\^{a}teaux derivative of the norm $\|\cdot\|_{\mathcal{B}}$ at $f=0$ by $\mathcal{G}(f):=0$. A Banach space $\mathcal{B}$ is said to be smooth if the norm $\|\cdot\|_{\mathcal{B}}$ is G\^{a}teaux differentiable at every $f\in\mathcal{B}\backslash\{0\}$. We have the following special result when $\mathcal{B}$ is a smooth Banach space.

\begin{thm}\label{representer-theorembyGateaux}
Suppose that $\mathcal{B}$ is a smooth Banach space with the dual space $\mathcal{B}^*$ and $\nu_j\in\mathcal{B}^*$, $j\in\mathbb{N}_m$, are linearly independent. Let $\mathbf{y}\in\mathbb{R}^m$ and $\mathcal{M}_{\mathbf{y}}$ be defined by \eqref{My}. Then $\hat{f}\in\mathcal{B}$ is a solution of the minimum norm interpolation problem \eqref{mni} with $\mathbf{y}$ if and only if  $\hat{f}\in\mathcal{M}_{\mathbf{y}}$ and there exist $c_j\in\mathbb{R}$, $j\in\mathbb{N}_m,$ such that
\begin{equation}\label{rt-Gateaux}
\mathcal{G}(\hat{f})=\sum_{j\in\mathbb{N}_m}c_j\nu_j.
\end{equation}
\end{thm}

\begin{proof}
If $\mathbf{y}:=[y_j:j\in\mathbb{N}_m]=0,$ the minimum norm interpolation problem \eqref{mni} has a unique solution $\hat{f}=0.$ It is clear that the trivial solution $\hat{f}\in\mathcal{M}_{0}$ and equation \eqref{rt-Gateaux} holds by choosing $c_j=0,$ $j\in\mathbb{N}_m$. Conversely, if $\hat{f}\in\mathcal{M}_{0}$ and satisfies \eqref{rt-Gateaux} for some $c_j\in\mathbb{R}$, $j\in\mathbb{N}_m,$ we have that
$$
\langle\mathcal{G}(\hat{f}),\hat{f}\rangle_{\mathcal{B}}
=\sum_{j\in\mathbb{N}_m}c_j\langle\nu_j,\hat{f}\rangle_{\mathcal{B}}
=\sum_{j\in\mathbb{N}_m}c_jy_j=0,
$$
which together with \eqref{relation-norm-GateauxDiff} implies $\hat{f}=0$.

We prove this theorem for the case that $\mathbf{y}\neq0$ by employing the equivalent conditions (i) and (v) in Proposition \ref{representer-theorem-summarize}.
To this end, we first show that the subdifferential of the norm $\|\cdot\|_{\mathcal{B}}$ at any $f\in\mathcal{B}\backslash\{0\}$ is the singleton $\mathcal{G}(f)$, that is,
\begin{equation}\label{singleton}
\partial\|\cdot\|_{\mathcal{B}}(f)=\{\mathcal{G}(f)\}.
\end{equation}
Suppose that $\nu\in\partial\|\cdot\|_{\mathcal{B}}(f)$. Let $t\in\mathbb{R}$ and $h\in\mathcal{B}$. Then equation \eqref{subdifferentiable} with $\phi:=\|\cdot\|_{\mathcal{B}}$ and $g:=f+th$ leads to
$$
t\langle\nu,h\rangle_{\mathcal{B}}\leq\|f+th\|_{\mathcal{B}}-\|f\|_{\mathcal{B}},
$$
which further implies
$$
\lim_{t\rightarrow 0^{-}}\frac{\|f+th\|_{\mathcal{B}}-\|f\|_{\mathcal{B}}}{t}
\leq\langle\nu,h\rangle_{\mathcal{B}}
\leq\lim_{t\rightarrow 0^{+}}\frac{\|f+th\|_{\mathcal{B}}-\|f\|_{\mathcal{B}}}{t}.
$$
Since $\mathcal{B}$ is smooth, the norm $\|\cdot\|_{\mathcal{B}}$ is G\^{a}teaux differentiable at $f$. Hence, we get that
$$
\langle\nu,h\rangle_{\mathcal{B}}
=\lim_{t\rightarrow 0}\frac{\|f+th\|_{\mathcal{B}}-\|f\|_{\mathcal{B}}}{t}.
$$
It follows from equation \eqref{GateauxDiff} that $\nu=\mathcal{G}(f)$. Due to the arbitrariness of $\nu\in\partial\|\cdot\|_{\mathcal{B}}(f)$, we obtain equation \eqref{singleton}.

According to Proposition \ref{representer-theorem-summarize}, $\hat{f}\in\mathcal{B}$ is a solution of the minimum norm interpolation problem \eqref{mni} with $\mathbf{y}$ if and only if $\hat{f}\in\mathcal{M}_{\mathbf{y}}$ and there exist $c_j\in\mathbb{R},\ j\in\mathbb{N}_m,$ such that the functional $\nu:=\sum_{j\in\mathbb{N}_m}c_j\nu_j$ belongs to $\partial\|\cdot\|_{\mathcal{B}}(\hat{f})$. By equation \eqref{singleton} and noting that $\hat{f}\neq0$, the functional $\nu$ coincides with the G\^{a}teaux derivative of the norm $\|\cdot\|_{\mathcal{B}}$ at $\hat{f}$, which completes the proof of the desired result.
\end{proof}

\subsection{Explicit Representer Theorems}

We derive explicit representer theorems for the minimum norm interpolation problem \eqref{mni} in this subsection. Theorems \ref{representer-theorem-subdifferential} and \ref{representer-theorembyGateaux} provide implicit representer theorems for a solution $\hat f$ of the minimum norm interpolation problem \eqref{mni}. It would be more informative to have an explicit representation. For this purpose, we first present a duality lemma which enables us to ``solve'' for $\hat f$ from the implicit representer theorems.

\begin{lemma}\label{subdifferential-solvable-thm}
Suppose that $\mathcal{B}$ is a Banach space with the dual space $\mathcal{B}^*$. Let $f\in\mathcal{B}\backslash\{0\}$ and $\nu\in\mathcal{B}^{*}\backslash\{0\}$. Then
\begin{equation}\label{subdifferential-solvable}
\frac{\nu}{\|\nu\|_{\mathcal{B}^{*}}}\in\partial\|\cdot\|_{\mathcal{B}}(f)
\end{equation}
if and only if
\begin{equation}\label{subdifferential-solvable1}
\frac{f}{\|f\|_{\mathcal{B}}}\in\partial\|\cdot\|_{\mathcal{B}^{*}}(\nu).
\end{equation}
\end{lemma}
\begin{proof}
According to equation \eqref{relation-subdifferential-functionals}, $\nu\in\mathcal{B}^*$ satisfies \eqref{subdifferential-solvable} if and only if
$$
\left\langle\frac{\nu}{\|\nu\|_{\mathcal{B}^{*}}},f\right\rangle_{\mathcal{B}}
=\|f\|_{\mathcal{B}},
$$
which is equivalent to
\begin{equation}\label{subdifferential-solvable2}
\langle\nu,f\rangle_{\mathcal{B}}=\|\nu\|_{\mathcal{B}^{*}}\|f\|_{\mathcal{B}}.
\end{equation}
It follows from equation \eqref{natural-map} that \eqref{subdifferential-solvable2} is equivalent to
\begin{equation}\label{subdifferential-solvable3}
\langle f,\nu\rangle_{\mathcal{B}^*}=\|\nu\|_{\mathcal{B}^{*}}\|f\|_{\mathcal{B}}.
\end{equation}
Equation \eqref{subdifferential-solvable3} holds if and only if
$$
\left\langle\frac{f}{\|f\|_{\mathcal{B}}},\nu\right\rangle_{\mathcal{B}^*}
=\|\nu\|_{\mathcal{B}^*}.
$$
Again by equation \eqref{relation-subdifferential-functionals} with $\mathcal{B}$ being replaced by $\mathcal{B}^*$, the above equation is equivalent to that $f\in\mathcal{B}$ satisfies \eqref{subdifferential-solvable1}. Consequently, we obtain the equivalence between \eqref{subdifferential-solvable} and \eqref{subdifferential-solvable1}.
\end{proof}

Combining Lemma \ref{subdifferential-solvable-thm} with the equivalence conditions (i) and (v) in Proposition \ref{representer-theorem-summarize}, we obtain the following explicit representer theorem.

\begin{thm}\label{representer-theorem-subdifferential-explicit}
Suppose that $\mathcal{B}$ is a Banach space with the dual space $\mathcal{B}^*$ and $\nu_j\in\mathcal{B}^*$, $j\in\mathbb{N}_m$, are linearly independent. Let $\mathbf{y}\in\mathbb{R}^m$ and $\mathcal{M}_{\mathbf{y}}$ be defined by \eqref{My}. Then $\hat{f}\in\mathcal{B}$ is a solution of the minimum norm interpolation problem \eqref{mni} with $\mathbf{y}$ if and only if $\hat{f}\in\mathcal{M}_{\mathbf{y}}$ and there exist $c_j\in\mathbb{R}$, $j\in\mathbb{N}_m,$ such that
\begin{equation}\label{rt-subdifferential-explicit}
\hat{f}\in\gamma\partial\|\cdot\|_{\mathcal{B}^*}
\left(\sum_{j\in\mathbb{N}_m}c_j\nu_j\right),
\end{equation}
with $\gamma:=\left\|\sum_{j\in\mathbb{N}_m}c_j\nu_j\right\|_{\mathcal{B}^*}.$
\end{thm}
\begin{proof}
We first prove this result for $\mathbf{y}=0.$ In this case, the minimum norm interpolation problem \eqref{mni} has a unique solution $\hat{f}=0.$ On one hand, if $\hat{f}=0$, there hold $\hat{f}\in\mathcal{M}_{0}$ and equation \eqref{rt-subdifferential-explicit}
with $c_j=0$, $j\in\mathbb{N}_m.$ On the other hand, suppose that $\hat{f}\in\mathcal{M}_{0}$ and equation \eqref{rt-subdifferential-explicit} holds for some $c_j\in\mathbb{R}$, $j\in\mathbb{N}_m.$ Set $\nu:=\sum_{j\in\mathbb{N}_m}c_j\nu_j$. It follows that
$\langle\nu,\hat{f}\rangle_{\mathcal{B}}=0$ and
\begin{equation}\label{rt-subdifferential0}  \hat{f}\in\|\nu\|_{\mathcal{B}^*}\partial\|\cdot\|_{\mathcal{B}^*}(\nu).
\end{equation}
If $\nu=0$, equation \eqref{rt-subdifferential0} leads to $\hat{f}=0$. If $\nu\neq0$, by \eqref{rt-subdifferential0} we have that
$$
\left\langle\frac{\hat{f}}{\|\nu\|_{\mathcal{B}^*}},\nu\right\rangle_{\mathcal{B}^*}
=\|\nu\|_{\mathcal{B}^*},
$$
which together with \eqref{natural-map} leads to
$\langle\nu,\hat{f}\rangle_{\mathcal{B}}=\|\nu\|_{\mathcal{B}^*}^2.$ Since $\langle\nu,\hat{f}\rangle_{\mathcal{B}}=0$, we get that $\nu=0$, which is a contradiction.

We prove this theorem for $\mathbf{y}\neq0$ by employing condition (v) in Proposition \ref{representer-theorem-summarize} and Lemma \ref{subdifferential-solvable-thm}. Proposition \ref{representer-theorem-summarize} ensures that $\hat{f}\in\mathcal{B}$ is a solution of the minimum norm interpolation problem \eqref{mni} if and only if $\hat{f}\in\mathcal{M}_{\mathbf{y}}$ and there exist $\hat{c}_j\in\mathbb{R}$, $j\in\mathbb{N}_m,$ such that
\begin{equation}\label{rt-subdifferential1}
\sum_{j\in\mathbb{N}_m}\hat{c}_j\nu_j\in\partial\|\cdot\|_{\mathcal{B}}(\hat{f}).
\end{equation}
Set $c_j:=\|\hat{f}\|_{\mathcal{B}}\hat{c}_j,$ $j\in\mathbb{N}_m$. It suffices to prove that  $\hat{c}_j\in\mathbb{R}$, $j\in\mathbb{N}_m,$ satisfies
\eqref{rt-subdifferential1} if and only if $c_j\in\mathbb{R},\ j\in\mathbb{N}_m,$ satisfies
\eqref{rt-subdifferential-explicit}.

On one hand, suppose that $\hat{c}_j\in\mathbb{R}$, $j\in\mathbb{N}_m,$ satisfies
\eqref{rt-subdifferential1}. Note that since $\mathbf{y}\neq0$, $\hat{f}\neq0$. Then by equation \eqref{rt-subdifferential1}, the functional $\hat{\nu}:=\sum_{j\in\mathbb{N}_m}\hat{c}_j\nu_j$ satisfies that $\|\hat{\nu}\|_{\mathcal{B}^*}=1$. By Lemma \ref{subdifferential-solvable-thm} we get that
\begin{equation}\label{rt-subdifferential2}
\hat{f}\in\|\hat{f}\|_{\mathcal{B}}\partial\|\cdot\|_{\mathcal{B}^*}(\hat{\nu}).
\end{equation}
Set $\nu:=\sum_{j\in\mathbb{N}_m}c_j\nu_j$. It follows that $\nu=\|\hat{f}\|_{\mathcal{B}}\hat{\nu}$ which together with $\|\hat{\nu}\|_{\mathcal{B}^*}=1$ yields that
\begin{equation}\label{rt-subdifferential3}
\|\nu\|_{\mathcal{B}^*}=\|\hat{f}\|_{\mathcal{B}}.
\end{equation}
Noting by equation \eqref{relation-subdifferential-functionals} that
\begin{equation}\label{rt-subdifferential4}
\partial\|\cdot\|_{\mathcal{B}^*}(\hat{\nu})
=\partial\|\cdot\|_{\mathcal{B}^*}(\nu).
\end{equation}
Substituting \eqref{rt-subdifferential3} and \eqref{rt-subdifferential4} into \eqref{rt-subdifferential2}, we get that
$$
\hat{f}\in\|\nu\|_{\mathcal{B}^*}\partial\|\cdot\|_{\mathcal{B}^*}(\nu).
$$
That is, $c_j\in\mathbb{R}$, $j\in\mathbb{N}_m,$ satisfies
\eqref{rt-subdifferential-explicit}.

On the other hand, suppose that $c_j\in\mathbb{R},\ j\in\mathbb{N}_m,$ satisfies
\eqref{rt-subdifferential-explicit}. Since $\hat{f}\neq0$, we have that the functional $\nu:=\sum_{j\in\mathbb{N}_m}c_j\nu_j$ are nonzero and then there holds \eqref{rt-subdifferential3}. By Lemma \ref{subdifferential-solvable-thm} we obtain that
$$
\frac{\nu}{\|\nu\|_{\mathcal{B}^*}}\in\partial\|\cdot\|_{\mathcal{B}}(\hat{f}),
$$
which together with \eqref{rt-subdifferential3} yields that $\hat{c}_j\in\mathbb{R},\ j\in\mathbb{N}_m,$ satisfies \eqref{rt-subdifferential1}.
\end{proof}


We may get a special representer theorem when $\mathcal{B}$ is a Banach space having the {\it smooth} dual space $\mathcal{B}^*$. Below, we denote by $\mathcal{G}^*(\nu)$ the G\^{a}teaux derivative of the norm $\|\cdot\|_{\mathcal{B}^*}$ at $\nu\in\mathcal{B}^*$.

\begin{thm}\label{representer-theorembyGateaux-explicit}
Suppose that $\mathcal{B}$ is a Banach space having the smooth dual space $\mathcal{B}^{*}$ and $\nu_j\in\mathcal{B}^*$, $j\in\mathbb{N}_m$, are linearly independent. Let $\mathbf{y}\in\mathbb{R}^m$ and $\mathcal{M}_{\mathbf{y}}$ be defined by \eqref{My}. Then $\hat{f}\in\mathcal{B}$ is a solution of the minimum norm interpolation problem \eqref{mni} with $\mathbf{y}$ if and only if $\hat{f}\in\mathcal{M}_{\mathbf{y}}$ and there exist $c_j\in\mathbb{R}$, $j\in\mathbb{N}_m,$ such that
\begin{equation}\label{rt-Gateaux-explicit}
\hat{f}=\rho\mathcal{G}^*\left(\sum_{j\in\mathbb{N}_m}c_j\nu_j\right),
\end{equation}
with $\rho:=\left\|\sum_{j\in\mathbb{N}_m}c_j\nu_j\right\|_{\mathcal{B}^*}$.
\end{thm}
\begin{proof}
We first show that the desired conclusion holds for $\mathbf{y}=0.$ The minimum norm interpolation problem \eqref{mni} with $\mathbf{y}=0$ has a unique solution $\hat{f}=0.$ The trivial solution $\hat{f}\in\mathcal{M}_{0}$ and equation \eqref{rt-Gateaux-explicit} holds by choosing $c_j=0,$ $j\in\mathbb{N}_m$. Conversely, suppose that $\hat{f}\in\mathcal{M}_{0}$ and equation \eqref{rt-Gateaux-explicit} holds for some $c_j\in\mathbb{R}$, $j\in\mathbb{N}_m.$ Set $\nu:=\sum_{j\in\mathbb{N}_m}c_j\nu_j$. If $\nu=0$, equation \eqref{rt-Gateaux-explicit} leads directly to $\hat{f}=0$. If $\nu\neq0$, combining \eqref{rt-Gateaux-explicit} with \eqref{natural-map} we have that
$$
\langle\mathcal{G}^*(\nu),\nu\rangle_{\mathcal{B}^*}
=\frac{1}{\|\nu\|_{\mathcal{B}^*}}\langle\hat{f},\nu\rangle_{\mathcal{B}^*}
=\frac{1}{\|\nu\|_{\mathcal{B}^*}}\langle\nu,\hat{f}\rangle_{\mathcal{B}}.
$$
It follows from $\hat{f}\in\mathcal{M}_{0}$ that $\langle\nu,\hat{f}\rangle_{\mathcal{B}}=0.$
Hence, we get that $\langle\mathcal{G}^*(\nu),\nu\rangle_{\mathcal{B}^*}=0$, which together with \eqref{relation-norm-GateauxDiff} implies $\nu=0$. This leads to an contradiction.

We prove this theorem for $\mathbf{y}\neq0$ by Theorem \ref{representer-theorem-subdifferential-explicit}. The minimum norm interpolation problem \eqref{mni} with $\mathbf{y}\neq0$ has no trival solution. According to Theorem \ref{representer-theorem-subdifferential-explicit}, $\hat{f}\in\mathcal{B}$ is a solution of the minimum norm interpolation problem \eqref{mni} with $\mathbf{y}$ if and only if
$\hat{f}\in\mathcal{M}_{\mathbf{y}}$ and there exist $c_j\in\mathbb{R}$, $j\in\mathbb{N}_m,$ such that \eqref{rt-subdifferential-explicit}. Since $\mathcal{B}^*$ is smooth, equation \eqref{singleton} with $\mathcal{B}$ being replaced by $\mathcal{B}^*$ leads to
$$
\partial\|\cdot\|_{\mathcal{B}^*}\left(\sum_{j\in\mathbb{N}_m}c_j\nu_j\right)
=\left\{\mathcal{G}^*\left(\sum_{j\in\mathbb{N}_m}c_j\nu_j\right)\right\}.
$$
Substituting the above equation into \eqref{rt-subdifferential-explicit}, with noting that the set $\partial\|\cdot\|_{\mathcal{B}^*}\left(\sum_{j\in\mathbb{N}_m}c_j\nu_j\right)$ is a singleton yields the formula \eqref{rt-Gateaux-explicit}.
\end{proof}

Next, we establish the representer theorems in a special case when the Banach space $\mathcal{B}$ has the pre-dual space $\mathcal{B}_{*}$ and $\nu_j\in\mathcal{B}_*$, $j\in\mathbb{N}_m$.
\begin{thm}\label{representer-theorem-subdifferential-predual}
Suppose that $\mathcal{B}$ is a Banach space having the pre-dual space $\mathcal{B}_{*}$ and $\nu_j\in\mathcal{B}_*$, $j\in\mathbb{N}_m$, are linearly independent. Let $\mathbf{y}\in\mathbb{R}^m$ and $\mathcal{M}_{\mathbf{y}}$ be defined by \eqref{My}. Then $\hat{f}\in\mathcal{B}$ is a solution of the minimum norm interpolation problem \eqref{mni} with $\mathbf{y}$ if and only if $\hat{f}\in\mathcal{M}_{\mathbf{y}}$ and there exist $c_j\in\mathbb{R}$, $j\in\mathbb{N}_m,$ such that
\begin{equation}\label{rt-subdifferential-predual}
\hat{f}\in\gamma\partial\|\cdot\|_{\mathcal{B}_*}
\left(\sum_{j\in\mathbb{N}_m}c_j\nu_j\right),
\end{equation}
with $\gamma:=\left\|\sum_{j\in\mathbb{N}_m}c_j\nu_j\right\|_{\mathcal{B}_*}$.
\end{thm}
\begin{proof}
Theorem \ref{representer-theorem-subdifferential-explicit}
ensures that $\hat{f}\in\mathcal{B}$ is a solution of \eqref{mni} with $\mathbf{y}$ if and only if $\hat{f}\in\mathcal{M}_{\mathbf{y}}$ and there exist $c_j\in\mathbb{R}$, $j\in\mathbb{N}_m,$ satisfying \eqref{rt-subdifferential-explicit}. It suffices to show that $\hat{f}\in\mathcal{B}$ satisfies \eqref{rt-subdifferential-explicit} if and only if it satisfies \eqref{rt-subdifferential-predual}. Set $\nu:=\sum_{j\in\mathbb{N}_m}c_j\nu_j.$ According to equation \eqref{relation-subdifferential-functionals}, $\hat{f}$ satisfies \eqref{rt-subdifferential-explicit} if and only if
$$
\langle\hat{f},\nu\rangle_{\mathcal{B}^*}
=\|\nu\|_{\mathcal{B}^*}^2.
$$
Since $\nu\in\mathcal{B}_*$, by equations \eqref{natural-map} and \eqref{natural-map-predual}, the above equation is equivalent to
$$
\langle\hat{f},\nu\rangle_{\mathcal{B}_*}
=\|\nu\|_{\mathcal{B}_*}^2.
$$
Again by \eqref{relation-subdifferential-functionals} with $\mathcal{B}$ being replaced by $\mathcal{B}_*$, we get that the above equation holds if and only if
$\hat{f}$ satisfies \eqref{rt-subdifferential-predual}, proving the desired result.
\end{proof}

Theorem \ref{representer-theorem-subdifferential-predual} may be reduced to a nice simple form when $\mathcal{B}$ is a Banach space having the {\it smooth} pre-dual space $\mathcal{B}_*$. We denote by $\mathcal{G}_*(\nu)$ the G\^{a}teaux derivative of the norm $\|\cdot\|_{\mathcal{B}_*}$ at $\nu\in\mathcal{B}_*$.

\begin{thm}\label{representer-theorembyGateaux-predual}
Suppose that $\mathcal{B}$ is a Banach space having the smooth pre-dual space $\mathcal{B}_{*}$ and $\nu_j\in\mathcal{B}_*$, $j\in\mathbb{N}_m$, are linearly independent. Let $\mathbf{y}\in\mathbb{R}^m$ and $\mathcal{M}_{\mathbf{y}}$ be defined by \eqref{My}. Then $\hat{f}\in\mathcal{B}$ is a solution of the minimum norm interpolation problem \eqref{mni} with $\mathbf{y}$ if and only if $\hat{f}\in\mathcal{M}_{\mathbf{y}}$ and there exist $c_j\in\mathbb{R}$, $j\in\mathbb{N}_m,$ such that
\begin{equation*}\label{rt-Gateaux-predual}
\hat{f}=\rho\mathcal{G}_*\left(\sum_{j\in\mathbb{N}_m}c_j\nu_j\right),
\end{equation*}
with $\rho:=\left\|\sum_{j\in\mathbb{N}_m}c_j\nu_j\right\|_{\mathcal{B}_*}$.
\end{thm}
\begin{proof}
By arguments similar to those used in the proof of Theorem \ref{representer-theorembyGateaux-explicit}, we can get the desired conclusion in the case that $\mathbf{y}=0$. For the case that $\mathbf{y}\neq 0$, by employing Theorem \ref{representer-theorem-subdifferential-predual} and noting that the subdifferential of the norm $\|\cdot\|_{\mathcal{B}_*}$ at $\nu:=\sum_{j\in\mathbb{N}_m}c_j\nu_j$ is the singleton $\mathcal{G}_*(\nu)$, we obtain the desired result.
\end{proof}

\subsection{Special Cases}

In this subsection, we consider several specific cases of practical interest, present special results for the cases and identify the connection of the representer theorems established here with existing results in the literature.

An immediate consequence of Theorem \ref{representer-theorembyGateaux-explicit} is the classical representer theorem for the minimum norm interpolation in an RKHS \cite{Wen}. In this special case, the functionals $\nu_j\in\mathcal{B}^*$, $j\in\mathbb{N}_m$, are the point-evaluation functionals $\delta_{x_j},$ $j\in\mathbb{N}_m$, where $x_j$, $j\in\mathbb{N}_m,$ are points in an input set $X$. We call a Hilbert space $\mathcal{H}$ of functions on $X$ an RKHS if the point-evaluation functionals are continuous on $\mathcal{H}$. According to the Riesz representation theorem, for each RKHS $\mathcal{H}$ there exists a unique reproducing kernel $K:X\times X\rightarrow\mathbb{R}$ such that $K(x,\cdot)\in \mathcal{H}$ for all $x\in X$ and
\begin{equation}\label{reproducing}
f(x)=\langle f,K(x,\cdot)\rangle_{\mathcal{H}},
\ \ \mbox{for all}\ \ x\in X
\ \ \mbox{and all}\ \ f\in\mathcal{H}.
\end{equation}

\begin{cor}\label{representer-theorem-RKHS}
Suppose that $\mathcal{H}$ is an RKHS on $X$ with the reproducing kernel $K$, $x_j\in X$, $j\in\mathbb{N}_m,$ and $K(x_j,\cdot)$, $j\in\mathbb{N}_m,$ are linearly independent. If $\mathbf{y}\in\mathbb{R}^m$ and $\mathcal{L}$ is defined by \eqref{functional-operator} with $\nu_j:=\delta_{x_j},$ $j\in\mathbb{N}_m,$ then the minimum norm interpolation problem \eqref{mni} with $\mathbf{y}$ has a unique solution $\hat{f}$ in the form
\begin{equation}\label{representer-RKHS}
\hat{f}=\sum_{j\in\mathbb{N}_m}c_jK(x_j,\cdot),
\end{equation}
for some $c_j\in\mathbb{R}$, $j\in\mathbb{N}_m.$
\end{cor}
\begin{proof}
Clearly, the minimum norm interpolation problem in the RKHS $\mathcal{H}$ has a unique solution $\hat{f}$. By the reproducing property \eqref{reproducing}, a function value of $\hat f$ is identical to its inner product with the kernel $K$. Theorem \ref{representer-theorembyGateaux-explicit} with $\mathcal{B}:=\mathcal{H}$ and $\nu_j:=K(x_j,\cdot)$, $j\in\mathbb{N}_m$, ensures that there exists $c_j\in\mathbb{R},$ $j\in\mathbb{N}_m$, such that
\begin{equation}\label{representer-RKHS1}
\hat{f}=\left\|\sum_{j\in\mathbb{N}_m}c_jK(x_j,\cdot)\right\|_{\mathcal{H}}
\mathcal{G}\left(\sum_{j\in\mathbb{N}_m}c_jK(x_j,\cdot)\right),
\end{equation}
Note that for any $f\in\mathcal{H}$ there holds
\begin{equation}\label{GateauxDiff-H}
f=\|f\|_{\mathcal{H}}\mathcal{G}(f).
\end{equation}
Substituting \eqref{GateauxDiff-H} with $f:=\sum_{j\in\mathbb{N}_m}c_jK(x_j,\cdot)$ into \eqref{representer-RKHS1}, we get formula \eqref{representer-RKHS}.
\end{proof}

We next consider the minimum norm interpolation in a Banach space $\mathcal{B}$ that has two special properties: uniform Fr\'{e}chet smoothness and uniform convexity. Such a space can provides a semi-inner-product as a useful tools for representing the solution of the minimum norm interpolation problem. The norm of a normed space $\mathcal{B}$ is said to be uniformly Fr\'{e}chet differentiable if the limit in \eqref{GateauxDiff-limit} exists for every $f\in\mathcal{B}\setminus\{0\}$ and for every $h\in\mathcal{B}$, and the convergence is uniform for all $f$, $h$ in the unit sphere of $\mathcal{B}$. Accordingly, a normed space is uniformly Fr\'{e}chet smooth if its norm is uniformly Fr\'{e}chet differentiable. A normed space $\mathcal{B}$ is uniformly convex if for all $\varepsilon>0$ there exists a $\delta>0$ such that
$$
\|f+g\|_{\mathcal{B}}\leq2-\delta
$$
for all $f$, $g$ in the unite sphere of $\mathcal{B}$ with $\|f-g\|_{\mathcal{B}}\geq\varepsilon.$
The Milman-Pettis Theorem \cite{Meg} states that every uniformly convex Banach space $\mathcal{B}$ is reflexive.

It follows from \cite{Giles} that for a smooth Banach space $\mathcal{B}$, there exists a unique semi-inner-product $[\cdot,\cdot]_{\mathcal{B}}:
\mathcal{B}\times\mathcal{B}\rightarrow\mathbb{R}$ that induces its norm by
$$
\|f\|_{\mathcal{B}}:=[f,f]_{\mathcal{B}}^{1/2},\ \ \mbox{for all}\ \ f\in\mathcal{B}.
$$
Note that the semi-inner-product $[\cdot,\cdot]_{\mathcal{B}}$ is {\it not} linear with respect to its second variable. For each $g\in\mathcal{B}$, we introduce the linear functional $\nu$ on $\mathcal{B}$ by
$$
\nu(f):=[f,g]_\mathcal{B},\ \ \mbox{for all}\ \ f\in\mathcal{B}.
$$
Then by the Cauchy-Schwartz inequality, the linear functional $\nu$ is continuous. Following \cite{ZXZ}, this functional, denoted by $g^{\sharp}$, is called the dual element of $g$. That is,
\begin{equation}\label{dual-element}
\langle g^{\sharp},f\rangle_{\mathcal{B}}=[f,g]_{\mathcal{B}},
\ \  \mbox{for all}\ \ f\in\mathcal{B}.
\end{equation}
A generalization of the Riesz Representation Theorem in Banach spaces given in \cite{Giles} states that if $\mathcal{B}$ is a uniformly Fr\'{e}chet smooth and uniformly convex Banach space, then for each $\nu\in\mathcal{B}^*$, there exists a unique $g\in\mathcal{B}$ such that
$$
\nu=g^{\sharp}\ \ \mbox{and} \ \ \|\nu\|_{\mathcal{B}^*}=\|g\|_{\mathcal{B}}.
$$
Accordingly, the mapping $f\rightarrow f^{\sharp}$ is bijective from $\mathcal{B}$ to $\mathcal{B}^*.$

There is a well-known relation between uniform Fr\'{e}chet smoothness and uniform convexity \cite{Meg}. Specifically, a normed space is uniformly convex if and only if its dual space is uniformly Fr\'{e}chet smooth, and is uniformly Fr\'{e}chet smooth if and only if its dual space is uniformly convex. Hence, if a Banach space $\mathcal{B}$ is uniformly Fr\'{e}chet smooth and uniformly convex, then so is its dual space $\mathcal{B}^*$. Accordingly, there also exists a unique semi-inner-product $[\cdot,\cdot]_{\mathcal{B}^*}:\mathcal{B}^*\times\mathcal{B}^*\rightarrow\mathbb{R}$ that induces the norm of $\mathcal{B}^*$. Furthermore, it was pointed out in \cite{Giles} that
\begin{equation}\label{semi-inner-product-B^*}
[f^{\sharp},g^{\sharp}]_{\mathcal{B}^*}=[g,f]_{\mathcal{B}},\ \  \mbox{for all}\ \ f,g\in\mathcal{B},
\end{equation}
defines the semi-inner-product on $\mathcal{B}^*$. Again, the semi-inner-product $[\cdot,\cdot]_{\mathcal{B}^*}$ is {\it not} linear with respect to the second variable.
Note that $\mathcal{B}$ is reflexive. According to the semi-inner-product \eqref{semi-inner-product-B^*} on $\mathcal{B}^*$, we can also define the dual element $\nu^{\sharp}\in\mathcal{B}$ of $\nu\in\mathcal{B}^*$ as
\begin{equation}\label{dual-element-B^*}
\langle\mu,\nu^{\sharp}\rangle_{\mathcal{B}}:=[\mu,\nu]_{\mathcal{B}^*}, \ \  \mbox{for all}\ \ \mu\in\mathcal{B}^*.
\end{equation}

As a consequence of Theorems \ref{representer-theorembyGateaux} and \ref{representer-theorembyGateaux-explicit}, we get the following representer theorems for the minimum norm interpolation problem in a uniformly Fr\'{e}chet smooth and uniformly convex Banach space $\mathcal{B}$. In this case, the linearly independent functionals $\nu_j$ can be identified with $g_j^{\sharp}$ for $g_j\in\mathcal{B}$, $j\in\mathbb{N}_m$.

\begin{thm}\label{representer-theorem-ufuc}
If $\mathcal{B}$ is a uniformly Fr\'{e}chet smooth and uniformly convex Banach space with the dual space $\mathcal{B}^*$ and $g_j\in\mathcal{B}$, $j\in\mathbb{N}_m,$ such that $g_j^{\sharp}\in\mathcal{B}^*$, $j\in\mathbb{N}_m,$ are linearly independent, then the minimum norm interpolation problem \eqref{mni} with $\mathbf{y}\in\mathbb{R}^m$ has a unique solution $\hat{f}$ such that
\begin{equation}\label{representer-ufuc}
\hat{f}^{\sharp}=\sum_{j\in\mathbb{N}_m}c_jg_j^{\sharp},
\end{equation}
for some $c_j\in\mathbb{R}$, $j\in\mathbb{N}_m.$
\end{thm}
\begin{proof}
Since $\mathcal{B}$ is uniformly convex, we have that $\mathcal{B}$ is reflexive and strict convex. By Corollary \ref{existence-mni-reflexive}, the reflexivity of $\mathcal{B}$ ensures the existence of a solution of \eqref{mni}. The uniqueness of the solution can also be obtained by strict convexity. That is, the minimum norm interpolation problem \eqref{mni} has a unique solution $\hat{f}$.

Since $\mathcal{B}$ is smooth, Theorem \ref{representer-theorembyGateaux} ensures that there exist $\tilde{c}_j\in\mathbb{R},\ j\in\mathbb{N}_m,$ such that
\begin{equation}\label{GateauxDiff-dual}
\mathcal{G}(\hat{f})=\sum_{j\in\mathbb{N}_m}\tilde{c}_jg_j^{\sharp}.
\end{equation}
It suffices to identify the G\^{a}teaux derivative $\mathcal{G}(\hat{f})$ with the dual element $\hat{f}^{\sharp}$ of $\hat{f}$. The relation between the semi-inner-product and the G\^{a}teaux derivative of the norm $\|\cdot\|_{\mathcal{B}}$ was given in \cite{Giles}, that is,
$$
\lim_{t\rightarrow 0}\frac{\|g+tf\|_{\mathcal{B}}-\|g\|_{\mathcal{B}}}{t}
=\frac{[f,g]_{\mathcal{B}}}{\|g\|_{\mathcal{B}}},
\ \ \mbox{for all}\ \ f,g\in\mathcal{B}\ \ \mbox{and}\  \ g\neq0.
$$
This together with \eqref{GateauxDiff} and \eqref{dual-element} leads to
\begin{equation}\label{relation-gateaux-dual}
\mathcal{G}(g)=\frac{g^{\sharp}}{\|g\|_{\mathcal{B}}},\ \ \mbox{for all}\ \ g\in\mathcal{B}\setminus\{0\}.
\end{equation}
Notice that for $g=0$ there holds $g^{\sharp}=0$ and $\mathcal{G}(g)=0$. Hence, we generalize formula
\eqref{relation-gateaux-dual} as
\begin{equation}\label{relation-gateaux-dual1}
g^{\sharp}=\|g\|_{\mathcal{B}}\mathcal{G}(g),\ \ \mbox{for all}\ \ g\in\mathcal{B}.
\end{equation}
Substituting this representation with $g:=\hat{f}$ into the left hand side of equation \eqref{GateauxDiff-dual}, we get that
$$
\hat{f}^{\sharp}=\sum_{j\in\mathbb{N}_m}\|\hat{f}\|_{\mathcal{B}}\tilde{c}_jg_j^{\sharp}.
$$
Choosing $c_j:=\|\hat{f}\|_{\mathcal{B}}\tilde{c}_j$, $j\in\mathbb{N}_m$, we obtain the desired formula \eqref{representer-ufuc}.
\end{proof}

Note that the dual space $\mathcal{B}^*$ is also uniformly Fr\'{e}chet smooth and uniformly convex. In this case, Theorem \ref{representer-theorembyGateaux-explicit} enables the solution $\hat{f}$ of the minimum norm interpolation problem to have an explicit representation. To this end, we denote by $\mathcal{G}^*(\nu)$ the G\^{a}teaux derivative of the norm $\|\cdot\|_{\mathcal{B}^*}$ at $\nu\in\mathcal{B}^*$.

\begin{thm}\label{representer-theorem-ufuc1}
If $\mathcal{B}$ is a uniformly Fr\'{e}chet smooth and uniformly convex Banach space with the dual space $\mathcal{B}^*$ and $g_j\in\mathcal{B}$, $j\in\mathbb{N}_m,$ such that $g_j^{\sharp}\in\mathcal{B}^*$, $j\in\mathbb{N}_m,$ are linearly independent, then the minimum norm interpolation problem \eqref{mni} with $\mathbf{y}\in\mathbb{R}^m$ has a unique solution $\hat{f}$ in the form
\begin{equation}\label{representer-ufuc1}
\hat{f}=\left(\sum_{j\in\mathbb{N}_m}c_jg_j^{\sharp}\right)^{\sharp},
\end{equation}
for some $c_j\in\mathbb{R},\ j\in\mathbb{N}_m.$
\end{thm}
\begin{proof}
We note that the minimum norm interpolation problem \eqref{mni} in the uniformly Fr\'{e}chet smooth and uniformly convex Banach space $\mathcal{B}$ has a unique solution $\hat{f}$. We establish equation \eqref{representer-ufuc1} by using Theorem \ref{representer-theorembyGateaux-explicit}. The fact that $\mathcal{B}$ is  uniformly convex guarantees that $\mathcal{B}^*$ is uniformly Fr\'{e}chet smooth. Thus, the hypotheses of Theorem \ref{representer-theorembyGateaux-explicit} are satisfied. By Theorem \ref{representer-theorembyGateaux-explicit} there exist $c_j\in\mathbb{R},$ $j\in\mathbb{N}_m,$ such that
\begin{equation*}
\hat{f}=\rho\mathcal{G}^*
\left(\sum_{j\in\mathbb{N}_m}c_jg_j^{\sharp}\right),
\end{equation*}
with $\rho:=\left\|\sum_{j\in\mathbb{N}_m}c_jg_j^{\sharp}\right\|_{\mathcal{B}^*}$.  Combining the above equation with \eqref{relation-gateaux-dual1}, we get the desired formula \eqref{representer-ufuc1}.
\end{proof}

We remark that Theorem \ref{representer-theorem-ufuc1} can also be obtained directly from Theorem \ref{representer-theorem-ufuc} with the following simple fact.

\begin{lemma}
If $\mathcal{B}$ is a uniformly Fr\'{e}chet smooth and uniformly convex Banach space, then for any $f\in\mathcal{B}$, there holds $f^{\sharp\sharp}=f.$
\end{lemma}
\begin{proof}
On one hand, since $f^{\sharp\sharp}$ is the dual element of $f^{\sharp}\in\mathcal{B}^*$, for any $g^{\sharp}\in\mathcal{B}^*$ there holds
$$
\langle g^{\sharp},f^{\sharp\sharp}\rangle_{\mathcal{B}}=\langle f^{\sharp\sharp},g^{\sharp}\rangle_{\mathcal{B}^*}
=[g^{\sharp},f^{\sharp}]_{\mathcal{B}^*}.
$$
On the other hand, for any $g^{\sharp}\in\mathcal{B}^*$ there holds
$$
\langle g^{\sharp},f\rangle_{\mathcal{B}}=[f,g]_{\mathcal{B}}.
$$
Hence, according to \eqref{semi-inner-product-B^*}, we get that
$$
\langle g^{\sharp},f^{\sharp\sharp}\rangle_{\mathcal{B}}
=\langle g^{\sharp},f\rangle_{\mathcal{B}},
\ \ \mbox{for all}\ \ g^{\sharp}\in\mathcal{B}^*,
$$
which further implies $f^{\sharp\sharp}=f.$
\end{proof}

As a special example, we consider the minimum norm interpolation problem in the Banach space $\ell_p(\mathbb{N})$ with $1<p<+\infty$, which is the Banach space of all real sequences
$$
\mathbf{x}:=(x_j: j\in \bN),\ \ \mbox{with}\ \ \|\mathbf{x}\|_p:=\left(\sum_{j\in\bN}|x_j|^p\right)^{1/p}<+\infty.
$$
The space $\ell_p(\bN)$ is uniformly Fr\'{e}chet smooth and uniformly convex and has $\ell_q(\bN)$ as its dual space, where $1/p+1/q=1$. The dual bilinear form $\langle\cdot,\cdot\rangle_{\ell_p}$ on $\ell_q(\bN)\times \ell_p(\bN)$ is defined by
$$
\langle\mathbf{u},\mathbf{x}\rangle_{\ell_p}:=\sum_{j\in\mathbb{N}}u_jx_j,
$$
for all $\mathbf{u}:=(u_j:j\in\mathbb{N})\in \ell_q(\bN)$ and all $\mathbf{x}:=(x_j:j\in\mathbb{N})\in \ell_p(\bN)$. In this case, we suppose that $\mathbf{u}_j$, $j\in\mathbb{N}_m,$ are a finite number of linearly independent elements in $\ell_q(\bN)$ and the operator $\mathcal{L}:\ell_p(\bN)\rightarrow\mathbb{R}^m$, defined as in \eqref{functional-operator}, has the form
\begin{equation}\label{functional-operator-lp}
\mathcal{L}(\mathbf{x}):=[\langle\mathbf{u}_j,
\mathbf{x}\rangle_{\ell_p}:j\in\mathbb{N}_m],
\ \ \mbox {for all}\ \ \mathbf{x}\in \ell_p(\bN).
\end{equation}

Applying Theorem \ref{representer-theorem-ufuc1} to the minimum norm interpolation in $\ell_p(\bN)$, we get below the explicit representation of the solution.

\begin{cor}\label{representer-theorem-lp}
Suppose that $\mathbf{u}_j$, $j\in\mathbb{N}_m,$ are a finite number of linearly independent elements in $\ell_q(\bN)$ with $1/p+1/q=1$. If $\mathbf{y}\in\mathbb{R}^m\backslash\{0\}$ and $\mathcal{L}$ is defined by \eqref{functional-operator-lp}, then the minimum norm interpolation problem \eqref{mni} with $\mathbf{y}$ in $\ell_p(\bN)$ has a unique solution $\hat{\mathbf{x}}:=(\hat{x}_j:j\in\mathbb{N})$ and it has the form
\begin{equation}\label{representer-lp}
\hat{x}_j=\frac{u_j|u_j|^{q-2}}{\|\mathbf{u}\|_{\ell_q}^{q-2}},
\end{equation}
with
\begin{equation}\label{definition-of-u}
\mathbf{u}=(u_j:j\in\mathbb{N}):=\sum_{j\in\mathbb{N}_m}c_j\mathbf{u}_j,
\end{equation}
for some $c_j\in\mathbb{R}$, $j\in\mathbb{N}_m.$
\end{cor}

\begin{proof}
Since the space $\ell_p(\bN)$ is uniformly Fr\'{e}chet smooth and uniformly convex, Theorem \ref{representer-theorem-ufuc1} ensures that the minimum norm interpolation problem \eqref{mni} has a unique solution $\hat{\mathbf{x}}\in\ell_p(\bN)$ and there exist $c_j\in\mathbb{R}$, $j\in\mathbb{N}_m,$ such that
$$
\hat{\mathbf{x}}=\left(\sum_{j\in\mathbb{N}_m}c_j\mathbf{u}_j\right)^{\sharp},
$$
Using the notation of $\mathbf{u}=(u_j:j\in\mathbb{N})$ defined by \eqref{definition-of-u}, the above representation is simplified as $\hat{\mathbf{x}}=\mathbf{u}^{\sharp}$. This together with the fact that
$$
\mathbf{u}^{\sharp}
=\left(\frac{u_j|u_j|^{q-2}}{\|\mathbf{u}\|_{q}^{q-2}}:j\in\mathbb{N}\right)
$$
leads to the desired equation \eqref{representer-lp}.
\end{proof}

We next consider minimum norm interpolation in two types of RKBSs $\mathcal{B}$. In these spaces, the functionals $\nu_j\in\mathcal{B}^*$, $j\in\mathbb{N}_m$, also refer to point-evaluation functionals $\delta_{x_j},$ $j\in\mathbb{N}_m$, where $x_j, j\in\mathbb{N}_m,$ are finite points in an input set $X$. The notion of RKBSs was originally introduced in \cite{ZXZ} based on the semi-inner-product. We begin by reviewing the notion of semi-inner-product RKBSs. A Banach space $\mathcal{B}$ of functions on a domain $X$ is called a semi-inner-product RKBS if it is uniformly Fr\'{e}chet smooth and uniformly convex, and the point-evaluation functionals are continuous linear functionals on $\mathcal{B}$. If $\mathcal{B}$ is a semi-inner-product RKBS, then there exists a unique semi-inner-product reproducing kernel $G:X\times X \rightarrow \mathbb{R}$ satisfies that $G(x,\cdot)\in\mathcal{B}$ for all $x\in X$ and
\begin{equation}\label{reproducing-property-ZXZ}
f(x)=[f,G(x,\cdot)]_{\mathcal{B}},
\ \ \mbox{for all}\ \ x\in X
\ \ \mbox{and all}\ \ f\in\mathcal{B}.
\end{equation}

Since a semi-inner-product RKBS is uniformly Fr\'{e}chet smooth and uniformly convex, Theorems \ref{representer-theorem-ufuc} and \ref{representer-theorem-ufuc1} lead directly to the following representer theorems for the minimum norm interpolation problem in a semi-inner-product RKBS.

\begin{cor}\label{representer-theorem-ZXZ}
If $\mathcal{B}$ is a semi-inner-product RKBS with the semi-inner-product reproducing kernel $G$ and $x_j\in X$, $j\in\mathbb{N}_m,$ such that $G(x_j,\cdot)^{\sharp}\in\mathcal{B}^*$, $j\in\mathbb{N}_m,$ are linearly independent, then the minimum norm interpolation problem \eqref{mni} with $\mathbf{y}\in\mathbb{R}^m$ has a unique solution $\hat{f}$ such that
\begin{equation}\label{representer-ZXZ}
\hat{f}^{\sharp}=\sum_{j\in\mathbb{N}_m}c_jG(x_j,\cdot)^{\sharp},
\end{equation}
for some $c_j\in\mathbb{R}$, $j\in\mathbb{N}_m.$
\end{cor}
\begin{proof}
We note by the reproducing property \eqref{reproducing-property-ZXZ} that for each $x\in X$, the dual element $G(x,\cdot)^{\sharp}$ of $G(x,\cdot)$ coincides exactly with the point-evaluation functional $\delta_x$. Thus, Theorem \ref{representer-theorem-ufuc} with $g_j:=G(x_j,\cdot)$, for $j\in\mathbb{N}_m,$ ensures that the minimum norm interpolation problem \eqref{mni} has a unique solution $\hat{f}$ such that \eqref{representer-ZXZ} for some $c_j\in\mathbb{R}$, $j\in\mathbb{N}_m.$
\end{proof}

We remark that the representer theorem for the minimum norm interpolation problem \eqref{mni} in a semi-inner-product RKBS, stated in Corollary \ref{representer-theorem-ZXZ}, was originally obtained in \cite{ZXZ} by a different approach. In \cite{ZXZ}, the representer theorem was proved by using the orthogonality in a semi-inner-product RKBS which can be characterized through the dual element and the semi-inner-product.

By employing Theorem \ref{representer-theorem-ufuc1}, we get below an explicit representation of  the solution $\hat{f}$ of the minimum norm interpolation problem in a semi-inner-product RKBS.

\begin{thm}\label{representer-theorem-ZXZ1}
If $\mathcal{B}$ is a semi-inner-product RKBS with the semi-inner-product reproducing kernel $G$ and $x_j\in X$, $j\in\mathbb{N}_m,$ such that $G(x_j,\cdot)^{\sharp}\in\mathcal{B}^*$, $j\in\mathbb{N}_m,$ are linearly independent, then the minimum norm interpolation problem \eqref{mni} with $\mathbf{y}\in\mathbb{R}^m$ has a unique solution $\hat{f}$ in the form
\begin{equation*}\label{representer-ZXZ1}
\hat{f}=\left(\sum_{j\in\mathbb{N}_m}c_jG(x_j,\cdot)^{\sharp}\right)^{\sharp},
\end{equation*}
for some $c_j\in\mathbb{R},\ j\in\mathbb{N}_m.$
\end{thm}
\begin{proof}
The desired result is an immediate consequence of Theorem \ref{representer-theorem-ufuc1} with $g_j:=G(x_j,\cdot)$, for $j\in\mathbb{N}_m$.
\end{proof}

An alternative definition of RKBSs was introduced in \cite{XY}. This definition is a natural generalization of RKHSs by replacing the inner product in the Hilbert spaces with the dual bilinear form in introducing the reproducing properties in RKBSs. We now apply Theorem \ref{representer-theorembyGateaux} and Theorem \ref{representer-theorembyGateaux-explicit} to the minimum norm interpolation problem in such an RKBS. To this end, we recall the definition of RKBSs. Suppose that $\mathcal{B}$ is a Banach space of functions on $X$ and the dual space $\mathcal{B}^*$ is isometrically equivalent to a Banach space of functions on $X^{'}$. Let $K:X\times X^{'} \rightarrow \mathbb{R}$. A Banach space $\mathcal{B}$ is called a right-sided RKBS and $K$ is its right-sided reproducing kernel if $K(x,\cdot)\in\mathcal{B}^{*}$ for all $x\in X$ and
\begin{equation}\label{reproducing-property-XY}
f(x)=\langle f,K(x,\cdot)\rangle_{\mathcal{B}}, \ \ \mbox{for all}\ \ x\in X \ \ \mbox{and all}\ \ f\in\mathcal{B}.
\end{equation}

In the framework of right-sided RKBSs, we consider the minimum norm interpolation problem with a finite number of point-evaluation functionals $\delta_{x_j},$ $j\in\mathbb{N}_m$, where $x_j\in X$, $j\in\mathbb{N}_m$. The representer theorem for this case can be obtained directly from Theorem \ref{representer-theorembyGateaux}.

\begin{cor}\label{representer-theorem-XY}
Suppose that $\mathcal{B}$ is a right-sided RKBS with the right-sided reproducing kernel $K$ and $x_j\in X$, $j\in\mathbb{N}_m,$ such that $K(x_j,\cdot)\in\mathcal{B}^*$, $j\in\mathbb{N}_m,$ are linearly independent. If $\mathcal{B}$ is reflexive,  strictly convex, and smooth, then the minimum norm interpolation problem \eqref{mni} with $\mathbf{y}\in\mathbb{R}^m$ has a unique solution $\hat{f}$ such that
\begin{equation}\label{representer-XY}
\mathcal{G}(\hat{f})=\sum_{j\in\mathbb{N}_m}c_jK(x_j,\cdot),
\end{equation}
for some $c_j\in\mathbb{R}$, $j\in\mathbb{N}_m.$
\end{cor}
\begin{proof}
The reflexivity and strict convexity of the RKBS $\mathcal{B}$ ensure that the minimum norm interpolation problem \eqref{mni} has a unique solution $\hat{f}$. Since $\mathcal{B}$ is smooth, Theorem \ref{representer-theorembyGateaux} with $\nu_j:=\delta_{x_j},$ $j\in\mathbb{N}_m$, ensures that there exist $c_j\in\mathbb{R},j\in\mathbb{N}_m,$ such that
$$
\mathcal{G}(\hat{f})=\sum_{j\in\mathbb{N}_m}c_j\delta_{x_j}.
$$
It follows from the reproducing property \eqref{reproducing-property-XY} that the right-sided reproducing kernel $K$ provides a closed-form function representation for the point-evaluation functionals. Hence, the above representation coincides with \eqref{representer-XY}.
\end{proof}

We note that the representer theorem stated in Corollary \ref{representer-theorem-XY} was originally established in \cite{XY} by a different approach. In \cite{XY}, the solution of the minimum norm interpolation problem was characterized by using orthogonality in the Banach space, which is described through the G\^{a}teaux derivatives and reproducing properties.
Again, formula \eqref{representer-XY} obtained in Corollary \ref{representer-theorem-XY} is an implicit representation of $\hat f$. We prefer having an explicit formula, which we derive from Theorem \ref{representer-theorembyGateaux-explicit}.

\begin{thm}\label{representer-theorem-XY1}
Suppose that $\mathcal{B}$ is a right-sided RKBS with the right-sided reproducing kernel $K$ and $x_j\in X$, $j\in\mathbb{N}_m,$ such that $K(x_j,\cdot)\in\mathcal{B}^*$, $j\in\mathbb{N}_m,$ are linearly independent. If $\mathcal{B}$ is reflexive and strictly convex, then the minimum norm interpolation problem \eqref{mni} with $\mathbf{y}\in\mathbb{R}^m$ has a unique solution $\hat{f}$ in the form
\begin{equation}\label{representer-XY1}
\hat{f}=\rho\mathcal{G}^*
\left(\sum_{j\in\mathbb{N}_m}c_jK(x_j,\cdot)\right),
\end{equation}
for some $c_j\in\mathbb{R}$, $j\in\mathbb{N}_m,$ where $\rho:=\left\|\sum_{j\in\mathbb{N}_m}c_jK(x_j,\cdot)\right\|_{\mathcal{B}^*}$.
\end{thm}
\begin{proof}
As pointed out in Corollary \ref{representer-theorem-XY}, the minimum norm interpolation problem \eqref{mni} has a unique solution $\hat{f}$. Since $\mathcal{B}$ is reflexive, it follows from the strict convexity of $\mathcal{B}$ that $\mathcal{B}^*$ is smooth. The hypotheses of Theorem \ref{representer-theorembyGateaux-explicit} are satisfied. Hence, by Theorem \ref{representer-theorembyGateaux-explicit}, there exist $c_j\in\mathbb{R}$, $j\in\mathbb{N}_m,$ such that
$$
\hat{f}=\rho\mathcal{G}^*\left(\sum_{j\in\mathbb{N}_m}c_j\delta_{x_j}\right),
$$
with $\rho:=\left\|\sum_{j\in\mathbb{N}_m}c_j\delta_{x_j}\right\|_{\mathcal{B}^*}$. By using the reproduction property \eqref{reproducing-property-XY},
we obtain from the above equation the desired formula \eqref{representer-XY1}.
\end{proof}

Finally, we consider the minimum norm interpolation problem in the Banach space $\ell_1(\mathbb{N})$, which was studied in \cite{CX}. Note that $\ell_1(\mathbb{N})$ has $c_0$ as its pre-dual space and $\ell_{\infty}(\mathbb{N})$ as its dual space.
The space $\ell_1(\mathbb{N})$ is not reflexive since the dual space of  $\ell_{\infty}(\mathbb{N})$ is not $\ell_1(\mathbb{N})$, and its pre-dual space $c_0$ is not smooth. The dual bilinear form $\langle\cdot,\cdot\rangle_{\ell_1}$ on $\ell_{\infty}(\mathbb{N})\times \ell_1(\mathbb{N})$ is defined by
$$
\langle\mathbf{u},\mathbf{x}\rangle_{\ell_1}:=\sum_{j\in\mathbb{N}}u_jx_j,
$$
for all $\mathbf{u}:=(u_j:j\in\mathbb{N})\in \ell_{\infty}(\mathbb{N})$ and all
$\mathbf{x}:=(x_j:j\in\mathbb{N})\in \ell_1(\mathbb{N})$. In this case, the functionals in the minimum norm interpolation problem belong to $c_0$. Specifically, we suppose that $\mathbf{u}_j, j\in\mathbb{N}_m,$ are a finite number of linearly independent elements in $c_0$. The operator $\mathcal{L}:\ell_1(\mathbb{N})\rightarrow\mathbb{R}^m$, defined by \eqref{functional-operator}, has the form
\begin{equation}\label{functional-operator-l1}
\mathcal{L}(\mathbf{x}):=[\langle\mathbf{u}_j,
\mathbf{x}\rangle_{\ell_1}:j\in\mathbb{N}_m],
\ \ \mbox {for all}\ \ \mathbf{x}\in \ell_1(\mathbb{N}).
\end{equation}

To obtain the representer theorem for this case from Theorem \ref{representer-theorem-subdifferential-predual}, we need to compute explicitly the subdifferentials of the norm $\|\cdot\|_\infty$ of $c_0$. Let $X$ be a vector space and $\mathcal{V}$ a subset. The convex hull of $\mathcal{V}$, denoted by ${\rm co}(\mathcal{V})$, is the collection of all convex combinations of elements of $\mathcal{V}$, that is,
$$
{\rm co}(\mathcal{V}):=\left\{\sum_{j\in\mathbb{N}_n}t_jx_j: x_j\in \mathcal{V}, t_j\in\mathbb{R}_+:=[0,+\infty), \sum_{j\in\mathbb{N}_n}t_j=1, j\in\mathbb{N}_n, n\in\mathbb{N} \right\}.
$$
For each $\mathbf{u}\in c_0$, we let $\mathbb{N}(\mathbf{u})$ denote the index set on which the sequence $\mathbf{u}$ attains its norm $\|\mathbf{u}\|_{\infty}$, that is,
\begin{equation}\label{extremal-index-set}
\|\mathbf{u}\|_{\infty}=|u_j|, \ \ j\in \mathbb{N}(\mathbf{u}) \ \ \mbox{and}\ \ \|\mathbf{u}\|_{\infty}>|u_j|, \ \ j\notin \mathbb{N}(\mathbf{u}).
\end{equation}
For each $\mathbf{u}\in c_0$, since $u_j$ tends to zero as $j\to +\infty$, we have that
$$
\#\mathbb{N}(\mathbf{u})<+\infty.
$$
To present the subdifferentials of the norm $\|\cdot\|_{\infty}$ of $c_0$ at any $\mathbf{u}:=(u_j:j\in\mathbb{N})\in c_0$, we introduce a subset of $\ell_1(\mathbb{N})$ as
\begin{equation}\label{convex-hull}
\mathcal{V}(\mathbf{u}):=\{{\rm sign}(u_j)\mathbf{e}_j:j\in\mathbb{N}(\mathbf{u})\},
\end{equation}
where for each $j\in\mathbb{N}$, $\mathbf{e}_j$ denotes the vector in $\ell_1(\mathbb{N})$ whose $j$th component is equal to 1 and all other components are zero.

The following lemma which was essentially proved in \cite{CX} presents the subdifferential of the norm $\|\cdot\|_{\infty}$ of $c_0$ at any nonzero $\mathbf{u}:=(u_j:j\in\mathbb{N})\in c_0$.

\begin{lemma}\label{subdifferentials-c0}
If $\mathbf{u}:=(u_j:j\in\mathbb{N})$ is a nonzero element in $c_0$ and $\mathbb{N}(\mathbf{u})$ and $\mathcal{V}(\mathbf{u})$ are defined by \eqref{extremal-index-set} and \eqref{convex-hull}, respectively, then there holds
\begin{equation}\label{subdifferentials-formula-c0}
\partial\|\cdot\|_{\infty}(\mathbf{u})={\rm co}(\mathcal{V}(\mathbf{u})).
\end{equation}
\end{lemma}

Combining Lemma \ref{subdifferentials-c0} and Theorem \ref{representer-theorem-subdifferential-predual}, we may get the following representer theorem for the minimum norm interpolation problem in $\ell_1(\mathbb{N})$.

\begin{thm}\label{representer-theorem-l1}
Suppose that $\mathbf{u}_j, j\in\mathbb{N}_m,$ are a finite number of linearly independent elements in $c_0$. Let $\mathbf{y}\in\mathbb{R}^m$ and $\mathcal{L}$ and $\mathcal{M}_{\mathbf{y}}$ be defined by \eqref{functional-operator-l1} and \eqref{My}, respectively. Then $\hat{\mathbf{x}}\in \ell_1(\mathbb{N})$ is a solution of the minimum norm interpolation problem \eqref{mni} with $\mathbf{y}$ in $\ell_1(\mathbb{N})$ if and only if
$\hat{\mathbf{x}}\in\mathcal{M}_{\mathbf{y}}$ and there exist $c_j\in\mathbb{R}$, $j\in\mathbb{N}_m,$ such that
\begin{equation}\label{rt-subdifferential-l1}
\hat{\mathbf{x}}\in\|\mathbf{u}\|_{\infty}{\rm co}(\mathcal{V}
(\mathbf{u})),
\end{equation}
with $\mathbf{u}:=\sum_{j\in\mathbb{N}_m}c_j\mathbf{u}_j.$
\end{thm}

\begin{proof}
We first consider the case when $\mathbf{y}=0.$ The minimum norm interpolation problem \eqref{mni} with $\mathbf{y}=0$ has a unique solution $\hat{\mathbf{x}}=0.$ On one hand, if $\hat{\mathbf{x}}=0$, there hold $\hat{\mathbf{x}}\in\mathcal{M}_{0}$ and \eqref{rt-subdifferential-l1} with $c_j=0$, $j\in\mathbb{N}_m.$ On the other hand, suppose that $\hat{\mathbf{x}}\in\mathcal{M}_{0}$ and equation \eqref{rt-subdifferential-l1} holds for some $c_j\in\mathbb{R}$, $j\in\mathbb{N}_m.$ Set $\mathbf{u}:=(u_j:j\in\mathbb{N})$. It follows that
$\langle\mathbf{u},\hat{\mathbf{x}}\rangle_{\ell_1}=0$ and
\begin{equation}\label{rt-subdifferential-l10}
\hat{\mathbf{x}}=\|\mathbf{u}\|_{\infty}
\sum_{k\in\mathbb{N}_n}t_k{\rm sign}(u_{j_k})\mathbf{e}_{j_k},
\end{equation}
for some $n\in\mathbb{N}$, $j_k\in \mathbb{N}(\mathbf{u})$, $t_k\in\mathbb{R}_+$, $k\in\mathbb{N}_n$, with $\sum_{k\in\mathbb{N}_n}t_k=1$.
If $\mathbf{u}=0$, equation \eqref{rt-subdifferential-l10} leads to $\hat{\mathbf{x}}=0$. If $\mathbf{u}\neq0$, by equation \eqref{rt-subdifferential-l10} we have that
\begin{equation}\label{rt-subdifferential-l11}
\langle\mathbf{u},\hat{\mathbf{x}}\rangle_{\ell_1}=\|\mathbf{u}\|_{\infty}
\sum_{k\in\mathbb{N}_n}t_k{\rm sign}(u_{j_k})
\langle\mathbf{u},\mathbf{e}_{j_k}\rangle_{\ell_1}
=\|\mathbf{u}\|_{\infty}
\sum_{k\in\mathbb{N}_n}t_k|u_{j_k}|.
\end{equation}
By definition \eqref{extremal-index-set} of $\mathbb{N}(\mathbf{u})$, there holds
$$
|u_{j_k}|=\|\mathbf{u}\|_{\infty},\ \ \mbox{for all}\ \ k\in\mathbb{N}_n.
$$
Substituting the equations above and the fact that $\sum_{k\in\mathbb{N}_n}t_k=1$ into the right hand side of equation \eqref{rt-subdifferential-l11}, we obtain that
$\langle\mathbf{u},\hat{\mathbf{x}}\rangle_{\ell_1}=\|\mathbf{u}\|_{\infty}^2$, which together with $\langle\mathbf{u},\hat{\mathbf{x}}\rangle_{\ell_1}=0$ yields $\mathbf{u}=0$. This is a contradiction.

We prove this result for $\mathbf{y}\neq0$ by employing Theorem \ref{representer-theorem-subdifferential-predual}. Note that the minimum norm interpolation problem \eqref{mni} with $\mathbf{y}\neq0$ has no trival solution. Since $\ell_1(\mathbb{N})$ has $c_0$ as its pre-dual space, Theorem \ref{representer-theorem-subdifferential-predual} ensures that $\hat{\mathbf{x}}\in \ell_1(\mathbb{N})$ is a solution of the minimum norm interpolation problem \eqref{mni} if and only if $\hat{\mathbf{x}}\in\mathcal{M}_{\mathbf{y}}$ and there exist $c_j\in\mathbb{R}$, $j\in\mathbb{N}_m,$ such that
$$
\hat{\mathbf{x}}\in\gamma\partial\|\cdot\|_{\infty}
\left(\sum_{j\in\mathbb{N}_m}c_j\mathbf{u}_j\right),
$$
with $\gamma:=\left\|\sum_{j\in\mathbb{N}_m}c_j\mathbf{u}_j\right\|_{\infty}$. Substituting equation \eqref{subdifferentials-formula-c0} of Lemma \ref{subdifferentials-c0} into the right hand side of the above equation and letting $\mathbf{u}:=\sum_{j\in\mathbb{N}_m}c_j\mathbf{u}_j$, we get the desired equation \eqref{rt-subdifferential-l1}.
\end{proof}

The minimum norm interpolation problem in $\ell_1(\bN)$ was considered in \cite{CX}, where the original optimization problem was reformulated as a dual problem in a finite dimensional vector space, which can be solved by linear programming.

\section{Representer Theorem Based Solution Methods for Minimum Norm Interpolation}
The representer theorems presented in the last section for the minimum norm interpolation problem \eqref{mni} give only forms of the solutions for the problem. They do {\it not} provide methods to determine the coefficients $c_j$ involved in the solution representations. We develop in this section approaches to determine these coefficients, leading to solution methods for solving the minimum norm interpolation problem \eqref{mni} when the Banach space $\mathcal{B}$ has the pre-dual space $\mathcal{B}_{*}$ and $\nu_j\in\mathcal{B}_*$, for $j\in\mathbb{N}_m$. We will consider both cases when the pre-dual space is smooth and when it is non-smooth.

As a preparation, we express the adjoint operator $\mathcal{L}^*$ of $\mathcal{L}$ defined by \eqref{functional-operator}. According to the continuity of the linear operator $\mathcal{L}$ on $\mathcal{B}$, there exists a unique bounded linear operator $\mathcal{L}^*:\mathbb{R}^m\rightarrow\mathcal{B}^*$, called the adjoint operator of $\mathcal{L}$,  such that for any $f\in\mathcal{B}$ and any $\mathbf{c}=[c_j:j\in\mathbb{N}_m]\in\mathbb{R}^m$ that
$$
\langle\mathcal{L}^*(\mathbf{c}),f\rangle_{\mathcal{B}}
=\langle\mathbf{c},\mathcal{L}(f)\rangle_{\mathbb{R}^m}.
$$
It follows from definition \eqref{functional-operator} of $\mathcal{L}$ that
$$
\langle\mathcal{L}^*(\mathbf{c}),f\rangle_{\mathcal{B}}
=\sum_{j\in\mathbb{N}_m}c_j\langle\nu_j,f\rangle_{\mathcal{B}}
=\left\langle\sum_{j\in\mathbb{N}_m}c_j\nu_j,f\right\rangle_{\mathcal{B}},
$$
which leads to
\begin{equation}\label{adjoint-operator}
\mathcal{L}^*(\mathbf{c})=\sum_{j\in\mathbb{N}_m}c_j\nu_j.
\end{equation}

\subsection{Solutions of Minimum Norm Interpolation in a Banach Space with a Smooth Pre-Dual Space}

In this subsection, we provide the complete solution of the minimum norm interpolation problem \eqref{mni} in a Banach space having the smooth pre-dual space $\mathcal{B}_*$. In this case, the solution of the minimum norm interpolation problem \eqref{mni} with data $\mathbf{y}$ can be obtained by employing Theorem \ref{representer-theorembyGateaux-predual} with the coefficients $c_j$ involved in it being chosen as a solution of a system of equations.

\begin{thm}\label{solution-mni-smooth-B_*}
Suppose that $\mathcal{B}$ is a Banach space having the smooth pre-dual space $\mathcal{B}_{*}$ and $\nu_j\in\mathcal{B}_*$, $j\in\mathbb{N}_m$, are linearly independent. Let $\mathbf{y}:=[y_j:j\in\mathbb{N}_m]\in\mathbb{R}^m$, $\mathcal{L}$ be the operator defined by \eqref{functional-operator} and $\mathcal{L}^*$ be the adjoint operator. Then \begin{equation}\label{rt-Gateaux-predual-solve}
\hat{f}:=\|\mathcal{L}^*(\mathbf{c})\|_{\mathcal{B}_*}
\mathcal{G}_*(\mathcal{L}^*(\mathbf{c})),
\end{equation}
is a solution of the minimum norm interpolation problem \eqref{mni} with $\mathbf{y}$ if and only if $\mathbf{c}\in\mathbb{R}^m$ is a solution of the system of equations
\begin{equation}\label{equation-c-smooth-B_*}
\langle\nu_k,\|\mathcal{L}^*(\mathbf{c})\|_{\mathcal{B}_*}
\mathcal{G}_*(\mathcal{L}^*(\mathbf{c}))\rangle_{\mathcal{B}}
=y_k,\ k\in\mathbb{N}_m.
\end{equation}
\end{thm}
\begin{proof}
Suppose that $\hat{f}$ in the form
\eqref{rt-Gateaux-predual-solve} for some $\mathbf{c}\in\mathbb{R}^m$ is a solution of \eqref{mni} with $\mathbf{y}$. Substituting \eqref{rt-Gateaux-predual-solve} into the interpolation condition $\mathcal{L}(\hat{f})=\mathbf{y}$, we have that the vector $\mathbf{c}$ satisfies \eqref{equation-c-smooth-B_*}.

Conversely, we suppose that the vector $\mathbf{c}$ satisfies the system of equations \eqref{equation-c-smooth-B_*}. We first comment that $\hat{f}$ in the form \eqref{rt-Gateaux-predual-solve} is in $\mathcal{B}$ since the operator $\mathcal{G}_*$ maps $\mathcal{B}_{*}$ to $\mathcal{B}$. We will prove by Theorem \ref{representer-theorembyGateaux-predual} that $\hat{f}$ is a solution of the minimum norm interpolation problem \eqref{mni} with $\mathbf{y}$. Substituting \eqref{rt-Gateaux-predual-solve} into \eqref{equation-c-smooth-B_*} leads to the interpolation condition $\mathcal{L}(\hat{f})=\mathbf{y}$. Then by Theorem \ref{representer-theorembyGateaux-predual} and the representation \eqref{adjoint-operator} of the adjoint operator $\mathcal{L}^*$, we conclude that $\hat{f}\in\mathcal{B}$ is a solution of the minimum norm interpolation problem \eqref{mni} with data $\mathbf{y}$.
\end{proof}

There are two interesting special cases. The first one concerns the minimum norm interpolation in a Hilbert space, that is, $\mathcal{B}=\mathcal{H}$ is a Hilbert space. In this case, $\mathcal{H}_*=\mathcal{H}$ and the linearly independent functionals $\nu_j$ can be identified with $g_j\in\mathcal{H},$ for $j\in\mathbb{N}_m$. We then introduce the Gram matrix as
\begin{equation}\label{Gram-matrix}
\mathbf{G}:=[\langle g_j,g_k\rangle_{\mathcal{H}}:j,k\in\mathbb{N}_m].
\end{equation}

\begin{cor}\label{solution-mni-H}
Suppose that $\mathcal{H}$ is a Hilbert space and $g_j\in\mathcal{H}$, $j\in\mathbb{N}_m$, are linearly independent. Let $\mathbf{y}\in\mathbb{R}^m$, $\mathcal{L}$ be defined by \eqref{functional-operator} and $\mathcal{L}^*$ be the adjoint operator. Then \begin{equation}\label{rt-H-solve}
\hat{f}:=\sum_{j\in\mathbb{N}_m}c_jg_j
\end{equation}
is the unique solution of the minimum norm interpolation \eqref{mni} with $\mathbf{y}$ if and only if $\mathbf{c}:=[c_j:j\in\mathbb{N}_m]\in\mathbb{R}^m$ is the solution of the linear system of equations
\begin{equation}\label{equation-c-H}
\mathbf{G}\mathbf{c}=\mathbf{y}.
\end{equation}
\end{cor}
\begin{proof}
According to the linear independence of $g_j\in\mathcal{H}$, $j\in\mathbb{N}_m$, the Gram matrix $\mathbf{G}$ is symmetric and positive definite. Therefore, the linear system \eqref{equation-c-H} has a unique solution.
Proposition \ref{solution-mni-smooth-B_*} ensures that $\hat{f}$ in the form
\eqref{rt-Gateaux-predual-solve} is the solution of the minimum norm interpolation \eqref{mni} with $\mathbf{y}$ if and only if $\mathbf{c}:=[c_j:j\in\mathbb{N}_m]\in\mathbb{R}^m$ satisfies
\eqref{equation-c-smooth-B_*}. It suffices to show that in this special case, $\hat{f}$ in the form \eqref{rt-Gateaux-predual-solve} has the form \eqref{rt-H-solve} and the system of equations
\eqref{equation-c-smooth-B_*} can be represented as the linear system \eqref{equation-c-H}. Substituting
\begin{equation}\label{equation-c-H-1}
\|\mathcal{L}^*(\mathbf{c})\|_{\mathcal{B}_*}
\mathcal{G}_{*}(\mathcal{L}^*(\mathbf{c}))=\mathcal{L}^*(\mathbf{c})
\end{equation}
into \eqref{rt-Gateaux-predual-solve}, we represent $\hat{f}$ in \eqref{rt-H-solve}. Again substituting \eqref{equation-c-H-1} into \eqref{equation-c-smooth-B_*}, we get that
$$
\langle g_k,\mathcal{L}^*(\mathbf{c})\rangle_{\mathcal{H}}
=y_k,\ k\in\mathbb{N}_m.
$$
By the representation \eqref{adjoint-operator} of the adjoint operator $\mathcal{L}^*$, the equation above can be rewritten in the form \eqref{equation-c-H}.
\end{proof}

We remark that in the case of Hilbert spaces, the {\it infinite} dimensional minimum norm interpolation problem is reduced to solving an equivalent {\it finite} dimensional linear system. The only infinite dimensional component in the linear system is computing the entries of the Gram matrix $\mathbf{G}$ which requires calculating the inner produces of $g_k$ and $g_j$, elements in the infinite dimensional space $\mathcal{H}$.

As shown in Corollary \ref{solution-mni-H}, the minimum interpolation problem in a Hilbert space is reduced to solving a {\it linear system}. However, in a Banach space, not Hilbert, the problem cannot be reduced to a linear system. This will be demonstrated in the next case, where $\mathcal{B}$ is a uniformly Fr\'{e}chet smooth and uniformly convex Banach space (thus, in this case $\mathcal{B}_*=\mathcal{B}^*$).

\begin{cor}\label{solution-mni-ufuc}
Suppose that $\mathcal{B}$ is a uniformly Fr\'{e}chet smooth and uniformly convex Banach space having the dual space $\mathcal{B}^*$, and $g_j\in\mathcal{B}$, $j\in\mathbb{N}_m,$ such that $g_j^{\sharp}\in\mathcal{B}^*$, $j\in\mathbb{N}_m,$ are linearly independent. Let $\mathbf{y}:=[y_j:j\in\mathbb{N}_m]\in\mathbb{R}^m$, $\mathcal{L}$ be defined by \eqref{functional-operator} and $\mathcal{L}^*$ be the adjoint operator. Then \begin{equation}\label{rt-ufuc-solve}
\hat{f}=(\mathcal{L}^*(\mathbf{c}))^{\sharp},
\end{equation}
is the unique solution of the minimum norm interpolation \eqref{mni} with $\mathbf{y}$ if and only if $\mathbf{c}\in\mathbb{R}^m$ is the solution of the system of equations
\begin{equation}\label{equation-c-ufuc}
[g_k^{\sharp},\mathcal{L}^*(\mathbf{c})]_{\mathcal{B}^*}
=y_k,\ k\in\mathbb{N}_m.
\end{equation}
\end{cor}
\begin{proof}
By Proposition \ref{solution-mni-smooth-B_*} we have that $\hat{f}$ in the form
\eqref{rt-Gateaux-predual-solve} is the solution of the minimum norm interpolation \eqref{mni} with $\mathbf{y}$ if and only if $\mathbf{c}\in\mathbb{R}^m$ satisfies
\eqref{equation-c-smooth-B_*}. According to the relation between the semi-inner-product and the G\^{a}teaux derivative of the norm $\|\cdot\|_{\mathcal{B}}$
\begin{equation}\label{equation-c-ufuc-1}
\|\mathcal{L}^*(\mathbf{c})\|_{\mathcal{B}^*}
\mathcal{G}^*(\mathcal{L}^*(\mathbf{c}))
=(\mathcal{L}^*(\mathbf{c}))^{\sharp},
\end{equation}
we represent $\hat{f}$ as in \eqref{rt-ufuc-solve}. Substituting \eqref{equation-c-ufuc-1} into \eqref{equation-c-smooth-B_*}, with noting that $\nu_k:=g_k^{\sharp}$ for all $k\in\mathbb{N}_m$, we have that
\begin{equation*}
\langle g_k^{\sharp},(\mathcal{L}^*(\mathbf{c}))^{\sharp}\rangle_{\mathcal{B}}
=y_k,\ k\in\mathbb{N}_m.
\end{equation*}
This together with \eqref{dual-element-B^*} leads to \eqref{equation-c-ufuc}. That is, $\hat{f}$ having the form \eqref{rt-ufuc-solve}
is a solution of the minimum norm interpolation \eqref{mni} with $\mathbf{y}$ if and only if $\mathbf{c}\in\mathbb{R}^m$ is a solution of \eqref{equation-c-ufuc}.
\end{proof}

Observing from above results, a solution of the minimum norm interpolation problem in a Banach space having a smooth pre-dual can be represented by a finite number of functionals $\nu_j$, whose coefficients can be obtained by solving a system of equations. The resulting systems of equations are generally {\it nonlinear} unless $\mathcal{B}$ is a Hilbert space.
In particular, in the Banach space defined by the semi-inner-product, equations
\eqref{equation-c-ufuc} are truly nonlinear due to the nonlinearity of the semi-inner-product with respect to the second variable. Similar to the Hilbert space case, the infinite dimensional component of this case lies in the computation of the semi-inner-product of two elements of the infinite dimensional Banach space.

\subsection{Solutions of Minimum Norm Interpolation in a Banach Space with a Non-Smooth Pre-Dual Space}

We consider in this subsection solving the minimum norm interpolation problem \eqref{mni} in a Banach space having a {\it non-smooth} pre-dual space by making use the representer theorem obtained in section 3. In this case, we do not assume that the pre-dual space is smooth. 

The solution methods presented in this subsection is mainly a continuation of Theorem \ref{representer-theorem-subdifferential-predual}. Recall that Theorem \ref{representer-theorem-subdifferential-predual} provides a characterization of a solution of  the minimum norm interpolation problem \eqref{mni} in the case when a Banach space $\mathcal{B}$ has the non-smooth pre-dual space $\mathcal{B}_{*}$. However, the theorem does not furnish a way to obtain the $m$ coefficients $c_j$ involved in the solution representation. Our task is to show that the coefficients $c_j$ can, in deed, be obtained by solving an optimization problem in $\mathbb{R}^m$. To this end, we introduce the finite dimensional minimization problem with $\mathbf{y}\in\mathbb{R}^m\setminus\{0\}$ as
\begin{equation}\label{equivalent-minimization-predual}
\inf\left\{\|\mathcal{L}^*(\mathbf{c})\|_{\mathcal{B}_*}:
\langle\mathbf{c},\mathbf{y}\rangle_{\mathbb{R}^m}=1,
\mathbf{c}\in\mathbb{R}^m\right\}.
\end{equation}
Note that minimization problem \eqref{equivalent-minimization-predual} is a somewhat twisted version of the compressed sensing problem \cite{CRT, D}.

We begin with characterizing the solutions of \eqref{equivalent-minimization-predual} by standard arguments in convex analysis. For any $\mathbf{c}\in\mathbb{R}^m$ we set
\begin{equation}\label{Definitionofphi}
\phi(\mathbf{c}):=\|\mathcal{L}^*(\mathbf{c})\|_{\mathcal{B}_*}
\end{equation}
and
\begin{equation}\label{Definitionofpsi}
\psi(\mathbf{c}):=\langle\mathbf{c},\mathbf{y}\rangle_{\mathbb{R}^m}-1.
\end{equation}
We also need to describe the chain rule of the subdifferential \cite{Showalter}. Let $\mathcal{B}_1$ and $\mathcal{B}_2$ be two real Banach spaces. Supposet that $\varphi:\mathcal{B}_2\to \mathbb{R}\cup\{+\infty\}$ is a convex function and $\mathcal{T}:\mathcal{B}_1\to \mathcal{B}_2$ is a bounded linear operator. If $\varphi$ is continuous at some point of the range of $\mathcal{T}$, then for all $f\in \mathcal{B}_1$
\begin{equation}\label{chain-rule}
\partial(\varphi\circ\mathcal{T})(f)=\mathcal{T}^*\partial\varphi(\mathcal{T}(f)).
\end{equation}

The solutions of the minimization problem \eqref{equivalent-minimization-predual} can be first characterized by the Lagrange multiplier method, stated in  Lemma \ref{lagrange-multiplier}.

\begin{prop}\label{Lagrange-multiplier-equ-mni}
Suppose that $\mathcal{B}$ is a Banach space having the pre-dual space $\mathcal{B}_{*}$ and $\nu_j\in\mathcal{B}_*$, $j\in\mathbb{N}_m$, are linearly independent. Let $\mathbf{y}\in\mathbb{R}^m\setminus\{0\}$, $\mathcal{L}$ be defined by \eqref{functional-operator} and $\mathcal{L}^*$ be the adjoint operator. Then $\hat{\mathbf{c}}\in\mathbb{R}^m$ is a solution of the minimization problem \eqref{equivalent-minimization-predual} with $\mathbf{y}$ if and only if $\langle\hat{\mathbf{c}},\mathbf{y}\rangle_{\mathbb{R}^m}=1$, and there exist $\lambda\in\mathbb{R}$ and
\begin{equation}\label{Lag-mul-1}
f\in\partial\|\cdot\|_{\mathcal{B}_{*}}
(\mathcal{L}^*(\hat{\mathbf{c}}))
\end{equation}
such that
\begin{equation}\label{Lag-mul-2}
\mathcal{L}(f)=\lambda \mathbf{y}.
\end{equation}
\end{prop}
\begin{proof}
As a composition of the linear function $\mathcal{L}^*(\cdot)$ and the convex function $\|\cdot\|_{\mathcal{B}_*}$, $\phi$ defined by \eqref{Definitionofphi} is convex on $\mathbb{R}^m$. Moreover, it is easy to see the convexity and the continuity of the function $\psi$ defined by \eqref{Definitionofpsi}. By Lemma \ref{lagrange-multiplier}, $\hat{\mathbf{c}}$ is a solution of the optimization problem \eqref{equivalent-minimization-predual} with $\mathbf{y}$ if and only if $\langle\hat{\mathbf{c}},\mathbf{y}\rangle_{\mathbb{R}^m}=1$ and there exists $\eta\in\mathbb{R}$ such that
$$
0\in\partial \phi(\hat{\mathbf{c}})+\eta\partial\psi(\hat{\mathbf{c}}).
$$
Note that $\phi=\|\cdot\|_{\mathcal{B}_{*}}\circ\mathcal{L}^*$, $\mathcal{L}^*:\mathbb{R}^m\rightarrow\mathcal{B}_*$ is a bounded linear operator and the norm $\|\cdot\|_{\mathcal{B}_{*}}$ is continuous on $\mathcal{B}_*$. Then by the chain rule \eqref{chain-rule} of subdifferentials, we have that
\begin{equation}\label{subdiff-phi}
\partial \phi(\mathbf{c})
=\mathcal{L}\partial\|\cdot\|_{\mathcal{B}_{*}}(\mathcal{L}^*\hat{\mathbf{c}}).
\end{equation}
Since $\psi$ is linear, there holds
\begin{equation}\label{subdiff-psi}
\partial(\psi)(\hat{\mathbf{c}})=\mathbf{y}.
\end{equation}
Combining \eqref{subdiff-phi} with \eqref{subdiff-psi}, we get that
\begin{equation*}\label{subdiff-phi-psi-formula}
\partial \phi(\hat{\mathbf{c}})+\eta\partial\psi(\hat{\mathbf{c}})
=\mathcal{L}\partial\|\cdot\|_{\mathcal{B}_{*}}
(\mathcal{L}^*(\hat{\mathbf{c}}))+\eta\mathbf{y}.
\end{equation*}
It follows that $\hat{\mathbf{c}}$ is a solution of \eqref{equivalent-minimization-predual} if and only if $\langle\hat{\mathbf{c}},\mathbf{y}\rangle_{\mathbb{R}^m}=1$ and there exist $\eta\in\mathbb{R}$ and $f$ satisfying \eqref{Lag-mul-1} such that
$$
\mathcal{L}(f)+\eta\mathbf{y}=0.
$$
By setting $\lambda=-\eta$, we get the desired conclusion.
\end{proof}

We next present an alternative characterization of solutions of the minimization problem \eqref{equivalent-minimization-predual}, which will be used to reveal the relation of solutions of minimization problems \eqref{mni} and \eqref{equivalent-minimization-predual}.

\begin{prop}\label{minimization-for-c}
Suppose that $\mathcal{B}$ is a Banach space having the pre-dual space $\mathcal{B}_{*}$ and $\nu_j\in\mathcal{B}_*$, $j\in\mathbb{N}_m$, are linearly independent. Let $\mathbf{y}\in\mathbb{R}^m\setminus\{0\}$, $\mathcal{L}$ be defined by \eqref{functional-operator}, $\mathcal{L}^*$ be the adjoint operator and $\mathcal{M}_{\mathbf{y}}$ be defined by \eqref{My}. Then $\hat{\mathbf{c}}\in\mathbb{R}^m$ is a solution of the minimization problem \eqref{equivalent-minimization-predual} with $\mathbf{y}$ if and only if
\begin{equation}\label{relationfc}
\frac{1}{\|\mathcal{L}^*(\hat{\mathbf{c}})\|_{\mathcal{B}_{*}}}
\partial\|\cdot\|_{\mathcal{B}_{*}}(\mathcal{L}^*(\hat{\mathbf{c}}))
\cap\mathcal{M}_{\mathbf{y}}\neq\emptyset.
\end{equation}
\end{prop}
\begin{proof}
We first suppose that $\hat{\mathbf{c}}\in\mathbb{R}^m$ is a solution of the minimization problem \eqref{equivalent-minimization-predual} with $\mathbf{y}$. Proposition \ref{Lagrange-multiplier-equ-mni} ensures that $\langle\hat{\mathbf{c}},\mathbf{y}\rangle_{\mathbb{R}^m}=1$ and there exists $\lambda\in\mathbb{R}$ and exists $f$ satisfying \eqref{Lag-mul-1} such that \eqref{Lag-mul-2} holds. It follows from \eqref{Lag-mul-1} that
$$
\langle\mathcal{L}^*(\hat{\mathbf{c}}),f\rangle_{\mathcal{B}}
=\|\mathcal{L}^*(\hat{\mathbf{c}})\|_{\mathcal{B}_{*}},
$$
which further yields that
$$
\langle\hat{\mathbf{c}},\mathcal{L}(f)\rangle_{\mathbb{R}^m}
=\|\mathcal{L}^*(\hat{\mathbf{c}})\|_{\mathcal{B}_{*}}.
$$
Substituting \eqref{Lag-mul-2} into the above equation, we have that
$$
\lambda\langle\hat{\mathbf{c}},\mathbf{y}\rangle_{\mathbb{R}^m}
=\|\mathcal{L}^*(\hat{\mathbf{c}})\|_{\mathcal{B}_{*}}.
$$
This together with $\langle\hat{\mathbf{c}},\mathbf{y}\rangle_{\mathbb{R}^m}=1$ leads to
$\lambda=\|\mathcal{L}^*(\hat{\mathbf{c}})\|_{\mathcal{B}_{*}}.$ Set $\hat{f}:=\frac{1}{\lambda}f.$ We will show that $\hat{f}$ belongs to the intersection in the left side hand of equation \eqref{relationfc}. Combining inclusion \eqref{Lag-mul-1} with the definition of $\hat{f}$, we get that
\begin{equation}\label{relationfc1}
\hat{f}\in\frac{1}{\|\mathcal{L}^*(\hat{\mathbf{c}})\|_{\mathcal{B}_{*}}}
\partial\|\cdot\|_{\mathcal{B}_{*}}(\mathcal{L}^*(\hat{\mathbf{c}})).
\end{equation}
Moreover, equation \eqref{Lag-mul-2} leads directly to the interpolation condition
$\mathcal{L}(\hat{f})=\mathbf{y}$. That is, $\hat{f}\in\mathcal{M}_{\mathbf{y}}$. Consequently,
we conclude that $\hat{f}$ belongs to the intersection in the left hand side of equation \eqref{relationfc}, which leads to the validity of \eqref{relationfc}.

Conversely, we suppose that \eqref{relationfc} holds. That is, there exists $\hat{f}\in\mathcal{B}$ satisfying $\hat{f}\in\mathcal{M}_{\mathbf{y}}$ and inclusion \eqref{relationfc1}.
We will prove by employing Proposition \ref{Lagrange-multiplier-equ-mni} that $\hat{\mathbf{c}}$ is a solution of the minimization problem \eqref{equivalent-minimization-predual}. By inclusion \eqref{relationfc1} we get that
$$
\langle\mathcal{L}^*(\hat{\mathbf{c}}),\hat{f}\rangle_{\mathcal{B}}=1,
$$
which yields that
$$
\langle\hat{\mathbf{c}},\mathcal{L}(\hat{f})\rangle_{\mathbb{R}^m}=1.
$$
Substituting the interpolation condition $\mathcal{L}(\hat{f})=\mathbf{y}$ into the above equation, we have that $\langle\hat{\mathbf{c}},\mathbf{y}\rangle_{\mathbb{R}^m}=1$. It suffices to verify that there exists $\lambda\in\mathbb{R}$ and exists $f$ satisfying \eqref{Lag-mul-1} such that \eqref{Lag-mul-2} holds. Set
$$
f:=\|\mathcal{L}^*(\hat{\mathbf{c}})\|_{\mathcal{B}_{*}}\hat{f}
\ \ \mbox{and}\ \
\lambda:=\|\mathcal{L}^*(\hat{\mathbf{c}})\|_{\mathcal{B}_{*}}.
$$
Inclusion \eqref{relationfc1} leads directly to \eqref{Lag-mul-1} and the interpolation condition $\mathcal{L}(\hat{f})=\mathbf{y}$ leads to \eqref{Lag-mul-2}. Hence, by using Proposition \ref{Lagrange-multiplier-equ-mni} we conclude that $\hat{\mathbf{c}}$ is a solution of the minimization problem \eqref{equivalent-minimization-predual} with $\mathbf{y}$.
\end{proof}

We show below that the coefficients $c_j$ appearing in the representer theorem (Theorem \ref{representer-theorem-subdifferential-predual}) can be obtained by solving the finite dimensional minimization problem \eqref{equivalent-minimization-predual}. This yields a complete solution of the minimum norm interpolation \eqref{mni} in a Banach space having the (non-smooth) pre-dual space.

\begin{thm}\label{minimization-for-c-f}
Suppose that $\mathcal{B}$ is a Banach space having the pre-dual space $\mathcal{B}_{*}$ and $\nu_j\in\mathcal{B}_*$, $j\in\mathbb{N}_m$, are linearly independent. Let $\mathbf{y}\in\mathbb{R}^m\setminus\{0\}$, $\mathcal{L}$ be defined by \eqref{functional-operator}, $\mathcal{L}^*$ be the adjoint operator and $\mathcal{M}_{\mathbf{y}}$ be defined by \eqref{My}. Then $\hat{f}\in\mathcal{B}$ is a solution of the minimum norm interpolation \eqref{mni} with $\mathbf{y}$ if and only if
\begin{equation}\label{intersectionoftwosets}
\hat{f}\in \frac{1}{\|\mathcal{L}^*(\hat{\mathbf{c}})\|_{\mathcal{B}_{*}}}
\partial\|\cdot\|_{\mathcal{B}_{*}}(\mathcal{L}^*(\hat{\mathbf{c}}))
\cap\mathcal{M}_{\mathbf{y}}
\end{equation}
for a solution $\hat{\mathbf{c}}$ of the minimization problem \eqref{equivalent-minimization-predual} with $\mathbf{y}$.
\end{thm}

\begin{proof}
We first suppose that $\hat{f}\in\mathcal{B}$ is a solution of the minimum norm interpolation \eqref{mni} with $\mathbf{y}$. Theorem \ref{representer-theorem-subdifferential-predual} ensures that
\begin{equation}\label{f-c}
\hat f\in \|\mathcal{L}^*(\mathbf{c})\|_{\mathcal{B}_{*}}
\partial\|\cdot\|_{\mathcal{B}_{*}}(\mathcal{L}^*(\mathbf{c}))
\cap\mathcal{M}_{\mathbf{y}}
\end{equation}
for some $\mathbf{c}\in\mathbb{R}^m.$ By setting
$$
\hat{\mathbf{c}}:=\frac{\mathbf{c}}{\|\mathcal{L}^*(\mathbf{c})\|_{\mathcal{B}_{*}}^2},
$$
we get that
\begin{equation}\label{c-hat-c1}
\|\mathcal{L}^*(\mathbf{c})\|_{\mathcal{B}_{*}}
=\frac{1}{\|\mathcal{L}^*(\hat{\mathbf{c}})\|_{\mathcal{B}_{*}}}
\end{equation}
and
\begin{equation}\label{c-hat-c2}
\partial\|\cdot\|_{\mathcal{B}_{*}}(\mathcal{L}^*(\mathbf{c}))
=\partial\|\cdot\|_{\mathcal{B}_{*}}(\mathcal{L}^*(\hat{\mathbf{c}})).
\end{equation}
Substituting equations \eqref{c-hat-c1} and \eqref{c-hat-c2} into the right hand side of inclusion \eqref{f-c}, we get the inclusion relation \eqref{intersectionoftwosets}, which further leads to \eqref{relationfc}. Thus, by employing Proposition \ref{minimization-for-c} we conclude that $\hat{\mathbf{c}}$ is a solution of the minimization problem \eqref{equivalent-minimization-predual} with $\mathbf{y}$.

Conversely, we suppose that $\hat{\mathbf{c}}$ is a solution of the minimization problem \eqref{equivalent-minimization-predual} with $\mathbf{y}$. Note by Proposition \ref{minimization-for-c} that \eqref{relationfc} holds. We also suppose that $\hat{f}$ is an element satisfying \eqref{intersectionoftwosets}. We will prove by Theorem \ref{representer-theorem-subdifferential-predual} that $\hat{f}$ is a solution of the minimum norm interpolation problem (\ref{mni}). Set
$$
\mathbf{c}:=\frac{\hat{\mathbf{c}}}{\|\mathcal{L}^*(\hat{\mathbf{c}})\|_{\mathcal{B}_{*}}^2}.
$$
It follows that equations \eqref{c-hat-c1} and \eqref{c-hat-c2} hold. Substituting these two  equations into \eqref{intersectionoftwosets} leads to \eqref{f-c}. Hence, Theorem \ref{representer-theorem-subdifferential-predual} ensures that $\hat{f}$ is a solution of the minimum norm interpolation problem (\ref{mni}).
\end{proof}

Theorem \ref{minimization-for-c-f} provides a scheme for finding a solution of the minimum norm
interpolation problem \eqref{mni}. We now describe the scheme as follows.

Step 1: Solve the finite dimensional optimization problem \eqref{equivalent-minimization-predual} and obtain a solution $\mathbf{c}:=[c_j: j\in\mathbb{N}_m]\in\mathbb{R}^m$.

Step 2: Construct
$$
\nu:=\sum_{j\in\mathbb{N}_m}c_j\nu_j.
$$

Step 3: Find an element $g\in\partial\|\cdot\|_{\mathcal{B}_*}(\nu)$ which satisfies
$$
\mathcal{L}(g)=\frac{\mathbf{y}}{\|\nu\|_{\mathcal{B}_*}}.
$$

Step 4: Obtain a solution of the minimum norm interpolation problem \eqref{mni} by
$$
\hat{f}:=(\|\nu\|_{\mathcal{B}_*})g.
$$

Actual implementation of the above scheme requires further investigation.
Note that although the minimization problem \eqref{equivalent-minimization-predual} in step 1 is of finite dimension, it still involves computation of the norm $\|\cdot\|_{\mathcal{B}_*}$, which is a hidden infinite dimensional component. Moreover, step 3 also involves an infinite dimensional component since it requires solving an infinite dimensional problem. In order to make the above scheme implementable, we have to deal with these hidden infinite dimensional components. One could use approximation approaches to replace the infinite dimensional components by finite dimensional ones. This approach will introduce ``truncation errors'', which we do {\it not} adopt here. Our idea is to make use certain intrinsic properties of these infinite dimensional components to remove their berries, developing {\it equivalent} implementable finite dimensional schemes.

Our approach to be described in section 6 is inspired by a recent result presented in \cite{CX}, where the minimum norm interpolation problem in a special case $\mathcal{B}=\ell_1(\mathbb{N})$ was solved by a different approach.  In this special case, the infinite dimensional components we mentioned above can be removed. This benefits from the characterization of the space $c_0$, the pre-dual space of $\ell_1(\mathbb{N})$. Firstly, the minimization problem \eqref{equivalent-minimization-predual} was reformulated in \cite{CX} as a linear programming problem. Specifically, suppose that $\mathbf{u}_j$, $j\in\mathbb{N}_m,$ are $m$ given linearly independent elements in $c_0$ and the operator $\mathcal{L}:\ell_1(\mathbb{N})\rightarrow\mathbb{R}^m$ is defined by \eqref{functional-operator-l1}. Instead of solving the minimization problem \eqref{equivalent-minimization-predual}, it was proposed in \cite{CX} to solve an equivalent problem
\begin{equation}\label{dual-problem}
\sup\left\{
\frac{\langle\mathbf{c},\mathbf{y}\rangle_{\mathbb{R}^m}}{\|\sum_{j\in\mathbb{N}_m}
c_j\mathbf{u}_j\|_{\infty}}:\mathbf{c}=[c_j:j\in\mathbb{N}_m]\in\mathbb{R}^m
\right\}.
\end{equation}
It was proved there that the unit sphere in $\mathbb{R}^m$ under the norm $\|\cdot\|_{*}$, defined by
$$
\|\mathbf{c}\|_{*}:=\left\|\sum_{j\in\mathbb{N}_m}
c_j\mathbf{u}_j\right\|_{\infty}, \ \ \mathbf{c}\in\mathbb{R}^m,
$$
is the surface of a convex polytope, which are formed by a finite number of planes. Hence, the optimization problem \eqref{dual-problem} is equivalent to a linear programming problem: finding a maximizer of the linear function
$$
g(\mathbf{c}):=\langle\mathbf{c},\mathbf{y}\rangle_{\mathbb{R}^m}, \ \ \mathbf{c}\in\mathbb{R}^m
$$
on the unit sphere
$$
\{\mathbf{c}: \mathbf{c}\in\mathbb{R}^m, \|\mathbf{c}\|_{*}=1\}.
$$
Moreover, Lemma \ref{subdifferentials-c0} shows that finding the subdifferentials of the norm $\|\cdot\|_{\infty}$ of $c_0$ at a nonzero element $\mathbf{u}\in c_0$ is of finite dimension and any vector in the subdifferentials has at most finite many nonzero components. Therefore, the minimum norm interpolation in $\ell_1(\mathbb{N})$  can be obtained by solving a linear system of $m$ coefficients.

Although the minimum norm interpolation problem in $\ell_1(\mathbb{N})$ can be solved as a truly finite dimensional problem, as described above, solving the resulting linear programming problem requires an exponential order ${\cal O}(2^m)$ of computational costs, where $m$ is the number of interpolation conditions used in the problem. When $m$ is large, which is often the case in data science, this method is not feasible. It is desirable to develop alternative representations of solutions of the minimum norm interpolation problem convenient for algorithmic development. Motivated from the success of the fixed-point approach used in machine learning \cite{AMPSX, LMX, Li-Song-Xu2018, Li-Song-Xu2019, PSW}, image processing \cite{CHZ, LMSX, LSXZ, LSXX, MSX}, medical imaging \cite{KLSX, LZKSVSLFX, ZLKSZX} and solutions of inverse problems \cite{FJL, JL}, we will develop representations of a solution of the minimum norm interpolation problem or the regularization problem in a Banach space, as fixed-points of nonlinear maps defined by proximity operators of functions involved in the problem. The fixed-point formulation well fits for designing iterative algorithms. Difficulty of developing implementable algorithms for this problem in a Banach space lies in infinite dimensional components of the problem. This challenge motivates us to develop a finite dimensional fixed-point approach to solve the minimum norm interpolation problem in the special Banach space $\ell_1(\mathbb{N})$ by making use special properties of this space and its pre-dual space. We present this approach in section 6. Extension of this approach to a general Banach space will be a future research topic.


\section{Infimum of Minimum Norm Interpolation}
We present in this section the infimum of the minimum norm interpolation problem in a Banach space.

From the explicit representer theorems and solution representations presented in previous sections, we can readily find the infimum of the minimum norm interpolation in a Banach space. In the next theorem, we identify it with the norm of the functional appearing in the explicit solution representations.

\begin{thm}\label{infimum}
Suppose that $\mathcal{B}$ is a Banach space with the dual space $\mathcal{B}^*$ and $\nu_j\in\mathcal{B}^*$, $j\in\mathbb{N}_m$, are linearly independent and $\mathbf{y}\in\mathbb{R}^m$. If $\hat{f}$ is a solution of the minimum norm interpolation problem \eqref{mni} with $\mathbf{y}$, which has either the form \eqref{rt-subdifferential-explicit} or \eqref{rt-Gateaux-explicit} with the coefficients $c_j\in\mathbb{R}$, $j\in\mathbb{N}_m$, then
$$
\|\hat{f}\|_{\mathcal{B}}
=\left\|\sum_{j\in\mathbb{N}_m}c_j\nu_j\right\|_{\mathcal{B}^*}.
$$
\end{thm}
\begin{proof}
As has been shown in the proof of Theorems \ref{representer-theorem-subdifferential-explicit} and
\ref{representer-theorembyGateaux-explicit}, for the trivial solution $\hat{f}=0$, the coefficients appearing in the solution representations of these theorems are all zeros. Clearly, we have that
$$
\|\hat{f}\|_{\mathcal{B}}
=\left\|\sum_{j\in\mathbb{N}_m}c_j\nu_j\right\|_{\mathcal{B}^*}=0.
$$

It remains to consider the nontrivial solution $\hat{f}\neq 0$. In this case, the function $\nu:=\sum_{j\in\mathbb{N}_m}c_j\nu_j$ is also nonzero. When $\hat{f}$ satisfies the inclusion relation \eqref{rt-subdifferential-explicit}, we get that
$$
\frac{\hat{f}}{\|\nu\|_{_{\mathcal{B}^*}}}
\in\partial\|\cdot\|_{_{\mathcal{B}^*}}(\nu).
$$
Equation \eqref{relation-subdifferential-functionals} ensures that
$$
\frac{\|\hat{f}\|_{\mathcal{B}}}{\|\nu\|_{_{\mathcal{B}^*}}}=1,
$$
that is, $\|\hat{f}\|_{\mathcal{B}}=\|\nu\|_{\mathcal{B}^*}.$ When $\hat{f}$ satisfies the equality \eqref{rt-Gateaux-explicit}, equation \eqref{norm-GateauxDiff} ensures that $\|\hat{f}\|_{\mathcal{B}}=\|\nu\|_{\mathcal{B}^*}.$
\end{proof}

Approaches to determine the coefficients $c_j\in\mathbb{R},$ $j\in\mathbb{N}_m,$ appearing in the solution representations were developed in the last section when the Banach space $\mathcal{B}$ has the pre-dual space $\mathcal{B}_{*}$ and $\nu_j\in\mathcal{B}_*$, for $j\in\mathbb{N}_m$. Accordingly, the infimum of the minimum norm interpolation problem \eqref{mni} can be obtained from the resulting coefficients.

\begin{thm}\label{infimum-smooth-predual}
If $\mathcal{B}$ is a Banach space having the smooth pre-dual space $\mathcal{B}_*$, $\nu_j\in\mathcal{B}_*$, $j\in\mathbb{N}_m$, are linearly independent and $\mathbf{y}\in\mathbb{R}^m$, then the infimum $m_0$ of the minimum norm interpolation problem \eqref{mni} with $\mathbf{y}$ has the form
$$
m_0=\left\|\sum_{j\in\mathbb{N}_m}c_j\nu_j\right\|_{\mathcal{B}_*},
$$
where $c_j\in\mathbb{R}$, $j\in\mathbb{N}_m$, are the solution of the system \eqref{equation-c-smooth-B_*} of equations.
\end{thm}
\begin{proof}
Theorem \ref{solution-mni-smooth-B_*} ensures that if the Banach space $\mathcal{B}$ has the smooth pre-dual space $\mathcal{B}_*$ and $\nu_j\in\mathcal{B}_*$, $j\in\mathbb{N}_m$, then the solution of the minimum norm interpolation problem \eqref{mni} with $\mathbf{y}$ has the form \eqref{rt-Gateaux-predual-solve} with $c_j\in\mathbb{R}$, $j\in\mathbb{N}_m$, satisfying the system \eqref{equation-c-smooth-B_*}. By arguments similar to those used in the proof of Theorem \ref{infimum} and by employing the solution representation \eqref{rt-Gateaux-predual-solve}, we get the desired conclusion of this theorem.
\end{proof}

For the case that the pre-dual space $\mathcal{B}_*$ of the Banach space $\mathcal{B}$ may not be smooth, we have the following representation of the infimum.

\begin{thm}\label{infimum-predual}
If $\mathcal{B}$ is a Banach space having the pre-dual space $\mathcal{B}_{*}$, $\nu_j\in\mathcal{B}_*$, $j\in\mathbb{N}_m$, are linearly independent and $\mathbf{y}\in\mathbb{R}^m\setminus\{0\}$, then the infimum $m_0$ of the minimum norm interpolation problem \eqref{mni} with $\mathbf{y}$ has the form
$$
m_0=\left\|\sum_{j\in\mathbb{N}_m}\hat{c}_j\nu_j\right\|_{\mathcal{B}_*}^{-1},
$$
where $\hat{\mathbf{c}}:=[\hat{c}_j:j\in\mathbb{N}_m]$ is a solution of the minimization problem \eqref{equivalent-minimization-predual} with $\mathbf{y}$.
\end{thm}
\begin{proof}
By employing Theorem \ref{minimization-for-c-f}, we have that a solution of the minimum norm interpolation problem \eqref{mni} with $\mathbf{y}$ has the from
$$
\hat{f}\in \frac{1}{\left\|\sum_{j\in\mathbb{N}_m}\hat{c}_j\nu_j\right\|_{\mathcal{B}_*}}
\partial\|\cdot\|_{\mathcal{B}_{*}}\left(\sum_{j\in\mathbb{N}_m}\hat{c}_j\nu_j\right),
$$
where $\hat{\mathbf{c}}:=[\hat{c}_j:j\in\mathbb{N}_m]$ is a solution of the minimization problem \eqref{equivalent-minimization-predual} with $\mathbf{y}$. By applying arguments similar to those used in the proof of Theorem \ref{infimum} to the inclusion relation above, we obtain the desired result.
\end{proof}

For the special cases discussed in subsection 3.3, infimum results similar to those stated in the above theorems remain valid. That is, in all the cases considered there, the infimum of the minimum norm interpolation problem is equal to the norm of the functional appearing in each corresponding explicit solution representation. We state these results below.


\begin{rem}
If $\mathcal{B}$ is a uniformly Fr\'{e}chet smooth and uniformly convex Banach space with the dual space $\mathcal{B}^*$, then the infimum $m_0$ of the minimum norm interpolation problem \eqref{mni} with $\mathbf{y}$ has the form
$$
m_0=\left\|\sum_{j\in\mathbb{N}_m}c_jg_j^{\sharp}\right\|_{\mathcal{B}^*},
$$
where $c_j\in\mathbb{R}$, $j\in\mathbb{N}_m$, are the solution of the system \eqref{equation-c-ufuc} of equations.
\end{rem}

\begin{rem} If $\mathcal{B}$ is a semi-inner-product RKBS with the semi-inner-product reproducing kernel $G$ and $\mathcal{B}^*$ is its dual space, then the infimum $m_0$ of the minimum norm interpolation problem \eqref{mni} with $\mathbf{y}$ has the form
$$
m_0=\left\|\sum_{j\in\mathbb{N}_m}c_jG(x_j,\cdot)^{\sharp}\right\|_{\mathcal{B}^*},
$$
where $c_j\in\mathbb{R}$, $j\in\mathbb{N}_m$, are the solution of the system \eqref{equation-c-ufuc} of equations with $g_k:=G(x_k,\cdot)$, $k\in\mathbb{N}_m$.
\end{rem}

\begin{rem}
Suppose that $\mathcal{B}$ is a right-sided RKBS with the right-sided reproducing kernel $K$ and $\mathcal{B}^*$ is its dual space. If $\mathcal{B}$ is reflexive and strictly convex, then the infimum $m_0$ of the minimum norm interpolation problem \eqref{mni} with $\mathbf{y}$ has the form
$$
m_0=\left\|\sum_{j\in\mathbb{N}_m}c_jK(x_j,\cdot)\right\|_{\mathcal{B}^*},
$$
where $c_j\in\mathbb{R}$, $j\in\mathbb{N}_m$, are the solution of the system \eqref{equation-c-smooth-B_*} of equations with $\nu_j:=K(x_j,\cdot),$ $j\in\mathbb{N}_m$.
\end{rem}

\begin{rem}
The infimum of the minimum norm interpolation problem \eqref{mni} in $\ell_1(\mathbb{N})$ has the form
$$
m_0=\left\|\sum_{j\in\mathbb{N}_m}\hat{c}_j\mathbf{u}_j\right\|_{\infty}^{-1},
$$
where $\hat{\mathbf{c}}:=[\hat{c}_j:j\in\mathbb{N}_m]$ is a solution of the minimization problem \eqref{equivalent-minimization-predual} with $\mathbf{y}$ and $\mathcal{B}_*=c_0$.
\end{rem}

\section{Fixed-Point Approach for Minimum Norm Interpolation}

The solution method established in section 4 for the minimum norm interpolation in a Banach space with non-smooth pre-dual space provides a foundation for further development of implementable schemes to find its solution by determining the coefficients $c_j$ which appear in the solution representations. Specifically, using the solution representation described in Theorem \ref{minimization-for-c-f} to find a solution of the problem \eqref{mni} requires to solve the minimization problem \eqref{equivalent-minimization-predual} and to verify \eqref{intersectionoftwosets}. Both of these steps involve solving inclusion relations. It is not computationally convenient to solve an inclusion relation, especially, when the set involved in the inclusion is described by sophisticated equations and/or inequalities. It requires further investigation to develop computationally convenient schemes based on the theoretical results that we have obtained.

In this section, we will take a different point of view to develop a fixed-point approach for the minimum norm interpolation problem \eqref{mni} in a Banach space. Specifically, we reformulate problem \eqref{mni} as an unconstrained minimization problem, and then re-express its solution as a fixed-point of a nonlinear map defined via the proximity operator of functions involved in the problem. The resulting fixed-point equations can be solved conveniently by iteration schemes. The reformulation will be done by using the fact that an inclusion involving subdifferential of a convex function can be converted to a fixed-point equation defined by the proximity operator of the function. The fixed-point formulation provides a sound basis for algorithmic development for numerical solutions of the problem. In particular, when the Banach space $\mathcal{B}$ is the special space $\ell_1(\mathbb{N})$, we develop an implementable fixed-point equation for finding a solution of this problem.

\subsection{Fixed-Point Approach for Minimum Norm Interpolation in a General Banach Space}
We formulate in this subsection fixed-point equations for a solution of the minimum norm interpolation problem \eqref{mni} in a general Banach space.

We first reformulate the minimum norm interpolation problem \eqref{mni} as an equivalent unconstrained minimization problem. Suppose that $\mathcal{B}$ is a real Banach space with the dual space $\mathcal{B}^*$ and $\nu_j\in\mathcal{B}^*$, $j\in\mathbb{N}_m$, are linearly independent. Let $\mathcal{L}$ be defined by \eqref{functional-operator} and $\mathcal{L}^*$ its adjoint operator. For a given data $\mathbf{y}\in\mathbb{R}^m$, we define the indicator function $\iota_{\mathbf{y}}:\mathbb{R}^m\rightarrow \mathbb{R}\cup\{+\infty\}$ of $\mathbf{y}$ at $\mathbf{c}\in\mathbb{R}^m$ as
\begin{equation}\label{indicator}
\iota_{\mathbf{y}}(\mathbf{c})
:=\left\{\begin{array}{rcl}
0,& \mbox{if}\ \ \mathbf{c}=\mathbf{y},\\
+\infty, & \mbox{if}\ \ \mathbf{c}\neq \mathbf{y}.
\end{array}\right.
\end{equation}
Note that the indicator function $\iota_{\mathbf{y}}$ is convex but not continuous at $\mathbf{c}=\mathbf{y}$. Using the indicator function, the minimum norm interpolation problem \eqref{mni}, which is a constrained minimization problem, is rewritten as an equivalent unconstrained one. We state this result in the next lemma for convenient reference.

\begin{lemma}\label{mni-unconstrain}
If for a given $\mathbf{y}\in\mathbb{R}^m$, the indicator function $\iota_{\mathbf{y}}$ is defined by \eqref{indicator}, then the minimum norm interpolation problem \eqref{mni} with $\mathbf{y}$ is equivalent to the unconstrained minimization problem
\begin{equation}\label{equi-min-unconstrain}
\inf\{\|f\|_{\mathcal{B}}
+\iota_{\mathbf{y}}(\mathcal{L}(f)):f\in\mathcal{B}\}.
\end{equation}
\end{lemma}
\begin{proof}
By the definition of the indicator function, the infimum in \eqref{equi-min-unconstrain} will be assumed at an element $f\in\mathcal{B}$ such that $\mathcal{L}(f)=\mathbf{y}$. Thus, the minimization problem \eqref{equi-min-unconstrain} can be rewritten as
$$
\inf\{\|f\|_{\mathcal{B}}
:f\in\mathcal{B},\mathcal{L}(f)=\mathbf{y}\},
$$
which coincides with the minimum norm interpolation problem \eqref{mni} with $\mathbf{y}$.
\end{proof}

We characterize a solution of \eqref{equi-min-unconstrain} in terms of fixed-point equations. To this end, we need the notion of the proximity operator in both spaces $\mathbb{R}^m$ and $\mathcal{B}$. We begin with reviewing the proximity operator on $\mathbb{R}^m$ which was originally introduced in \cite{Mo}. Let $\psi:\mathbb{R}^m\rightarrow\mathbb{R}\cup\{+\infty\}$ be a convex function such that
$$
\mathrm{dom}(\psi):=\{\mathbf{c}\in\mathbb{R}^m:\psi(\mathbf{c})<+\infty\}\neq\emptyset.
$$
The proximity operator $\mathrm{prox}_{\psi}: \mathbb{R}^m\to \mathbb{R}^m$ of $\psi$ is defined for $\mathbf{a}\in\mathbb{R}^m$ by
\begin{equation}\label{prox-Rm}
\mathrm{prox}_{\psi}(\mathbf{a})
:=\arg\inf\left\{\frac12\|\mathbf{a}-\mathbf{c}\|_{\mathbb{R}^m}^2
+\psi(\mathbf{c}):\mathbf{c}\in\mathbb{R}^m\right\}.
\end{equation}

The following relation between the proximity operator of $\psi$ and its subdifferential can be found in \cite{BC, MSX}.

\begin{lemma}\label{subdiff-prox-Rm}
If $\psi$ is a convex function from $\mathbb{R}^m$ to $\mathbb{R}\cup\{+\infty\}$ and $\mathbf{a}\in\mathrm{dom}(\psi)$, then
\begin{equation}\label{subdiff-prox-Rm-formula}
\mathbf{b}\in\partial\psi(\mathbf{a})
\ \ \mbox{if and only if}\ \
\mathbf{a}={\rm prox}_{\psi}(\mathbf{a}+\mathbf{b}).
\end{equation}
\end{lemma}

The proximity operator of a convex function in an infinite dimensional Hilbert space may be found in \cite{BC}. We now define the proximity operator of a convex function in a Banach space $\mathcal{B}$. This requires the availability of a Hilbert space and a linear map between it and the Banach space $\mathcal{B}$. Suppose that $\mathcal{H}$ is a Hilbert space, $\mathcal{T}$ is a bounded linear operator from $\mathcal{B}$ to $\mathcal{H}$ and $\mathcal{T}^*$ is its adjoint operator from $\mathcal{H}$ to $\mathcal{B}^*$. The proximity operator $\mathrm{prox}_{\psi,\mathcal{H},\mathcal{T}}: \mathcal{B} \to \mathcal{B}$ of a convex function $\psi:\mathcal{B}\to\mathbb{R}\cup\{+\infty\}$ with respect to $\mathcal{H}$ and $\mathcal{T}$ is defined by
\begin{equation}\label{prox-predual}
\mathrm{prox}_{\psi,\mathcal{H},\mathcal{T}}(f)
:=\arg\inf\left\{\frac12\|\mathcal{T}(f-h)\|_{\mathcal{H}}^2
+\psi(h):h\in\mathcal{B}\right\},\ \ \mbox{for all}\ \ f\in\mathcal{B}.
\end{equation}
The proximity operator $\mathrm{prox}_{\psi}$ defined by \eqref{prox-Rm} is a special case of the definition \eqref{prox-predual}. Specifically, let $\mathcal{B}$ be the Euclidean space $\mathbb{R}^m$ with a norm $\|\cdot\|$. If we choose $\mathcal{H}:=\mathbb{R}^m$ with the Euclidean norm $\|\cdot\|_{\mathbb{R}^m}$ and $\mathcal{T}$ as the identity operator from $\mathbb{R}^m$ with the norm $\|\cdot\|$ to $\mathbb{R}^m$ with the Euclidean norm $\|\cdot\|_{\mathbb{R}^m}$, the proximity operator $\mathrm{prox}_{\psi,\mathcal{H},\mathcal{T}}$ reduces to $\mathrm{prox}_{\psi}$.

In a manner similar to Lemma \ref{subdiff-prox-Rm}, the proximity operator defined by \eqref{prox-predual} of a convex function $\psi$ defined on $\mathcal{B}$ is intimately related to the subdifferential of $\psi$.

\begin{prop}\label{subdiff-prox-Banach}
Suppose that $\mathcal{B}$ is a Banach space with the dual space $\mathcal{B}^*$ and $\mathcal{H}$ is a Hilbert space. Let $\mathcal{T}$ be a bounded linear operator from $\mathcal{B}$ to $\mathcal{H}$ and $\mathcal{T}^*$ be its adjoint operator from $\mathcal{H}$ to $\mathcal{B}^*$. If $\psi:\mathcal{B}\to\mathbb{R}\cup\{+\infty\}$ is a convex function, then for all $f\in{\rm dom}(\psi)$ and $g\in\mathcal{B}$
\begin{equation}\label{subdiff-prox-Banach-formula}
\mathcal{T}^*\mathcal{T}(g)\in\partial\psi(f)
\ \ \mbox{if and only if}\ \
f=\mathrm{prox}_{\psi,\mathcal{H},\mathcal{T}}(f+g).
\end{equation}
\end{prop}
\begin{proof}
By definition \eqref{prox-predual} of the proximity operator on the Banach space $\mathcal{B}$, for each $f\in{\rm dom}(\psi)$ and $g\in\mathcal{B}$ the equation
$$
f=\mathrm{prox}_{\psi,\mathcal{H},\mathcal{T}}(f+g)
$$
is equivalent to
$$
f=\arg\inf\left\{\frac12\|\mathcal{T}(f+g-h)\|_{\mathcal{H}}^2
+\psi(h):h\in\mathcal{B}\right\}.
$$
According to the Fermat rule \cite{Zalinescu}, the above equation holds if and only if
$$
0\in\partial\left(\frac12\|\mathcal{T}(\cdot-f-g)\|_{\mathcal{H}}^2
+\psi(\cdot)\right)(f).
$$
By employing the chain rule \eqref{chain-rule} of the subdifferential and noting that the subdifferential of the function $\|\cdot\|_{\mathcal{H}}^2$ at any element in the Hilbert space $\mathcal{H}$ is a singleton, the inclusion relation above is thus equivalent to
\begin{equation}\label{relation-sub-prox}
0\in\mathcal{T}^*\mathcal{T}(f-f-g)+\partial\psi(f).
\end{equation}
The inclusion relation \eqref{relation-sub-prox} is further equivalent to
$$
\mathcal{T}^*\mathcal{T}(g)\in\partial\psi(f),
$$
proving the desired result.
\end{proof}

Proposition \ref{subdiff-prox-Banach} is a generalization of Lemma \ref{subdiff-prox-Rm}. We explain this point below. Let $\mathcal{B}$ be the Euclidean space $\mathbb{R}^m$ with a norm $\|\cdot\|$. For the norm $\|\cdot\|$, its dual norm $\|\cdot\|_{\sharp}$ is defined, for all
$\mathbf{b}\in\mathbb{R}^m$ by
$$
\|\mathbf{b}\|_{\sharp}:=\max\{|\langle\mathbf{a},
\mathbf{b}\rangle_{\mathbb{R}^m}|:\|\mathbf{a}\|=1,
\mathbf{a}\in\mathbb{R}^m\}.
$$
Accordingly, the dual space $\mathcal{B}^*$ is identified with $\mathbb{R}^m$ with the dual norm $\|\cdot\|_{\sharp}$. Choose the Hilbert space $\mathcal{H}$ as $\mathbb{R}^m$ with the Euclidean norm $\|\cdot\|_{\mathbb{R}^m}$ and the operator $\mathcal{T}$ as the identity operator from $\mathcal{B}$ to $\mathcal{H}$. Clearly, the adjoint operator $\mathcal{T}^*$ is the identity operator from $\mathcal{H}$ to $\mathcal{B}^*$. Hence, $\mathcal{T}^*\mathcal{T}$ coincides with the identity operator from $\mathcal{B}$ to $\mathcal{B}^*$. In this special case, relation \eqref{subdiff-prox-Banach-formula} in Proposition \ref{subdiff-prox-Banach} reduces to relation \eqref{subdiff-prox-Rm-formula} in Lemma \ref{subdiff-prox-Rm}.

We also need the notion of the conjugate function to develop the characterization for the solution of the minimization problem \eqref{equi-min-unconstrain} in terms of fixed-point equations. The conjugate function of a convex function $\psi:\mathbb{R}^m\rightarrow\mathbb{R}\cup\{+\infty\}$ is defined as
$$
\psi^*(\mathbf{c}):=\sup\{\langle\mathbf{a},\mathbf{c}\rangle_{\mathbb{R}^m}
-\psi(\mathbf{a}):\mathbf{a}\in\mathbb{R}^m\},\ \ \mbox{for all}\ \ \mathbf{c}\in\mathbb{R}^m.
$$
There is a relation between the subdifferential of a convex function and that of its conjugate function. Specifically, if $\psi$ is a convex function on $\mathbb{R}^m$, then for all $\mathbf{a}\in \mathrm{dom}(\psi)$ and all $\mathbf{b}\in\mathrm{dom}(\psi^*)$ there holds
\begin{equation}\label{conjugate}
\mathbf{a}\in\partial\psi^*(\mathbf{b})\ \ \mbox{if and only if}\ \ \mathbf{b}\in\partial\psi(\mathbf{a}).
\end{equation}
This leads to the relation between the proximity operators of $\psi$ and $\psi^*$, that is,
\begin{equation}\label{prox-psi-psi*}
\mathrm{prox}_{\psi}=\mathbf{I}-\mathrm{prox}_{\psi^*},
\end{equation}
where $\mathbf{I}$ denotes the $m\times m$ identity matrix. As an example, the conjugate fucntion $\iota_{\mathbf{y}}^*$ of the indicator function $\iota_{\mathbf{y}}$ has the form \begin{equation}\label{indicator-conjugate}
\iota_{\mathbf{y}}^*(\mathbf{c}):=\langle\mathbf{y},\mathbf{c}\rangle_{\mathbb{R}^m},
\ \ \mbox{for all}\ \ \mathbf{c}\in\mathbb{R}^m.
\end{equation}

The minimization problem \eqref{equi-min-unconstrain} involves the composition of the indicator function $\iota_{\mathbf{y}}$ and the linear operator $\mathcal{L}$. We need to compute the subdifferential of the composition function by the chain rule \eqref{chain-rule} of the subdifferential. However, the pair $\iota_{\mathbf{y}}$ and $\mathcal{L}$ does not satisfy the hypothesis of the chain rule \eqref{chain-rule} of the subdifferential since $\iota_{\mathbf{y}}$ is not continuous at every point in the range of $\mathcal{L}$. Thus, we cannot use the chain rule \eqref{chain-rule} directly. In the next lemma, we verify that the chain rule for the subdifferential of the composition of these two functions remains valid by using special property of the indicator function $\iota_{\mathbf{y}}$.

\begin{lemma}\label{chain-rule-indicator}
Suppose that $\mathcal{B}$ is a Banach space with the dual space $\mathcal{B}^*$ and $\nu_j\in\mathcal{B}^*$, $j\in\mathbb{N}_m$, are linearly independent. Let $\mathcal{L}$ be defined by \eqref{functional-operator} and $\mathcal{L}^*$ be the adjoint operator. If for a given $\mathbf{y}\in\mathbb{R}^m$, $\mathcal{M}_y$ is defined by \eqref{My} and the indicator function $\iota_{\mathbf{y}}$ is defined by \eqref{indicator}, then for all $f\in\mathcal{M}_y$
\begin{equation}\label{iota-L-composition}
\partial(\iota_{\mathbf{y}}\circ\mathcal{L})(f)
=\mathcal{L}^*\partial\iota_{\mathbf{y}}(\mathcal{L}(f)).
\end{equation}
\end{lemma}
\begin{proof}
Let $f\in\mathcal{M}_{\mathbf{y}}$. By definition \eqref{subdifferential} of the subdifferential, we have that
$$
\nu\in\partial(\iota_{\mathbf{y}}\circ\mathcal{L})(f)
$$
if and only if
\begin{equation}\label{indicator-function1}
\iota_{\mathbf{y}}(\mathcal{L}(g))-\iota_{\mathbf{y}}(\mathcal{L}(f))
\geq\langle\nu,g-f\rangle_{\mathcal{B}},
 \ \ \mbox{for all}\ \ g\in\mathcal{B}.
\end{equation}
By the definition of the indicator function $\iota_{\mathbf{y}}$, we observe that
$$
\iota_{\mathbf{y}}(\mathcal{L}(f))=0
\ \ \mbox{and}\ \
\iota_{\mathbf{y}}(\mathcal{L}(g))=0,
\ \ \mbox{for all}\ \ g \in \mathcal{M}_{\mathbf{y}}.
$$
Thus, condition \eqref{indicator-function1} is equivalent to
\begin{equation}\label{indicator-function2}
\langle\nu,g-f\rangle_{\mathcal{B}}\leq 0,
 \ \ \mbox{for all}\ \ g\in\mathcal{M}_{\mathbf{y}}.
\end{equation}
The relation of $\mathcal{M}_{\mathbf{y}}$ and $\mathcal{M}_{0}$ ensures that  condition \eqref{indicator-function2} is equivalent to
\begin{equation}\label{indicator-function3}
\langle\nu,h\rangle_{\mathcal{B}}\leq 0,
 \ \ \mbox{for all}\ \ h\in\mathcal{M}_{0}.
\end{equation}
Since $\mathcal{M}_{0}$ is a subspace of $\mathcal{B}$, we can rewrite the inequalities \eqref{indicator-function3} in their equivalent form
$$
\langle\nu,h\rangle_{\mathcal{B}}=0,
 \ \ \mbox{for all}\ \ h\in\mathcal{M}_{0}.
$$
That is, $\nu\in\mathcal{M}_{0}^{\perp}$, which guaranteed by Lemma \ref{finite} is equivalent to
$$
\nu\in\span\{\nu_j:j\in\mathbb{N}_m\}.
$$
Therefore, we conclude that
\begin{equation}\label{iota-L-composition1}
\partial(\iota_{\mathbf{y}}\circ\mathcal{L})(f)
=\span\{\nu_j:j\in\mathbb{N}_m\}.
\end{equation}

On the other hand, clearly, we have that
\begin{equation}\label{indicator-function4}
\partial\iota_{\mathbf{y}}(\mathbf{y})=\mathbb{R}^m.
\end{equation}
Using \eqref{indicator-function4}, for $\mathcal{L}(f)=\mathbf{y}$  we get that
$$
\mathcal{L}^*\partial\iota_{\mathbf{y}}(\mathcal{L}(f))=\mathcal{L}^*\partial\iota_{\mathbf{y}}(\mathbf{y})=\mathcal{L}^*(\mathbb{R}^m).
$$
By the representation \eqref{adjoint-operator} of $\mathcal{L}^*$, we conclude from the above equation that
\begin{equation}\label{iota-L-composition2}
\mathcal{L}^*\partial\iota_{\mathbf{y}}(\mathcal{L}(f))=\span\{\nu_j:j\in\mathbb{N}_m\}.
\end{equation}
Combining equations \eqref{iota-L-composition1} and \eqref{iota-L-composition2}, we obtain the desired chain rule \eqref{iota-L-composition}.
\end{proof}

With the help of the chain rule \eqref{iota-L-composition}, we turn to characterizing the solution of the minimization problem \eqref{equi-min-unconstrain} as a fixed-point of a nonlinear map defined via the proximity operators of the norm $\|\cdot\|_{\mathcal{B}}$ and the conjugate function $\iota_{\mathbf{y}}^*$.
For convenience, we set
\begin{equation}\label{linear-span-nu-j}
\mathcal{V}:=\span\{\nu_j:j\in\mathbb{N}_m\}.
\end{equation}

\begin{thm}\label{characterize-prox-f}
Suppose that $\mathcal{B}$ is a Banach space with the dual space $\mathcal{B}^*$ and $\nu_j\in\mathcal{B}^*$, $j\in\mathbb{N}_m$, are linearly independent and $\mathbf{y}\in\mathbb{R}^m$. Let $\mathcal{L}$, $\iota_{\mathbf{y}}$  and $\mathcal{V}$ be defined by \eqref{functional-operator}, \eqref{indicator} and \eqref{linear-span-nu-j}, respectively, and $\mathcal{L}^*$ be the adjoint operator of $\mathcal{L}$. Let $\mathcal{H}$ be a Hilbert space and $\mathcal{T}$ a bounded linear operator from $\mathcal{B}$ to $\mathcal{H}$ with the adjoint operator $\mathcal{T}^*$. If $\mathcal{T}^*\mathcal{T}$ is a one-to-one mapping from the inverse image of $\mathcal{V}$ onto $\mathcal{V}$, then $\hat{f}\in\mathcal{B}$ is a solution of the minimization problem \eqref{equi-min-unconstrain} with $\mathbf{y}$ if and only if there exists $\mathbf{c}\in\mathbb{R}^m$ such that
\begin{equation}\label{characterize-prox-f-formula1}
\mathbf{c}=\mathrm{prox}_{\iota_{\mathbf{y}}^*}
(\mathbf{c}+\mathcal{L}(\hat{f}))
\end{equation}
and
\begin{equation}\label{characterize-prox-f-formula2}
\hat{f}=\mathrm{prox}_{\|\cdot\|_{\mathcal{B}},\mathcal{H},\mathcal{T}}
\left(\hat{f}-(\mathcal{T}^*\mathcal{T})^{-1}\mathcal{L}^*(\mathbf{c})\right).
\end{equation}
\end{thm}
\begin{proof}
By using the Fermat rule together with the chain rule \eqref{chain-rule-indicator} for the subdifferential of the composition function $\iota_{\mathbf{y}}\circ\mathcal{L}$, we see that $\hat{f}\in\mathcal{B}$ is a solution of the minimization problem \eqref{equi-min-unconstrain} with $\mathbf{y}$ if and only if
\begin{equation*}
0\in\partial\|\cdot\|_{\mathcal{B}}(\hat{f})
+\mathcal{L}^*\partial\iota_{\mathbf{y}}(\mathcal{L}(\hat{f})).
\end{equation*}
This is equivalent to that there exists $\mathbf{c}\in\mathbb{R}^m$ such that
\begin{equation}\label{characterize-prox-f-formula22}
\mathbf{c}\in\partial\iota_{\mathbf{y}}(\mathcal{L}(\hat{f}))
\end{equation}
and
\begin{equation}\label{characterize-prox-f-formula11}
-\mathcal{L}^*(\mathbf{c})\in\partial\|\cdot\|_{\mathcal{B}}(\hat{f}).
\end{equation}
According to the relation \eqref{conjugate}, we rewrite the inclusion relation \eqref{characterize-prox-f-formula22} as
\begin{equation*}\label{characterize-prox-f-formula222}
\mathcal{L}(\hat{f})\in\partial\iota_{\mathbf{y}}^*(\mathbf{c}).
\end{equation*}
Lemma \ref{subdiff-prox-Rm} ensures the equivalence between the inclusion relation above and \eqref{characterize-prox-f-formula1}. By the assumptions on $\mathcal{T}$ and noting that $\mathcal{L}^*(\mathbf{c})\in\mathcal{V}$, we rewrite \eqref{characterize-prox-f-formula11} as
\begin{equation}\label{characterize-prox-f-formula111}
-(\mathcal{T}^*\mathcal{T})(\mathcal{T}^*\mathcal{T})^{-1}
\mathcal{L}^*(\mathbf{c})
\in\partial\|\cdot\|_{\mathcal{B}}(\hat{f}).
\end{equation}
By employing Proposition \ref{subdiff-prox-Banach} with
$$
\psi:=\|\cdot\|_{\mathcal{B}},\ \ g:=-(\mathcal{T}^*\mathcal{T})^{-1}\mathcal{L}^*(\mathbf{c}),\ \ \mbox{and}\ \ f:=\hat{f},
$$
we get that the inclusion relation \eqref{characterize-prox-f-formula111} is equivalent to \eqref{characterize-prox-f-formula2}. Consequently, we conclude that $\hat{f}\in\mathcal{B}$ is a solution of the minimization problem \eqref{equi-min-unconstrain} with $\mathbf{y}$ if and only if there exists $\mathbf{c}\in\mathbb{R}^m$ such that
\eqref{characterize-prox-f-formula1} and \eqref{characterize-prox-f-formula2}.
\end{proof}

Theorem \ref{minimization-for-c-f} shows that when the Banach space $\mathcal{B}$ has the non-smooth pre-dual space $\mathcal{B}_*$ and $\nu_j\in\mathcal{B}_*,$ $j\in\mathbb{N}_m$, the coefficients $c_j$, $j\in\mathbb{N}_m$, appearing in Theorem \ref{representer-theorem-subdifferential-predual} can be obtained by solving the finite dimensional minimization problem \eqref{equivalent-minimization-predual}. A solution of \eqref{equivalent-minimization-predual} can be alternatively characterized via fixed-point equations. We next present this result.

\begin{thm}\label{characterize-prox-c}
Suppose that $\mathcal{B}$ is a Banach space having the pre-dual space $\mathcal{B}_{*}$, $\nu_j\in\mathcal{B}_*$, $j\in\mathbb{N}_m$, are linearly independent and $\mathbf{y}\in\mathbb{R}^m\setminus\{0\}$. Let $\mathcal{L}$, $\iota_{\mathbf{y}}$,  and $\mathcal{V}$ be defined by \eqref{functional-operator}, \eqref{indicator}, and \eqref{linear-span-nu-j}, respectively, and $\mathcal{L}^*$ be the adjoint operator of $\mathcal{L}$. Let $\mathcal{H}$ be a Hilbert space and $\mathcal{T}$ a bounded linear operator from $\mathcal{B}$ to $\mathcal{H}$ with the adjoint operator $\mathcal{T}^*$. If $\mathcal{T}^*\mathcal{T}$ is a one-to-one mapping from the inverse image of $\mathcal{V}$ onto $\mathcal{V}$, then $\hat{\mathbf{c}}\in\mathbb{R}^m$ is a solution of the minimization problem \eqref{equivalent-minimization-predual} with $\mathbf{y}$ if and only if there exists $\hat{f}\in\mathcal{B}$ such that the pair $\hat{f}$ and $\mathbf{c}:=-\|\hat{f}\|_{\mathcal{B}}\hat{\mathbf{c}}$ satisfies the fixed-point equations \eqref{characterize-prox-f-formula1} and \eqref{characterize-prox-f-formula2}.
\end{thm}

\begin{proof}
Proposition \ref{minimization-for-c} ensures that $\hat{\mathbf{c}}\in\mathbb{R}^m$ is a solution of \eqref{equivalent-minimization-predual} with $\mathbf{y}$ if and only if there exists $\hat{f}\in\mathcal{B}$ such that $\hat{f}\in\mathcal{M}_{\mathbf{y}}$ and
\begin{equation}\label{characterize-prox-c-1}
\hat{f}\in \frac{1}{\|\mathcal{L}^*(\hat{\mathbf{c}})\|_{\mathcal{B}_{*}}}
\partial\|\cdot\|_{\mathcal{B}_{*}}(\mathcal{L}^*(\hat{\mathbf{c}})).
\end{equation}
Set $\mathbf{c}:=-\|\hat{f}\|_{\mathcal{B}}\hat{\mathbf{c}}.$ It suffices to verify that $\hat{f}\in\mathcal{B}$ satisfies $\hat{f}\in\mathcal{M}_{\mathbf{y}}$ and
\eqref{characterize-prox-c-1} if and only if the pair $\hat{f}$ and $\mathbf{c}$ satisfies the fixed-point equations \eqref{characterize-prox-f-formula1} and \eqref{characterize-prox-f-formula2}. Note that $\hat{f}\in\mathcal{M}_{\mathbf{y}}$ is equivalent to $\mathcal{L}(\hat{f})=\mathbf{y}$. According to the representation \eqref{indicator-conjugate} of the conjugate function $\iota_{\mathbf{y}}^*$, its subdifferential at each $\mathbf{a}\in\mathbb{R}^m$ is a singleton, that is, $\partial\iota_{\mathbf{y}}^*(\mathbf{a})=\{\mathbf{y}\}$. We then conclude that  $\mathcal{L}(\hat{f})=\mathbf{y}$ if and only if $\mathcal{L}(\hat{f})\in\partial\iota_{\mathbf{y}}^*(\mathbf{c})$, which guaranteed by Lemma \ref{subdiff-prox-Rm} is equivalent to \eqref{characterize-prox-f-formula1}. Therefore, we have that $\hat{f}\in\mathcal{M}_{\mathbf{y}}$ if and only if $\hat{f}$ and $\mathbf{c}$ satisfies the fixed-point equation \eqref{characterize-prox-f-formula1}.

We next show that $\hat{f}\in\mathcal{B}$ satisfies
\eqref{characterize-prox-c-1} if and only if $\hat{f}$ and $\mathbf{c}$ satisfies the fixed-point equation \eqref{characterize-prox-f-formula2}. We rewrite \eqref{characterize-prox-c-1} as
\begin{equation}\label{characterize-prox-c-2}
\|\mathcal{L}^*(\hat{\mathbf{c}})\|_{\mathcal{B}_{*}}\hat{f}\in \partial\|\cdot\|_{\mathcal{B}_{*}}(\mathcal{L}^*(\hat{\mathbf{c}})).
\end{equation}
Note by the definition of $\mathbf{c}$ that
\begin{equation}\label{characterize-prox-c-3}
\partial\|\cdot\|_{\mathcal{B}_{*}}(\mathcal{L}^*(\hat{\mathbf{c}}))
=\partial\|\cdot\|_{\mathcal{B}_{*}}(-\mathcal{L}^*(\mathbf{c}))
\end{equation}
and
\begin{equation}\label{characterize-prox-c-4}
\|\mathcal{L}^*(\mathbf{c})\|_{\mathcal{B}_{*}}
=\|\hat{f}\|_{\mathcal{B}}
\|\mathcal{L}^*(\hat{\mathbf{c}})\|_{\mathcal{B}_{*}}.
\end{equation}
According to \eqref{characterize-prox-c-3}, we have that \eqref{characterize-prox-c-2} holds if and only if
\begin{equation}\label{characterize-prox-c-5}
\|\mathcal{L}^*(\hat{\mathbf{c}})\|_{\mathcal{B}_{*}}\hat{f}\in \partial\|\cdot\|_{\mathcal{B}_{*}}(-\mathcal{L}^*(\mathbf{c})),
\end{equation}
which guaranteed by \eqref{relation-subdifferential-functionals} is equivalent to
\begin{equation}\label{characterize-prox-c-6}
\|\hat{f}\|_{\mathcal{B}}\|\mathcal{L}^*(\hat{\mathbf{c}})\|_{\mathcal{B}_{*}}=1
\end{equation}
and
\begin{equation}\label{characterize-prox-c-7}
\frac{\hat{f}}{\|\hat{f}\|_{\mathcal{B}}}
\in\partial\|\cdot\|_{\mathcal{B}_{*}}(-\mathcal{L}^*(\mathbf{c})).
\end{equation}
Note by equation \eqref{characterize-prox-c-4} that \eqref{characterize-prox-c-6} is equivalent to $\|\mathcal{L}^*(\mathbf{c})\|_{\mathcal{B}_{*}}=1$. Accordingly, we conclude that  \eqref{characterize-prox-c-5} holds if and only if there hold $\|-\mathcal{L}^*(\mathbf{c})\|_{\mathcal{B}_{*}}=1$ and \eqref{characterize-prox-c-7}. Lemma \ref{subdifferential-solvable-thm} ensures the equivalence between the latter and inclusion relation \eqref{characterize-prox-f-formula11}. As has been shown in the proof of Theorem \ref{characterize-prox-f}, inclusion relation \eqref{characterize-prox-f-formula11} is equivalent to \eqref{characterize-prox-f-formula2}. Consequently, we have that $\hat{f}\in\mathcal{B}$ satisfies \eqref{characterize-prox-c-1} if and only if $\hat{f}$ and $\mathbf{c}$ satisfies the fixed-point equation \eqref{characterize-prox-f-formula2}. This completes the proof of this theorem.
\end{proof}

Theorem \ref{characterize-prox-f} shows that solving the minimum interpolation problem \eqref{mni} can be done by solving fixed-point equations \eqref{characterize-prox-f-formula1} and \eqref{characterize-prox-f-formula2}. These two fixed-point equations are coupled together and they have to be solved simultaneously by iteration. In general, the fixed-point equation \eqref{characterize-prox-f-formula2} is of infinite dimension, which requires further investigation to convert it to a finite dimensional fixed-point equation. We demonstrate this point by considering the case when $\mathcal{B}=\ell_1(\mathbb{N})$ whose special property will enable us to reduce the corresponding fixed-point equation \eqref{characterize-prox-f-formula2} to a finite dimensional one.

\subsection{Fixed-Point Approach for Minimum Norm Interpolation in $\ell_1(\mathbb{N})$}
In this subsection, we establish a fixed-point characterization for a solution of the minimum norm interpolation problem \eqref{mni} in the special Banach space $\ell_1(\mathbb{N})$. We are especially interested in showing how the fixed-point equations \eqref{characterize-prox-f-formula1} and \eqref{characterize-prox-f-formula2} which is of infinite dimension is reduced to equivalent finite dimensional fixed-point equations.

We first derive the proximity operator of convex functions on $\ell_1(\mathbb{N})$. According to the definition \eqref{prox-predual}, we need to choose an appropriate Hilbert space $\mathcal{H}$ and the operator $\mathcal{T}:\ell_{1}(\mathbb{N})\rightarrow\mathcal{H}$. Noting that there hold the inclusion relations
$$
\ell_1(\mathbb{N})\subset\ell_2(\mathbb{N})\subset c_0,
$$
we define the embedding operator $\mathcal{T}_0:\ell_1(\mathbb{N})\rightarrow\ell_2(\mathbb{N})$ as
\begin{equation}\label{T-l1}
\mathcal{T}_0(\mathbf{x}):=\mathbf{x},\ \ \mbox{for all}\ \ \mathbf{x}\in\ell_1(\mathbb{N}).
\end{equation}
Since there holds
$$
\|\mathbf{x}\|_2\leq\|\mathbf{x}\|_{1},\ \ \mbox{for all}\ \ \mathbf{x}\in\ell_1(\mathbb{N}),
$$
we conclude that $\mathcal{T}_0$ is bounded. We next express the adjoint operator $\mathcal{T}_0^*$ of $\mathcal{T}_0$. There holds for all $\mathbf{u}\in\ell_2(\mathbb{N})$ and all $\mathbf{x}\in\ell_1(\mathbb{N})$ that
$$
\langle\mathcal{T}_0^*(\mathbf{u}),\mathbf{x}\rangle_{\ell_1}
=\langle\mathbf{u},\mathcal{T}_0(\mathbf{x})\rangle_{\ell_2}
=\sum_{j\in\mathbb{N}}u_jx_j.
$$
This yields that $\mathcal{T}_0^*:\ell_2(\mathbb{N})\rightarrow\ell_{\infty}(\mathbb{N})$
has the form
$$
\mathcal{T}_0^*(\mathbf{u})=\mathbf{u},
\ \ \mbox{for all}\ \ \mathbf{u}\in\ell_2(\mathbb{N}).
$$
By choosing $\mathcal{H}:=\ell_2(\mathbb{N})$ and $\mathcal{T}:=\mathcal{T}_0$ defined by \eqref{T-l1}, the proximity operator
$\mathrm{prox}_{\psi,\ell_2(\mathbb{N}),\mathcal{T}_0}: \ell_1(\mathbb{N})\to \ell_1(\mathbb{N})$ of a convex function $\psi:\ell_1(\mathbb{N})\rightarrow\mathbb{R}\cup\{+\infty\}$ defined by \eqref{prox-predual} has the form
\begin{equation}\label{prox-l1}
\mathrm{prox}_{\psi,\ell_2(\mathbb{N}),\mathcal{T}_0}(\mathbf{x})=\arg\inf\left\{\frac12\|\mathbf{x}-\mathbf{z}\|_2^2
+\psi(\mathbf{z}):\mathbf{z}\in\ell_1(\mathbb{N})\right\},
\ \ \mbox{for all}\ \ \mathbf{x}\in\ell_1(\mathbb{N}).
\end{equation}
Since $\mathcal{T}_0$ and $\mathcal{T}_0^*$ are both embedding operator, there holds
$$
\mathcal{T}_0^*\mathcal{T}_0\mathbf{z}=\mathbf{z},\ \ \mbox{for all}\ \ \mathbf{z}\in\ell_1(\mathbb{N}).
$$
Hence, as a consequence of Proposition \ref{subdiff-prox-Banach}, we may get the relation between the proximity operator \eqref{prox-l1} of $\psi$ and its subdifferential. Specifically,
if $\psi:\ell_1(\mathbb{N})\to\mathbb{R}\cup\{+\infty\}$ is a convex function, then for all $\mathbf{x}\in{\rm dom}(\psi)$ and $\mathbf{z}\in\ell_1(\mathbb{N})$ that
\begin{equation}\label{subdiff-prox-l1}
\mathbf{z}\in\partial\psi(\mathbf{x})
\ \ \mbox{if and only if}\ \
\mathbf{x}=\mathrm{prox}_{\psi,\ell_2(\mathbb{N}),\mathcal{T}_0}
(\mathbf{x}+\mathbf{z}).
\end{equation}

We can explicitly calculate the proximity operator \eqref{prox-l1} of the norm $\psi:=\|\cdot\|_{1}$ of $\ell_1(\mathbb{N})$ and its subdifferential. Specifically, for each $\mathbf{x}:=(x_j:j\in\mathbb{N})\in\ell_1(\mathbb{N})$ there holds
$$
\partial\|\cdot\|_{1}(\mathbf{x}) =\{\mathbf{u}\in\ell_{\infty}(\mathbb{N}): \mathbf{u} =(u_j:j\in\mathbb{N}), u_j\in\partial|\cdot|(x_j), j\in\mathbb{N}\},
$$
where for each $x\in\mathbb{R}$
\begin{equation*}
\partial|\cdot|(x):=\left\{\begin{array}{ll}
{\rm sign}(x), &\mbox{if}\  x\neq 0, \\

[-1,1], & \mbox{if}\  x=0,
\end{array}
\right.
\end{equation*}
and
\begin{equation}\label{prox-1norm}
\mathrm{prox}_{\|\cdot\|_{1},\ell_2(\mathbb{N}),\mathcal{T}_0}(\mathbf{x})
=\left(\max\{|x_j|-1,0\}{\rm sign}(x_j):j\in\mathbb{N}\right).
\end{equation}

We denote by $c_c$ the space of all real sequences on $\mathbb{N}$ having at most a finite number of nonzero components. Clearly, we have that $c_c\subset c_0$. For each $\mathbf{x}\in c_c$, the support of $\mathbf{x}$, denoted by $\mathrm{supp}(\mathbf{x})$, is defined to be the index set on which $\mathbf{x}$ is nonzero. The next proposition ensures that the proximity operator of the norm function $\|\cdot\|_1$ is a mapping from $\ell_1(\mathbb{N})$ to $c_c$.

\begin{prop}\label{prox-l1-l0}
If $\mathcal{T}_0:\ell_1(\mathbb{N})\rightarrow \ell_2(\mathbb{N})$ is defined by \eqref{T-l1} and the proximity operator $\mathrm{prox}_{\|\cdot\|_1,\ell_2(\mathbb{N}),\mathcal{T}_0}$ is defined by \eqref{prox-l1} with $\psi:=\|\cdot\|_1$, then for all $\mathbf{x}\in\ell_1(\mathbb{N})$,
$$
\mathrm{prox}_{\|\cdot\|_1,\ell_2(\mathbb{N}),\mathcal{T}_0}(\mathbf{x})
\in c_c.
$$
\end{prop}
\begin{proof}
For $\mathbf{x}:=(x_j:j\in\mathbb{N})$, we let
$$
\mathbf{y}=(y_j:j\in\mathbb{N}):=\mathrm{prox}_{\|\cdot\|_1,\ell_2(\mathbb{N}),\mathcal{T}_0}(\mathbf{x}).
$$
It follows from equation \eqref{prox-1norm} that
$$
y_j=\max\{|x_j|-1,0\}{\rm sign}(x_j),\ \ \mbox{for all}\ \ j\in\mathbb{N}.
$$
Since $\mathbf{x}\in\ell_1(\mathbb{N})$, there exists an positive integer $N$ such that $|x_j|<1$, for all $j>N$. This together with the above equations leads to $y_j=0$, for all $j>N$.
That is, $\mathbf{y}\in c_c$.
\end{proof}

We now turn to solving the minimization problem \eqref{equi-min-unconstrain}
in the case that $\mathcal{B}:=\ell_1(\mathbb{N})$. The minimization problem \eqref{equi-min-unconstrain} in this case has the form
\begin{equation}\label{equi-min-l1-noconstrain}
\inf\{\|\mathbf{x}\|_{1}
+\iota_{\mathbf{y}}(\mathcal{L}(\mathbf{x})):\mathbf{x}\in\ell_1(\mathbb{N})\},
\end{equation}
where $\mathcal{L}$ is defined by \eqref{functional-operator-l1} with $\mathbf{u}_j\in c_0$, $j\in\mathbb{N}_m.$ The solution of the minimization problem \eqref{equi-min-l1-noconstrain} can be characterized as a fixed-point of a map defined on a finite dimensional space. This benefits from an important property of the space $c_0$. Specifically, for each $\mathbf{u}\in c_0$, since $u_j$ tends to zero as $j\rightarrow+\infty,$ $\mathbf{u}$ attains its norm $\|\mathbf{u}\|_{\infty}$ on the finite index set $\mathbb{N}(\mathbf{u})$, defined by \eqref{extremal-index-set}. By virtue of this property, we define a truncation operator $\mathcal{S}:c_0\rightarrow c_c$ as
\begin{equation}\label{truncation}
\mathcal{S}(\mathbf{u}):=(\tilde{u}_j:j\in\mathbb{N})
\end{equation}
with
\begin{equation}\label{truncation-vector}
\tilde{u}_j:=\left\{\begin{array}{ll}
u_j, &\mbox{if}\  j\in\mathbb{N}(\mathbf{u}), \\
0, & \mbox{if}\  j\notin\mathbb{N}(\mathbf{u}).
\end{array}
\right.
\end{equation}
Clearly, we have that
$$
\|\mathbf{u}\|_{\infty}=\|\mathcal{S}(\mathbf{u})\|_{\infty}, \ \ \mbox{for all} \ \ \mathbf{u}\in c_0.
$$

As a direct consequence of Lemma \ref{subdifferentials-c0}, we get the following relation between the subdifferential of the norm $\|\cdot\|_{\infty}$ of $c_0$ at nonzero $\mathbf{u}\in c_0$ and that at $\mathcal{S}(\mathbf{u})$.

\begin{lemma}\label{improtant-character1}
If the truncation operator $\mathcal{S}:c_0\rightarrow c_c$ is defined by \eqref{truncation} and \eqref{truncation-vector}, then for each nonzero $\mathbf{u}\in c_0$,
\begin{equation}\label{important1}
\partial\|\cdot\|_{\infty}(\mathbf{u})=\partial\|\cdot\|_{\infty}(\mathcal{S}(\mathbf{u})).
\end{equation}
\end{lemma}
\begin{proof}
It follows from the definition of $\mathcal{S}$ that for each $\mathbf{u}\in c_0$, there holds
$$
\mathbb{N}(\mathbf{u})=\mathbb{N}(\mathcal{S}(\mathbf{u})),
$$
which together with definition \eqref{convex-hull} leads to
$$
\mathcal{V}(\mathbf{u})=\mathcal{V}(\mathcal{S}(\mathbf{u})).
$$
Hence, by Lemma \ref{subdifferentials-c0} we get that
$$
\partial\|\cdot\|_{\infty}(\mathbf{u})={\rm co}(\mathcal{V}(\mathbf{u}))={\rm co}(\mathcal{V}(\mathcal{S}(\mathbf{u}))) = \partial\|\cdot\|_{\infty}(\mathcal{S}(\mathbf{u})),
$$
proving the desired result.
\end{proof}

By employing the above lemma, we get below a technical lemma, which is useful for establishing the finite dimensional fixed-point equations for a solution of the minimization problem \eqref{equi-min-l1-noconstrain}.

\begin{lemma}\label{improtant-character2}
If the truncation operator $\mathcal{S}:c_0\rightarrow c_c$ is defined by \eqref{truncation} and \eqref{truncation-vector}, then for all $\mathbf{x}\in\ell_1(\mathbb{N})$ and all $\mathbf{u}\in c_0$,
\begin{equation}\label{important2}
\mathbf{u}\in\partial\|\cdot\|_{1}(\mathbf{x})\ \ \mbox{if and only if}\ \
\mathcal{S}(\mathbf{u})\in\partial\|\cdot\|_1(\mathbf{x}).
\end{equation}
\end{lemma}
\begin{proof}
By the definition of the subdifferential, we have that
$$
\partial\|\cdot\|_{1}(0)=\{\mathbf{v}\in\ell_{\infty}(\mathbb{N}):\|\mathbf{v}\|_{\infty}\leq1\}.
$$
This together with the fact that $\|\mathbf{u}\|_{\infty}=\|\mathcal{S}(\mathbf{u})\|_{\infty}$ ensures that \eqref{important2} holds for $\mathbf{x}=0$. It remains to prove the desired conclusion in the case that $\mathbf{x}\neq 0$. By Lemma \ref{subdifferential-solvable-thm}, there holds
$$
\mathbf{u}\in\partial\|\cdot\|_{1}(\mathbf{x})\ \ \mbox{if and only if}\ \ \frac{\mathbf{x}}{\|\mathbf{x}\|_1}\in\partial\|\cdot\|_{\infty}(\mathbf{u}).
$$
Lemma \ref{improtant-character1} ensures that the last inclusion relation is equivalent to
$$
\frac{\mathbf{x}}{\|\mathbf{x}\|_1}\in\partial\|\cdot\|_{\infty}(\mathcal{S}(\mathbf{u})).
$$
Again, using Lemma \ref{subdifferential-solvable-thm} and noting that $\|\mathcal{S}(\mathbf{u})\|_{\infty}=\|\mathbf{u}\|_{\infty}=1$, we conclude that the above inclusion relation
is equivalent to
$$
\mathcal{S}(\mathbf{u})\in\partial\|\cdot\|_1(\mathbf{x}),
$$
proving the desired result \eqref{important2}.
\end{proof}

We are now ready to characterize the solution of the minimization problem \eqref{equi-min-l1-noconstrain} by fixed-point equations.

\begin{thm}\label{characterize-prox-f-l1}
Suppose that $\mathbf{u}_j\in c_0$, $j\in\mathbb{N}_m,$ are linearly independent,  $\mathbf{y}\in\mathbb{R}^m$ and that $\mathcal{L}$, $\iota_{\mathbf{y}}$ and $\mathcal{T}_0$ are defined by \eqref{functional-operator-l1}, \eqref{indicator} and \eqref{T-l1}, respectively. Let the truncation operator $\mathcal{S}$ be defined by \eqref{truncation} and \eqref{truncation-vector}. Then $\hat{\mathbf{x}}\in\ell_1(\mathbb{N})$ is a solution of the minimization problem \eqref{equi-min-l1-noconstrain} with $\mathbf{y}$ if and only if there exists $\mathbf{c}\in\mathbb{R}^m$ such that
\begin{equation}\label{characterize-prox-formula-l1-1}
\mathbf{c}=\mathrm{prox}_{\iota_{\mathbf{y}}^*}
(\mathbf{c}+\mathcal{L}(\hat{\mathbf{x}}))
\end{equation}
and
\begin{equation}\label{characterize-prox-formula-l1-2}
\hat{\mathbf{x}}=\mathrm{prox}_{\|\cdot\|_1,\ell_2(\mathbb{N}),\mathcal{T}_0}
(\hat{\mathbf{x}}-\mathcal{S}\mathcal{L}^*(\mathbf{c})).
\end{equation}
\end{thm}
\begin{proof}
As has been shown in the proof of Theorem \ref{characterize-prox-f}, $\hat{\mathbf{x}}\in\ell_1(\mathbb{N})$ is a solution of the minimization problem \eqref{equi-min-l1-noconstrain} with $\mathbf{y}$ if and only if there exists $\mathbf{c}\in\mathbb{R}^m$ such that
\begin{equation}\label{characterize-prox-l1-1}
\mathbf{c}\in\partial\iota_{\mathbf{y}}(\mathcal{L}(\hat{\mathbf{x}}))
\end{equation}
and
\begin{equation}\label{characterize-prox-l1-2}
-\mathcal{L}^*(\mathbf{c})\in\partial\|\cdot\|_{1}(\hat{\mathbf{x}}).
\end{equation}
By relation \eqref{conjugate}, the inclusion relation \eqref{characterize-prox-l1-1} has the  equivalent form
$$
\mathcal{L}(\hat{\mathbf{x}})\in\partial\iota_{\mathbf{y}}^*(\mathbf{c}),
$$
which guaranteed by Lemma \ref{subdiff-prox-Rm} is equivalent to the fixed-point equation \eqref{characterize-prox-formula-l1-1}. Since $\mathbf{u}_j\in c_0$, $j\in\mathbb{N}_m,$ we have that  $\mathcal{L}^*(\mathbf{c})\in c_0$. Hence, by Lemma \ref{improtant-character2}, we conclude that the inclusion relation \eqref{characterize-prox-l1-2} holds if and only if
\begin{equation}\label{characterize-prox-l1-3}
-\mathcal{S}\mathcal{L}^*(\mathbf{c})\in\partial\|\cdot\|_{1}(\hat{\mathbf{x}}).
\end{equation}
Note that $-\mathcal{S}\mathcal{L}^*(\mathbf{c})\in\ell_1(\mathbb{N})$. Relation \eqref{subdiff-prox-l1} ensures that the inclusion relation \eqref{characterize-prox-l1-3} is equivalent to the fixed-point equation \eqref{characterize-prox-formula-l1-2}. Consequently, we have that $\hat{\mathbf{x}}\in\ell_1(\mathbb{N})$ is a solution of the minimization problem \eqref{equi-min-l1-noconstrain} with $\mathbf{y}$ if and only if there exists $\mathbf{c}\in\mathbb{R}^m$ satisfying \eqref{characterize-prox-formula-l1-1} and \eqref{characterize-prox-formula-l1-2}.
\end{proof}

Following Theorem \ref{characterize-prox-f-l1}, we can give a characterization by fixed-point equations for the solution of the dual problem
\begin{equation}\label{equivalent-minimization-l1}
\inf\left\{\|\mathcal{L}^*(\mathbf{c})\|_{\infty}:
\langle\mathbf{c},\mathbf{y}\rangle_{\mathbb{R}^m}=1,
\mathbf{c}\in\mathbb{R}^m\right\}.
\end{equation}
of the minimum norm interpolation problem in the space $\ell_1(\mathbb{N})$.

\begin{thm}\label{characterize-prox-c-l1}
Suppose that $\mathbf{u}_j\in c_0$, $j\in\mathbb{N}_m,$ are linearly independent and $\mathbf{y}\in\mathbb{R}^m\setminus\{0\}$ and that $\mathcal{L}$, $\iota_{\mathbf{y}}$ and $\mathcal{T}_0$ are defined by \eqref{functional-operator-l1}, \eqref{indicator} and \eqref{T-l1}, respectively. Let the truncation operator $\mathcal{S}$ by defined by \eqref{truncation} and \eqref{truncation-vector}. Then $\hat{\mathbf{c}}\in\mathbb{R}^m$ is a solution of the minimization problem \eqref{equivalent-minimization-l1} with $\mathbf{y}$ if and only if there exists $\hat{\mathbf{x}}\in\ell_1(\mathbb{N})$ such that the pair $\hat{\mathbf{x}}$ and $\mathbf{c}:=-\|\hat{\mathbf{x}}\|_{1}\hat{\mathbf{c}}$ satisfy the fixed-point equations \eqref{characterize-prox-formula-l1-1} and \eqref{characterize-prox-formula-l1-2}.
\end{thm}
\begin{proof}
By Proposition \ref{minimization-for-c}, we have that $\hat{\mathbf{c}}\in\mathbb{R}^m$ is a solution of \eqref{equivalent-minimization-l1} with $\mathbf{y}$ if and only if there exists $\hat{\mathbf{x}}\in\ell_1(\mathbb{N})$ such that
\begin{equation}\label{characterize-prox-c-l1-1}
\hat{\mathbf{x}}\in \frac{1}{\|\mathcal{L}^*(\hat{\mathbf{c}})\|_{\infty}}
\partial\|\cdot\|_{\infty}(\mathcal{L}^*(\hat{\mathbf{c}}))\cap\mathcal{M}_{\mathbf{y}}.
\end{equation}
It suffices to verify that $\hat{\mathbf{x}}\in\ell_1(\mathbb{N})$ satisfies \eqref{characterize-prox-c-l1-1} if and only if the pair $\hat{\mathbf{x}}$ and $\mathbf{c}:=-\|\hat{\mathbf{x}}\|_{1}\hat{\mathbf{c}}$ satisfy  \eqref{characterize-prox-formula-l1-1} and \eqref{characterize-prox-formula-l1-2}.
As pointed out in the proof of Theorem \ref{characterize-prox-c}, we conclude that $\hat{\mathbf{x}}\in\mathcal{M}_{\mathbf{y}}$ if and only if $\hat{\mathbf{x}}$ and $\mathbf{c}$ satisfy \eqref{characterize-prox-formula-l1-1}. We also have that
\begin{equation}\label{characterize-prox-c-f-l1-3}
\hat{\mathbf{x}}\in \frac{1}{\|\mathcal{L}^*(\hat{\mathbf{c}})\|_{\infty}}
\partial\|\cdot\|_{\infty}(\mathcal{L}^*(\hat{\mathbf{c}}))
\end{equation}
if and only if
$$
-\mathcal{L}^*(\mathbf{c})\in\partial\|\cdot\|_{1}(\hat{\mathbf{x}}).
$$
This guaranteed by Lemma \ref{improtant-character2} is equivalent to
$$
-\mathcal{S}\mathcal{L}^*(\mathbf{c})\in\partial\|\cdot\|_{1}(\hat{\mathbf{x}}).
$$
By relation \eqref{subdiff-prox-l1} and noting that $-\mathcal{S}\mathcal{L}^*(\mathbf{c})\in\ell_1(\mathbb{N})$, the conclusion relation above can be characterized by fixed-point equation \eqref{characterize-prox-formula-l1-2}. Therefore, we get the conclusion that $\hat{\mathbf{x}}$ satisfies \eqref{characterize-prox-c-f-l1-3} if $\hat{\mathbf{x}}$ and $\mathbf{c}$ satisfy \eqref{characterize-prox-formula-l1-2}. This completes the proof of this theorem.
\end{proof}

The fixed-points equations \eqref{characterize-prox-formula-l1-1} and \eqref{characterize-prox-formula-l1-2} appearing in both Theorems \ref{characterize-prox-f-l1} and \ref{characterize-prox-c-l1} are in fact of finite dimension. We unfold this fact in the remaining part of this section. To this end, we first present a technical lemma.

\begin{lemma}\label{support-x-su}
If $\mathbf{u}\in c_0$ is nonzero, then for each $\mathbf{x}\in\partial\|\cdot\|_{\infty}(\mathbf{u})$, there hold
\begin{equation}\label{support}
\mathbf{x}\in c_c\ \  \mbox{and}\ \ \mathrm{supp}(\mathbf{x})\subseteq\mathrm{supp}(\mathcal{S}(\mathbf{u})).
\end{equation}
\end{lemma}
\begin{proof}
Note that for all nonzero element $\mathbf{u}\in c_0$, there holds
$$
\mathrm{supp}(\mathcal{S}(\mathbf{u}))=\mathbb{N}(\mathbf{u}).
$$
By Lemma \ref{subdifferentials-c0}, each $\mathbf{x}\in\partial\|\cdot\|_{\infty}(\mathbf{u})$ is a convex combination of elements of $\mathcal{V}(\mathbf{u})$ whose supports are contained in $\mathbb{N}(\mathbf{u})$. This leads to the inclusion relation in \eqref{support} and thus, $\mathbf{x}\in c_c$.
\end{proof}

With the help of Lemma \ref{support-x-su}, we reveal the finite dimension component of the fixed-points equations \eqref{characterize-prox-formula-l1-1} and \eqref{characterize-prox-formula-l1-2}.
It is convenient to write the fixed-point equations \eqref{characterize-prox-formula-l1-1} and \eqref{characterize-prox-formula-l1-2} in a compact form. To this end, we stack the vector $\mathbf{c}\in\mathbb{R}^m$ on the top of finite dimensional vector $\hat{\mathbf{x}}$ to form a new vector
\begin{equation*}
\mathbf{s}:=\left[\begin{array}{c}
\mathbf{c}\\ \hat{\mathbf{x}}
\end{array}
\right].
\end{equation*}
We also introduce two matrices of operators by
\begin{equation}\label{definitionofP}
\mathcal{P}:=\left[\begin{array}{c}
\mathrm{prox}_{\iota_{\mathbf{y}}^*}\\
\mathrm{prox}_{\|\cdot\|_1,\ell_2(\mathbb{N}),\mathcal{T}_0}
\end{array}
\right]
\end{equation}
and
\begin{equation}\label{definitionofR}
\mathcal{R}:=\left[\begin{array}{cc}
\mathcal{I}& \mathcal{L}\\
-\mathcal{S}\mathcal{L}^* &\mathcal{I}
\end{array}
\right].
\end{equation}
In the above notion, we rewrite the equations \eqref{characterize-prox-formula-l1-1} and \eqref{characterize-prox-formula-l1-2} in the following compact form
\begin{equation}\label{fixed-point-l1}
\mathbf{s}=(\mathcal{P}\circ\mathcal{R})(\mathbf{s}).
\end{equation}
The following theorem implies that the fixed-point equation \eqref{fixed-point-l1} (or equivalently the system of the fixed-points equations \eqref{characterize-prox-formula-l1-1} and \eqref{characterize-prox-formula-l1-2} is of finite dimension.

\begin{thm}\label{finite-dimension-fixed-point}
If operators $\mathcal{P}$ and $\mathcal{R}$ are defined respectively by \eqref{definitionofP} and \eqref{definitionofR}, then $\mathcal{P}\circ\mathcal{R}$ is an operator from $(\mathbb{R}^m,\ell_1(\mathbb{N}))$ to $(\mathbb{R}^m, c_c)$ and its fixed-point $\mathbf{s}=\left[\begin{array}{c}
\mathbf{c}\\ \hat{\mathbf{x}}\end{array}\right]
\in(\mathbb{R}^m,\ell_1(\mathbb{N}))$  satisfies

\begin{equation}\label{support1}
\hat{\mathbf{x}}\in c_c\ \ \mbox{and}\ \ \mathrm{supp}(\hat{\mathbf{x}})\subseteq
\mathrm{supp}(\mathcal{S}(\mathcal{L}^*(\mathbf{c})).
\end{equation}
\end{thm}
\begin{proof}
Proposition \ref{prox-l1-l0} ensures that the proximity operator $\mathrm{prox}_{\|\cdot\|_1,\ell_2(\mathbb{N}),\mathcal{T}_0}$ is a mapping from $\ell_1(\mathbb{N})$ to $c_c$. Thus, for any $\mathbf{s}\in(\mathbb{R}^m,\ell_1(\mathbb{N}))$, we get that
$$
\hat{\mathbf{x}}-\mathcal{S}\mathcal{L}^*(\mathbf{c})\in\ell_1(\mathbb{N})
$$
and then
\begin{equation}\label{finite-dimension-nature1}
\mathrm{prox}_{\|\cdot\|_1,\ell_2(\mathbb{N}),\mathcal{T}_0}
(\hat{\mathbf{x}}-\mathcal{S}\mathcal{L}^*(\mathbf{c}))\in c_c.
\end{equation}
On the other hand, the proximity operator $\mathrm{prox}_{\iota_{\mathbf{y}}^*}$ is a mapping from $\mathbb{R}^m$ to itself. Note that for any $\mathbf{s}
\in(\mathbb{R}^m,\ell_1(\mathbb{N}))$, there holds
$$
\mathbf{c}+\mathcal{L}(\hat{\mathbf{x}})\in\mathbb{R}^m.
$$
Therefore, we have that
\begin{equation}\label{finite-dimension-nature2}
\mathrm{prox}_{\iota_{\mathbf{y}}^*}(\mathbf{c}+\mathcal{L}(\hat{\mathbf{x}}))\in\mathbb{R}^m.
\end{equation}
Combining \eqref{finite-dimension-nature1} with \eqref{finite-dimension-nature2}, we conclude that $\mathcal{P}\circ\mathcal{R}$ is an operator from $(\mathbb{R}^m,\ell_1(\mathbb{N}))$ to $(\mathbb{R}^m, c_c)$.

It remains to verify that the fixed-point $\mathbf{s}$ of operator $\mathcal{P}\circ\mathcal{R}$ satisfies the assertion of this theorem. Suppose that  $\mathbf{s}$ is a fixed-point of operator $\mathcal{P}\circ\mathcal{R}$. That is, $\hat{\mathbf{x}}$ and $\mathbf{c}$ satisfy the fixed-point equations \eqref{characterize-prox-formula-l1-1} and \eqref{characterize-prox-formula-l1-2}.
According to the proof of Theorem \ref{characterize-prox-f-l1}, we observe that $\hat{\mathbf{x}}$ satisfies the inclusion \eqref{characterize-prox-l1-2}, which guaranteed by Lemma \ref{subdifferential-solvable-thm} leads to
$$
\frac{\hat{\mathbf{x}}}{\|\hat{\mathbf{x}}\|_{1}}
\in\partial\|\cdot\|_{\infty}(-\mathcal{L}^*(\mathbf{c})).
$$
By Lemma \ref{support-x-su}, the above inclusion ensures that  $\mathbf{s}$  satisfies \eqref{support1}.
\end{proof}

A solution of the minimum norm interpolation \eqref{mni} with $\mathcal{B}:=\ell_1(\mathbb{N})$ guaranteed by Theorems \ref{characterize-prox-f-l1} and \ref{finite-dimension-fixed-point} has an additional property.

\begin{rem}\label{Remark-on-support}
Each solution $\hat{\mathbf{x}}\in\ell_1(\mathbb{N})$ of the minimum norm interpolation \eqref{mni} with $\mathcal{B}:=\ell_1(\mathbb{N})$ together with $\mathbf{c}\in\mathbb{R}^m$ satisfying the fixed-point equations \eqref{characterize-prox-formula-l1-1} and \eqref{characterize-prox-formula-l1-2} is of finite dimension, that is, it satisfies
\eqref{support1}.
\end{rem}
\begin{proof}
Let $\hat{\mathbf{x}}$ be a solution of the minimum norm interpolation \eqref{mni} with $\mathcal{B}:=\ell_1(\mathbb{N})$. By Theorem \ref{characterize-prox-f-l1}, there exists $\mathbf{c}\in\mathbb{R}^m$ such that $\hat{\mathbf{x}}$ and $\mathbf{c}$ satisfy the fixed-point equations \eqref{characterize-prox-formula-l1-1} and \eqref{characterize-prox-formula-l1-2}. That is, $\left[\begin{array}{c}\mathbf{c}\\ \hat{\mathbf{x}}\end{array}\right]$ is the fixed-point of $\mathcal{P}\circ\mathcal{R}$. By employing Theorem \ref{finite-dimension-fixed-point}, we conclude that $\hat{\mathbf{x}}$ satisfies \eqref{support1}.
\end{proof}

Theorem \ref{characterize-prox-f-l1} reveals that to solve the minimum norm interpolation \eqref{mni} with $\mathcal{B}:=\ell_1(\mathbb{N})$, it suffices to find a solution of the fixed-point equations \eqref{characterize-prox-formula-l1-1} and \eqref{characterize-prox-formula-l1-2} by iterative algorithms designed based on these fixed-point equations. A remarkable fact is that according to Theorem \ref{finite-dimension-fixed-point}, the fixed-point equations \eqref{characterize-prox-formula-l1-1} and \eqref{characterize-prox-formula-l1-2} are both of {\it finite} dimension. Therefore, solving the infinitely dimensional minimum norm interpolation \eqref{mni} with $\mathcal{B}:=\ell_1(\mathbb{N})$ reduces to finding a fixed-point of a nonlinear map defined on a finite dimensional space.

To develop efficient iterative algorithms with convergence guaranteed based on these fixed-point equations, we need to consider additional issues: The first issue is the computation of the proximity operators of the two functions involved in the fixed-point equations. Moreover, the direct iteration from \eqref{characterize-prox-formula-l1-1} and \eqref{characterize-prox-formula-l1-2} may not lead to convergent algorithms. One needs to reformulate these fixed-point equations to equivalent ones guided by the theory of firmly non-expansive maps. This is the second issue. The third issue is how convergence of the resulting convergent iterative schemes can be accelerated by introducing some parameters or matrices.

We now address the first issue. Note that the closed-form formula for the proximity operator $\mathrm{prox}_{\|\cdot\|_1,\ell_2(\mathbb{N}),\mathcal{T}_0}$ has been given in \eqref{prox-1norm}. We now present the proximity operator of $\iota_{\mathbf{y}}^*$ below. Clearly, by the definition of the indicator function $\iota_{\mathbf{y}}$, its proximity operator has the form $\mathrm{prox}_{\iota_{\mathbf{y}}}(\mathbf{a}):=\mathbf{y}$ for all $\mathbf{a}\in\mathbb{R}^m$. Then equation \eqref{prox-psi-psi*} with $\psi:=\iota_{\mathbf{y}}$ leads to closed-form formula of the proximity operator $\mathrm{prox}_{\iota_{\mathbf{y}}^*}$ as
$$
\mathrm{prox}_{\iota_{\mathbf{y}}^*}(\mathbf{a})
=\mathbf{a}-\mathbf{y},
\ \ \mbox{for all}\ \ \mathbf{a}\in\mathbb{R}^m.
$$
The two closed-form formulas enable us to implement the iteration efficiently.

We next discuss the second issue. Since the equations \eqref{characterize-prox-formula-l1-1} and \eqref{characterize-prox-formula-l1-2} are represented in the compact form \eqref{fixed-point-l1}, one may define the Picard iteration based on \eqref{fixed-point-l1} to find the fixed-point $\mathbf{s}$, that is
\begin{equation}\label{fixed-point-algorithm}
\mathbf{s}^{k+1}=(\mathcal{P}\circ\mathcal{R})(\mathbf{s}^k), \ \ k=0,1, \dots.
\end{equation}
When it converges, the Picard sequence $\mathbf{s}^k$,  $k=0,1, \dots$, generated by the Picard iteration \eqref{fixed-point-algorithm}, converges to a fixed-point of the map $\mathcal{P}\circ\mathcal{R}$, which gives a solution of the minimum norm interpolation problem \eqref{mni}. However, convergence of the Picard sequence is not guaranteed. Normally, we need to reformulate the fixed-point equation \eqref{fixed-point-algorithm} by appropriately split the matrix $\mathcal{R}$ guided by the theory of the non-expansive map. That is, we will construct from the map $\mathcal{P}\circ\mathcal{R}$ a non-expansive map $\mathcal{M}$ which has the same fixed-point set as $\mathcal{P}\circ\mathcal{R}$, so that the Picard sequence of the new map $\mathcal{M}$ converges to a fixed-point of $\mathcal{M}$, guaranteed by its non-expansiveness. Interested readers are referred to \cite{LSXZ} for further algorithmic development along this line. We will address this issue together with other computational issues in a different occasion.

\section{Regularization Problem and its Connection with Minimum Norm Interpolation}

We now turn to considering regularization problems. In the remaining part of this paper, the term ``regularization problem'' will refer to both regularized learning and other semi-discrete inverse problems unless stated otherwise. This is because these two types of problems have the same mathematical structure as far as their representer theorems are concerned, as we have explained earlier. Regularization problems are closely related to minimum norm interpolation problems. We shall translate the representer theorems obtained in section 3 for solutions of minimum norm interpolation problems to those of regularization problems. Since the regularization problem in a general Banach space is described as an infinity dimensional minimization problem, we first comment on the existence of a solution of the problem following general results regarding the existence of a solution of an infinity dimensional minimization problem. Moreover, we establish an intrinsic connection between the regularization problem and the minimum norm interpolation problem. Specifically, we shall show that there always exists a solution of the regularization problem which is also a solution of the minimum norm interpolation problem with specific data.

\subsection{Regularization Problem}
In this subsection, we first describe the regularization problem in a Banach space, review its background and several practical examples of importance, and establish existence of its solution under a rather mild condition.

We begin with describing the regularization problem under investigation. Let $\mathcal{B}$ be a real Banach space with the dual space $\mathcal{B}^*$. Suppose that a set of linearly independent functionals $\nu_j\in\mathcal{B}^*$, $j\in\mathbb{N}_m,$ is given and operator $\mathcal{L}:\mathcal{B}\rightarrow\mathbb{R}^m$ is defined by equation \eqref{functional-operator}.
Learning a target element in $\mathcal{B}$ from the given set of sampled data $\{(\nu_j, y_j):j\in\mathbb{N}_m\}$ consists of solving the following first kind operator equation
\begin{equation}\label{ill-posed}
\mathcal{L}(f)=\mathbf{y}
\end{equation}
for $f\in\mathcal{B}$, where $\mathbf{y}:=[y_j: j\in\mathbb{N}_m]\in\mathbb{R}^m$. Equation \eqref{ill-posed} is a typical ill-posed problem. That is, the inverse of $\mathcal{L}$ is not bounded. A commonly used approach to address the ill-posedness of \eqref{ill-posed} is regularization. Specifically, we define a data fidelity term $\mathcal{Q}_{\mathbf{y}}(\mathcal{L}(f))$ from \eqref{ill-posed} by using a loss function $\mathcal{Q}_{\mathbf{y}}:\mathbb{R}^m\rightarrow\mathbb{R}_{+}$, and solve the minimization problem
\begin{equation}\label{regularization}
\inf\{\mathcal{Q}_{\mathbf{y}}(\mathcal{L}(f))
+\lambda\varphi(\|f\|_{\mathcal{B}}): f\in\mathcal{B}\},
\end{equation}
where $\varphi:\mathbb{R}_{+}\rightarrow\mathbb{R}_{+}$ is a regularizer and $\lambda$ is a positive regularization parameter.

The regularization problem \eqref{regularization} appears in many applied areas.
We present several examples of the loss function and the regularizer that are used frequently in applications. In machine learning, classical regularization network and support vector machines for both classification and regression are reformulated as \eqref{regularization} \cite{EPP,SS}. Specifically, if for $\mathbf{y}\in\mathbb{R}^m$ the loss function is chosen as
\begin{equation}\label{RN}
\mathcal{Q}_{\mathbf{y}}(\mathbf{z})
:=\|\mathbf{z}-\mathbf{y}\|_{\mathbb{R}^m}^2,
\ \ \mathbf{z}\in\mathbb{R}^m,
\end{equation}
and the regularizer as
\begin{equation}\label{varphi=t^2}
\varphi(t):=t^2,\ \ t\in\mathbb{R}_{+},
\end{equation}
then the regularization problem \eqref{regularization} reduces to the regularization networks. The support vector machine regression has the form \eqref{regularization} with the loss function being chosen as
\begin{equation}\label{SVMR}
\mathcal{Q}_{\mathbf{y}}(\mathbf{z}):=\sum_{j\in\mathbb{N}_m}\max\{|y_j-z_j|-\epsilon,0\},\ \mathbf{z}\in\mathbb{R}^m,
\end{equation}
where $\epsilon$ is a positive constant and the regularizer as \eqref{varphi=t^2}. If the loss function is chosen as
\begin{equation}\label{SVMC}
\mathcal{Q}_{\mathbf{y}}(\mathbf{z}):=\sum_{j\in\mathbb{N}_m}\max\{1-y_jz_j,0\},\ \mathbf{z}:=[z_j:j\in\mathbb{N}_m]\in\{-1,1\}^m,
\end{equation}
for $\mathbf{y}:=[y_j:j\in\mathbb{N}_m]\in\{-1,1\}^m$ and the regularizer as \eqref{varphi=t^2}, the regularization problem \eqref{regularization} describes the support vector machine classification. Moreover, $\ell_1$ support vector machine regression and classification \cite{Li-Song-Xu2018,Li-Song-Xu2019,SS,ZRHT} are formulated as \eqref{regularization} with the loss function \eqref{SVMR} and \eqref{SVMC}, respectively, and the regularizer
\begin{equation}\label{varphi=t}
\varphi(t):=t,\ t\in\mathbb{R}_{+}.
\end{equation}
The Lasso regularized model \cite{T,ZY} is also formulated as \eqref{regularization} with the loss function as \eqref{RN} and the regularizer as \eqref{varphi=t} with an appropriate choice of the Banach space. Another example concerns the $l_p$-norm regularization \cite{Z} in which the regularizer is chosen as
$\varphi(t):=t^p,\ t\in\mathbb{R}_{+}$.

Most data science problems are described as semi-discrete inverse problems \cite{Daubechies04, Wen}. Such inverse problems covers many important application areas including image restoration \cite{CDOS, Lu-Shen-Xu2010} and medical imaging
\cite{Chen-Huang-Li-Lu-Xu2020, Jiang-Li-Xu2018}. Semi-discrete inverse problems often solved by regularization methods \cite{CMX, CXY} are formulated in the form \eqref{regularization} with appropriate choices of the loss function and regularizer. The form of the loss function is normally determined by types of noise contaminated in given data.

We now consider the existence of a solution of the regularization problem \eqref{regularization}.
By using arguments similar to those used in the proof of the existence of solutions of the minimum norm interpolation problem \eqref{mni}, we can get the existence of solutions of \eqref{regularization} under the conditions that $\mathcal{B}$ has a pre-dual space $\mathcal{B}_{*}$ and $\nu_j\in\mathcal{B}_{*}$, $j\in\mathbb{N}_m$. To this end, we review a few useful properties of functions appearing in \eqref{regularization}. A function $\mathcal{T}$ mapping from a topological space $\mathcal{X}$ into $\mathbb{R}$ is said to be lower semi-continuous if
$$
\mathcal{T}(f)\leq\liminf_{\alpha}\mathcal{T}(f_{\alpha})
$$
whenever $f_{\alpha}$, $\alpha\in I,$ for some index set $I$ is a net in $\mathcal{X}$ converging to some element $f\in \mathcal{X}$. The notion of weakly$^{*}$ lower semi-continuous is defined accordingly under the weak$^{*}$ topology. We say a function $\mathcal{T}$ mapping from a normed space $\mathcal{X}$ into $\mathbb{R}$ is coercive if
$$
\lim_{\|x\|\rightarrow+\infty}\mathcal{T}(x)=\infty.
$$
The following lemma ensures the existence of a bounded minimizing sequence in $\mathcal{B}$. For notational convenience, we set
\begin{equation}\label{R}
\mathcal{R}(f):=\mathcal{Q}_{\mathbf{y}}(\mathcal{L}(f))
+\lambda\varphi(\|f\|_{\mathcal{B}}),
\ \ \mbox{for all}\ \ f\in\mathcal{B}.
\end{equation}

\begin{lemma}\label{existence}
Suppose that $\mathcal{B}$ is a Banach space with the dual space $\mathcal{B}^*$, $\nu_j\in\mathcal{B}^*$, $j\in\mathbb{N}_m,$ and $\mathcal{L}$ is defined by \eqref{functional-operator}. Let $\mathbf{y}\in\mathbb{R}^m$, $\mathcal{Q}_{\mathbf{y}}:\mathbb{R}^m\rightarrow\mathbb{R}_{+}$, $\varphi:\mathbb{R}_{+}\rightarrow\mathbb{R}_{+}$, $\lambda>0$ be as those appearing in
\eqref{regularization} and $\mathcal{R}$ be defined by \eqref{R}. If $\varphi$ is coercive, then there exists a bounded sequence $f_n,n\in\mathbb{N},$ in $\mathcal{B}$ such that
\begin{equation}\label{minimizing}
\lim_{n\rightarrow+\infty}\mathcal{R}(f_n)=\inf_{f\in\mathcal{B}}\mathcal{R}(f).
\end{equation}
\end{lemma}
\begin{proof}
For any $\epsilon>0$, there exists an element $g\in\mathcal{B}$ such that
$$
\inf_{f\in\mathcal{B}}\mathcal{R}(f)\leq\mathcal{R}(g)
<\inf_{f\in\mathcal{B}}\mathcal{R}(f)+\epsilon.
$$
Hence, there exists a sequence $f_n,n\in\mathbb{N},$ in $\mathcal{B}$ satisfying \eqref{minimizing}. It remains to show that the sequence is bounded. It follows from \eqref{minimizing} that $\{\mathcal{R}(f_n):n\in\mathbb{N}\}$ is a bounded set. Moreover, by the definition \eqref{R} of $\mathcal{R}$, we have that
$$
\mathcal{R}(f_n)\geq\lambda\varphi(\|f_n\|_{\mathcal{B}}),\ \ \mbox{for all}\ \ n\in \mathbb{N}.
$$
This together with the boundedness of the set $\{\mathcal{R}(f_n):n\in\mathbb{N}\}$ implies that $\{\varphi(\|f_n\|_{\mathcal{B}}):n\in\mathbb{N}\}$ is also a bounded set. By the coercivity of $\varphi$, the boundedness of the set $\{\varphi(\|f_n\|_{\mathcal{B}}):n\in\mathbb{N}\}$ leads to the boundedness of the sequence $f_n,n\in\mathbb{N}$.
\end{proof}

With the help of Lemma \ref{existence}, we establish in the following proposition a sufficient condition which ensures the existence of a solution of the regularization problem \eqref{regularization}.

\begin{prop}\label{existence1}
Suppose that $\mathbf{y}\in\mathbb{R}^m$, both $\mathcal{Q}_{\mathbf{y}}:\mathbb{R}^m\rightarrow\mathbb{R}_{+}$ and $\varphi:\mathbb{R}_{+}\rightarrow\mathbb{R}_{+}$ are lower semi-continuous, $\lambda>0$ and moreover, $\varphi$ is increasing and coercive. If $\mathcal{B}$ is a Banach space having the pre-dual space $\mathcal{B}_{*}$ and the functionals $\nu_j,j\in\mathbb{N}_m,$ appearing in the definition of $\mathcal{L}$ are in $\mathcal{B}_{*}$, then the regularization problem \eqref{regularization} has at least one solution.
\end{prop}
\begin{proof}
Since $\varphi$ is coercive, by Lemma \ref{existence} there exists a bounded sequence $f_n,n\in\mathbb{N},$ in $\mathcal{B}$ satisfying \eqref{minimizing} with $\mathcal{R}$ being defined as in \eqref{R}. It follows from the Banach-Alaoglu theorem that there exists a subsequence $f_{n_k},k\in\mathbb{N},$ weakly$^{*}$ converges to $\hat{f}\in\mathcal{B}$. We shall prove that the weak$^{*}$ accumulation point $\hat{f}$ is a solution of the regularization problem \eqref{regularization}. This is done by showing that
\begin{equation}\label{lsc}
\mathcal{R}(\hat{f})\leq\liminf_{j\rightarrow+\infty}\mathcal{R}(f_{n_{k_j}}),
\end{equation}
where $f_{n_{k_j}},j\in\mathbb{N},$ is a subsequence of the sequence $f_{n_k},k\in\mathbb{N}$.

By the definition \eqref{R} of $\mathcal{R}$, we consider the fidelity term $\mathcal{Q}_{\mathbf{y}}(\mathcal{L}(\hat{f}))$ and the regularization term $\varphi(\|\hat{f}\|_{\mathcal{B}})$ separately. We first consider the fidelity term.
Since $\nu_j\in\mathcal{B}_{*}$, $j\in\mathbb{N}_m$, the linear functionals $\nu_j$, $j\in\mathbb{N}_m,$ are weakly$^{*}$ continuous. Hence, we conclude that the linear operator $\mathcal{L}$ defined by \eqref{functional-operator} in terms of the linear functionals $\nu_j$, $j\in\mathbb{N}_m,$ is weakly$^{*}$ continuous. The assumption that $\mathcal{Q}_\mathbf{y}$ is lower semi-continuous yields $\mathcal{Q}_{\mathbf{y}}(\mathcal{L})$ is weakly$^{*}$ lower semi-continuous. Hence, by the weak$^{*}$ convergence of the sequence $f_{n_{k}},j\in\mathbb{N},$ we obtain that
\begin{equation}\label{lsc-Q}
\mathcal{Q}_{\mathbf{y}}(\mathcal{L}(\hat{f}))
\leq\liminf_{j\rightarrow+\infty}\mathcal{Q}_{\mathbf{y}}(\mathcal{L}(f_{n_{k}})).
\end{equation}

We now consider the regularization term. Noting that the norm $\|\cdot\|_{\mathcal{B}}$ is weak$^{*}$ continuous on $\mathcal{B}$, by the weak$^{*}$ convergence of the sequence $f_{n_k},k\in\mathbb{N},$ we get that
\begin{equation}\label{lsc-norm}
\|\hat{f}\|_{\mathcal{B}}\leq\liminf_{k\rightarrow+\infty}\|f_{n_k}\|_{\mathcal{B}}.
\end{equation}
Let $f_{n_{k_j}},j\in\mathbb{N},$ be the subsequence of the sequence $f_{n_{k}},k\in\mathbb{N}$  which attains the limit inferior in \eqref{lsc-norm}. It follows that
\begin{equation*}\label{lsc-norm1}
\|\hat{f}\|_{\mathcal{B}}\leq\lim_{j\rightarrow+\infty}\|f_{n_{k_j}}\|_{\mathcal{B}}.
\end{equation*}
Since $\varphi$ is lower semi-continuity and increasing, we have that \begin{equation}\label{lsc-norm2}
\varphi(\|\hat{f}\|_{\mathcal{B}})
\leq\varphi\left(\lim_{j\rightarrow+\infty}\|f_{n_{k_j}}\|_{\mathcal{B}}\right)
\leq\liminf_{j\rightarrow+\infty}\varphi(\|f_{n_{k_j}}\|_{\mathcal{B}}).
\end{equation}

Finally, combining inequalities \eqref{lsc-Q} and \eqref{lsc-norm2} yields the inequality \eqref{lsc}, which
together with \eqref{minimizing} leads to
\begin{equation}\label{lsc1}
\mathcal{R}(\hat{f})\leq\inf_{f\in\mathcal{B}}\mathcal{R}(f).
\end{equation}
Clearly, inequality \eqref{lsc1} ensures that $\hat{f}$ is a solution of the regularization problem \eqref{regularization}.
\end{proof}

The reflexive Banach space is a special case of the Banach space having a pre-dual space.
In this special case, we can also have the existence result for a solution of the regularization problem \eqref{regularization}.

\begin{cor}\label{existence2}
Suppose that $\mathbf{y}\in\mathbb{R}^m$, both $\mathcal{Q}_{\mathbf{y}}:\mathbb{R}^m\rightarrow\mathbb{R}_{+}$ and $\varphi:\mathbb{R}_{+}\rightarrow\mathbb{R}_{+}$ are lower semi-continuous, $\lambda>0$ and moreover, $\varphi$ is increasing and coercive. If $\mathcal{B}$ is a reflexive Banach space having the dual space $\mathcal{B}^{*}$and the functionals $\nu_j, j\in\mathbb{N}_m,$ appearing in the definition of $\mathcal{L}$ are in $\mathcal{B}^{*}$, then the regularization problem \eqref{regularization} has at least one solution.
\end{cor}
\begin{proof}
A reflexive Banach space $\mathcal{B}$ always has its pre-dual space $\mathcal{B}_{*}$ being its dual space $\mathcal{B}^{*}$. Hence, the desired result follows directly from Proposition \ref{existence1}.
\end{proof}

\subsection{Connection between regularization problem and Minimum Norm Interpolation}

We now investigate the connection between a solution of the regularization problem \eqref{regularization} and that of the minimum norm interpolation problem \eqref{mni}. Specifically, we shall show that if the regularizer is increasing then there always exists a solution of the regularization problem \eqref{regularization} which is also a solution of the minimum norm interpolation problem with specific data. Furthermore, if the regularizer is strictly increasing then every solution of the regularization problem \eqref{regularization} is also a solution of the minimum norm interpolation problem with specific data. Throughout the rest of this paper, we always assume that each of the two minimization problems has a solution without further mention. In particular, it is guaranteed by Proposition \ref{existence1} that this assumption holds when $\mathcal{B}$ has the pre-dual space $\mathcal{B}_{*}$, $\nu_j\in\mathcal{B}_{*}$, $j\in\mathbb{N}_m$,  $\mathcal{Q}_{\mathbf{y}}$, $\varphi$ are both lower semi-continuous and $\varphi$ is increasing and coercive.


\begin{prop}\label{relation-two-optimization}
Suppose that $\mathcal{B}$ is a Banach space with the dual space $\mathcal{B}^*$, $\nu_j\in\mathcal{B}^*$, $j\in\mathbb{N}_m,$ and $\mathcal{L}$ is defined by \eqref{functional-operator}. For a given $\mathbf{y}_0\in\mathbb{R}^m$, let $\mathcal{Q}_{\mathbf{y}_0}:\mathbb{R}^m\rightarrow\mathbb{R}_{+}$ be a loss function, $\varphi:\mathbb{R}_{+}\rightarrow\mathbb{R}_{+}$ be an increasing regularizer and $\lambda>0$.
Let $\hat{f}\in\mathcal{B}$ be a solution of the regularization problem \eqref{regularization} with $\mathbf{y}:=\mathbf{y}_0$.

\indent (1) A solution $\hat{g}\in\mathcal{B}$ of the minimum norm interpolation problem \eqref{mni} with $\mathbf{y}:=\mathcal{L}(\hat{f})$ is a solution of the regularization problem \eqref{regularization} with $\mathbf{y}:=\mathbf{y}_0$.

\indent (2) If $\varphi$ is strictly increasing, then $\hat{f}$ is a solution of the minimum norm interpolation problem \eqref{mni} with $\mathbf{y}:=\mathcal{L}(\hat{f})$.
\end{prop}
\begin{proof}
We first prove statement (1). Suppose that $\hat{g}$ is a solution of the minimum norm interpolation problem \eqref{mni} for data $\mathbf{y}:=\mathcal{L}(\hat{f})$. It follows from $\hat{f}\in\mathcal{M}_{\mathbf{y}}$ with $\mathbf{y}:=\mathcal{L}(\hat{f})$ that
\begin{equation}\label{Lg}
\mathcal{L}(\hat{g})=\mathcal{L}(\hat{f})
\end{equation}
and
\begin{equation}\label{g}
 \|\hat{g}\|_{\mathcal{B}}\leq\|\hat{f}\|_{\mathcal{B}}.
\end{equation}
On one hand, equation \eqref{Lg} further implies that
\begin{equation}\label{QLg}
\mathcal{Q}_{\mathbf{y}_0}(\mathcal{L}(\hat{g}))
=\mathcal{Q}_{\mathbf{y}_0}(\mathcal{L}(\hat{f})).
\end{equation}
On the other hand, since $\varphi$ is increasing, from \eqref{g} we have that
\begin{equation}\label{varphig}
\varphi(\|\hat{g}\|_{\mathcal{B}})\leq\varphi(\|\hat{f}\|_{\mathcal{B}}).
\end{equation}
Combining \eqref{QLg} and \eqref{varphig}, with noting that $\lambda$ is positive, we obtain that
$$
\mathcal{Q}_{\mathbf{y}_0}(\mathcal{L}(\hat{g}))+\lambda\varphi(\|\hat{g}\|_{\mathcal{B}})\leq
\mathcal{Q}_{\mathbf{y}_0}(\mathcal{L}(\hat{f}))+\lambda\varphi(\|\hat{f}\|_{\mathcal{B}}).
$$
This ensures that $\hat{g}$ is a solution of the regularization problem \eqref{regularization} with given data $\mathbf{y}:=\mathbf{y}_0$.

We next show statement (2). Suppose that $\varphi$ is strictly increasing. Set $\hat{\mathbf{y}}:=\mathcal{L}(\hat{f})$. It suffices to verify that
\begin{equation}\label{f}
\|\hat{f}\|_{\mathcal{B}}\leq\|f\|_{\mathcal{B}},\ \ \mbox{for all}\ \ f\in\mathcal{M}_{\hat{\mathbf{y}}}.
\end{equation}
On one hand, for all $f\in\mathcal{M}_{\hat{\mathbf{y}}}$ we have that $\mathcal{L}(\hat{f})=\mathcal{L}(f),$ which leads to
\begin{equation}\label{QLf}
\mathcal{Q}_{\mathbf{y}_0}(\mathcal{L}(\hat{f}))
=\mathcal{Q}_{\mathbf{y}_0}(\mathcal{L}(f)), \ \ \mbox{for all}\ \ f\in\mathcal{M}_{\hat{\mathbf{y}}}.
\end{equation}
On the other hand, since $\hat{f}$ is a solution of \eqref{regularization} with given data $\mathbf{y}:=\mathbf{y}_0$, we get that
\begin{equation}\label{Rf}
\mathcal{Q}_{\mathbf{y}_0}(\mathcal{L}(\hat{f}))+\lambda\varphi(\|\hat{f}\|_{\mathcal{B}})
\leq\mathcal{Q}_{\mathbf{y}_0}(\mathcal{L}(f))+\lambda\varphi(\|f\|_{\mathcal{B}}), \ \ \mbox{for all} \ \ f\in\mathcal{B}.
\end{equation}
Combining \eqref{QLf} and \eqref{Rf}, with noting that $\lambda$ is positive, we have that
\begin{equation}\label{varphi-ineq}
\varphi(\|\hat{f}\|_{\mathcal{B}})\leq\varphi(\|f\|_{\mathcal{B}}),  \ \ \mbox{for all}\ \ f\in\mathcal{M}_{\hat{\mathbf{y}}}.
\end{equation}
Since $\varphi$ is strictly increasing, we get the result \eqref{f} from inequality \eqref{varphi-ineq}. This ensures that $\hat{f}$ is a solution of the minimum norm interpolation problem \eqref{mni} for data $\mathbf{y}:=\mathcal{L}(\hat{f})$.
\end{proof}

Statement (2) of Proposition \ref{relation-two-optimization} was claimed in \cite{MP} without details of proof. We provide above a complete proof for this statement for convenience of the readers.


\section{Representer Theorems and Solution Methods for Regularization Problems}

In this section, we establish representer theorems and solution methods for solutions of the regularization problem \eqref{regularization}. Using the connection enacted in the last section between a solution of the minimum norm interpolation \eqref{mni} and that of the regularization problem, we first present both implicit and explicit representer theorems for a solution of the regularization problem \eqref{regularization}. We then develop solution methods for solving the regularization problem. We present two types of solution methods: one based on the representer theorems and the other being direct methods. We also consider special cases and give special results for them. In particular, for the regularization problem in $\ell_1(\mathbb{N})$, we put forward a fixed-point formulation which serves as a basis for further development of efficient iterative algorithms for solving the problem. Although results to be presented in this section are parallel to those for the minimum norm interpolation, they will provide a foundation for applications due to wide utilizations of regularization problems in many areas. We will keep our presentation concise by skipping some details.


\subsection{Representer Theorems for Regularization}
We present in this subsection representer theorems for regularization problem \eqref{regularization}.
Recall that we have established several implicit representer theorems in Proposition \ref{representer-theorem-summarize} for solutions of the minimum norm interpolation problem \eqref{mni}. Through the connection between a solution of the regularization problem \eqref{regularization} and that of the minimum norm interpolation problem \eqref{mni}, in the next proposition we first translate the results in Proposition \ref{representer-theorem-summarize} originally for the minimum norm interpolation problem \eqref{mni} to those for the regularization problem \eqref{regularization}.
Throughout this subsection, we assume that for any given data $\mathbf{y}$, each of the minimum norm interpolation problem \eqref{mni} and the regularization problem \eqref{regularization} has a solution in the same Banach space under consideration.

\begin{prop}\label{representer-regularization}
Suppose that $\mathcal{B}$ is a Banach space with the dual space $\mathcal{B}^*$, $\nu_j\in\mathcal{B}^*$, $j\in\mathbb{N}_m,$ and $\mathcal{L}$ is defined by \eqref{functional-operator}. For a given  $\mathbf{y}_0\in\mathbb{R}^m$, let $\mathcal{Q}_{\mathbf{y}_0}:\mathbb{R}^m\rightarrow\mathbb{R}_{+}$ be a loss function, $\varphi:\mathbb{R}_{+}\rightarrow\mathbb{R}_{+}$ be a regularizer and $\lambda>0$.
If $\varphi$ is increasing, then there exists a solution $f_0$ of the regularization problem \eqref{regularization} with $\mathbf{y}:=\mathbf{y}_0$ satisfying the following conditions:

(i) There exist $c_j\in\mathbb{R}$, $j\in\mathbb{N}_m,$ such that the linear functional $\nu:=\sum_{j\in\mathbb{N}_m}c_j\nu_j$ satisfying
\begin{equation*}\label{representer-regularization-functional}
\|\nu\|_{\mathcal{B}^*}=1\ \ \mbox{and}\ \ \langle\nu,f_0\rangle_{\mathcal{B}}=\|f_0\|_{\mathcal{B}}.
\end{equation*}

(ii) There exist $c_j\in\mathbb{R}$, $j\in\mathbb{N}_m,$ such that the linear functional $\nu:=\sum_{j\in\mathbb{N}_m}c_j\nu_j$ peaks at $f_0$, that is,
\begin{equation*}\label{representer-regularization-peak-functional}
\langle\nu,f_0\rangle_{\mathcal{B}}
=\left\|\nu\right\|_{\mathcal{B}^*}\|f_0\|_{\mathcal{B}}.
\end{equation*}

(iii) There exist $c_j\in\mathbb{R}$, $j\in\mathbb{N}_m,$ such that
\begin{equation*}\label{representer-regularization-duality}
\sum_{j\in\mathbb{N}_m}c_j\nu_j\in\mathcal{J}(f_0).
\end{equation*}

(iv) There exist $c_j\in\mathbb{R},\ j\in\mathbb{N}_m,$ such that
\begin{equation}\label{representer-regularization-subdifferential}
\sum_{j\in\mathbb{N}_m}c_j\nu_j\in\partial\|\cdot\|_{\mathcal{B}}(f_0).
\end{equation}

If $\varphi$ is strictly increasing, then every solution $f_0$ of the regularization problem \eqref{regularization} with $\mathbf{y}:=\mathbf{y}_0$ satisfies the above conditions (i)-(iv).
\end{prop}
\begin{proof}
Suppose that $\hat{f}$ is a solution of the regularization problem \eqref{regularization} with $\mathbf{y}:=\mathbf{y}_0$. Let $f_0$ be a solution of the minimum norm interpolation problem \eqref{mni} with $\mathbf{y}:=\mathcal{L}(\hat{f})$. Proposition \ref{relation-two-optimization} ensures that $f_0$ is also a solution of the regularization problem \eqref{regularization} with $\mathbf{y}:=\mathbf{y}_0$. As a solution of the minimum norm interpolation problem, $f_0$ has the representations as described in Proposition \ref{representer-theorem-summarize}. Hence, $f_0$ satisfies conditions (i)-(iv) of this proposition.

We now consider the case that $\varphi$ is strictly increasing. Suppose that $f_0$ is a solution of the regularization problem \eqref{regularization} with $\mathbf{y}:=\mathbf{y}_0$. Statement (2) of Proposition \ref{relation-two-optimization} ensures that $f_0$ is also a solution of the minimum norm interpolation problem \eqref{mni} with $\mathbf{y}:=\mathcal{L}(f_0)$. Again by Proposition \ref{representer-theorem-summarize}, we get that $f_0$ satisfies conditions (i)-(iv). Due to the arbitrariness of $f_0$, we obtain the desired conclusion.
\end{proof}

We remark that in a non-smooth Banach space, the regularized learning was studied in \cite{HLTY, Unser}. Specifically, condition (iv) of Proposition \ref{representer-regularization} was established in \cite{HLTY} for a special regularizer $\varphi(t):=t^2$, $t\in\mathbb{R}_{+}$, and condition (iii) of Proposition \ref{representer-regularization} was obtained in \cite{Unser} via a different approach, the duality mapping.

When the Banach space $\mathcal{B}$ is smooth, we may get a representer theorem for a solution of the regularization problem \eqref{regularization} in a simple form.

\begin{prop}\label{representer-regularization-Gateaux}
Suppose that $\mathcal{B}$ is a smooth Banach space with the dual space $\mathcal{B}^*$, $\nu_j\in\mathcal{B}^*$, $j\in\mathbb{N}_m,$ and $\mathcal{L}$ is defined by \eqref{functional-operator}. For a given  $\mathbf{y}_0\in\mathbb{R}^m$, let $\mathcal{Q}_{\mathbf{y}_0}:\mathbb{R}^m\rightarrow\mathbb{R}_{+}$ be a loss function, $\varphi:\mathbb{R}_{+}\rightarrow\mathbb{R}_{+}$ be a regularizer and $\lambda>0$.

(1) If $\varphi$ is increasing, then there exists a solution $f_0$ of the regularization problem \eqref{regularization} with $\mathbf{y}:=\mathbf{y}_0$ such that
\begin{equation}\label{representer-regularization-Gateaux-formula}
\mathcal{G}(f_0)=\sum_{j\in\mathbb{N}_m}c_j\nu_j,
\end{equation}
for some $c_j\in\mathbb{R}$, $j\in\mathbb{N}_m$.

(2) If $\varphi$ is strictly increasing, then every solution $f_0$ of the regularization problem \eqref{regularization} with $\mathbf{y}:=\mathbf{y}_0$ satisfies \eqref{representer-regularization-Gateaux-formula} for some $c_j\in\mathbb{R}$, $j\in\mathbb{N}_m$.
\end{prop}

\begin{proof}
We prove this proposition by using condition (iv) of Proposition \ref{representer-regularization}.
If the regularization problem \eqref{regularization} with $\mathbf{y}:=\mathbf{y}_0$ has $f_0=0$ as its solution, equation \eqref{representer-regularization-Gateaux-formula} can be obtained by choosing $c_j=0$, $j\in\mathbb{N}_m$. Otherwise, by condition (iv) of Proposition \ref{representer-regularization},
there exists a nonzero solution $f_0$ of the regularization problem \eqref{regularization} with $\mathbf{y}:=\mathbf{y}_0$ satisfying the inclusion relation \eqref{representer-regularization-subdifferential} for some $c_j\in\mathbb{R}$, $j\in\mathbb{N}_m.$ Note that since $\mathcal{B}$ is smooth, the subdifferential of the norm function $\|\cdot\|_{\mathcal{B}}$ at $f_0\in\mathcal{B}$ is a singleton, that is,
\begin{equation}\label{singleton1}
\partial\|\cdot\|_{\mathcal{B}}(f_0)=\{\mathcal{G}(f_0)\}.
\end{equation}
This together with \eqref{representer-regularization-subdifferential} leads to equation \eqref{representer-regularization-Gateaux-formula}.

Suppose that $\varphi$ is strictly increasing and $f_0$ is a solution of the regularization problem \eqref{regularization}. If $f_0=0$, equation \eqref{representer-regularization-Gateaux-formula} holds for $c_j=0,$ $j\in\mathbb{N}_m.$ If $f_0$ is nonzero, then condition (iv) of Proposition \ref{representer-regularization} and equation \eqref{singleton1} yield the conclusion of this Proposition.
\end{proof}

We next develop explicit representer theorems for regularization problem \eqref{regularization}.
These explicit representer theorems are obtained from the implicit representer theorems presented above in conjunction with Lemma \ref{subdifferential-solvable-thm}, a duality argument.
We first consider the case when a Banach space $\mathcal{B}$ has the dual space $\mathcal{B}^*$.

\begin{thm}\label{representer-regularization-subdifferential-explicit}
Suppose that $\mathcal{B}$ is a Banach space with the dual space $\mathcal{B}^*$, $\nu_j\in\mathcal{B}^*$, $j\in\mathbb{N}_m$ and $\mathcal{L}$ is defined by \eqref{functional-operator}. For a given  $\mathbf{y}_0\in\mathbb{R}^m$, let $\mathcal{Q}_{\mathbf{y}_0}:\mathbb{R}^m\rightarrow\mathbb{R}_{+}$ be a loss function, $\varphi:\mathbb{R}_{+}\rightarrow\mathbb{R}_{+}$ be a regularizer and $\lambda>0$.

(1) If $\varphi$ is increasing, then there exists a solution $f_0$ of the regularization problem \eqref{regularization} with $\mathbf{y}:=\mathbf{y}_0$ such that
\begin{equation}\label{representer-regularization-subdifferential-explicit-formula}
f_0\in\gamma\partial\|\cdot\|_{\mathcal{B}^*}
\left(\sum_{j\in\mathbb{N}_m}c_j\nu_j\right),
\end{equation}
for some $c_j\in\mathbb{R}$, $j\in\mathbb{N}_m,$
with $\gamma:=\left\|\sum_{j\in\mathbb{N}_m}c_j\nu_j\right\|_{\mathcal{B}^*}$.

(2) If $\varphi$ is strictly increasing, then every solution $f_0$ of the regularization problem \eqref{regularization} with $\mathbf{y}:=\mathbf{y}_0$ satisfies \eqref{representer-regularization-subdifferential-explicit-formula} for some $c_j\in\mathbb{R}$, $j\in\mathbb{N}_m$.
\end{thm}
\begin{proof}
We prove this result by employing condition (iv) of Proposition \ref{representer-regularization}.
If the regularization problem \eqref{regularization} with $\mathbf{y}:=\mathbf{y}_0$ has a trivial solution $f_0=0$, we get equation  \eqref{representer-regularization-subdifferential-explicit-formula} by choosing $c_j=0$, $j\in\mathbb{N}_m$. We now consider the case that the regularization problem \eqref{regularization} with $\mathbf{y}:=\mathbf{y}_0$ does not have a trivial solution. In this case, condition (iv) of Proposition \ref{representer-regularization} ensures that there exists a nonzero solution $f_0$ of the regularization problem \eqref{regularization} with $\mathbf{y}:=\mathbf{y}_0$ satisfying \eqref{representer-regularization-subdifferential} for some $\hat{c}_j\in\mathbb{R}$, $j\in\mathbb{N}_m.$ Since $f_0\neq0,$ by equation \eqref{relation-subdifferential-functionals} we get that the functional $\hat{\nu}:=\sum_{j\in\mathbb{N}_m}\hat{c}_j\nu_j$ is nonzero. Hence, Lemma \ref{subdifferential-solvable-thm} leads to
\begin{equation}\label{representer-regularization-subdifferential-explicit1}
f_0\in\|f_0\|_{\mathcal{B}}\partial\|\cdot\|_{\mathcal{B}^*}(\hat{\nu}).
\end{equation}
Set
$$
c_j:=\|f_0\|_{\mathcal{B}}\hat{c}_j, \ \ j\in\mathbb{N}_m, \ \ \mbox{and}\ \ \nu:=\sum_{j\in\mathbb{N}_m}c_j\nu_j.
$$
It remains to prove that $c_j, j\in\mathbb{N}_m,$  satisfies \eqref{representer-regularization-subdifferential-explicit-formula}, that is,
\begin{equation}\label{representer-regularization-subdifferential-explicit2}
f_0\in\|\nu\|_{\mathcal{B}^*}\partial\|\cdot\|_{\mathcal{B}^*}(\nu).
\end{equation}
Since $\nu=\|f_0\|_{\mathcal{B}}\hat{\nu}$ and $\|\hat{\nu}\|_{\mathcal{B}^*}=1$, we get that
\begin{equation}\label{representer-regularization-subdifferential-explicit3}
\|\nu\|_{\mathcal{B}^*}=\|f_0\|_{\mathcal{B}}.
\end{equation}
It follows from equation \eqref{relation-subdifferential-functionals} that
\begin{equation}\label{subdifferential-relation}
\partial\|\cdot\|_{\mathcal{B}^*}(\hat{\nu})
=\partial\|\cdot\|_{\mathcal{B}^*}(\nu).
\end{equation}
Substituting \eqref{representer-regularization-subdifferential-explicit3} and \eqref{subdifferential-relation} into the right-hand side of \eqref{representer-regularization-subdifferential-explicit1},
we get equation \eqref{representer-regularization-subdifferential-explicit2}.

Next, we consider the case that $\varphi$ is strictly increasing. If $f_0=0$ is a solution of the regularization problem \eqref{regularization} with $\mathbf{y}:=\mathbf{y}_0$, it has the form \eqref{representer-regularization-subdifferential-explicit-formula} with $c_j=0$, $j\in\mathbb{N}_m$. By condition (iv) of Proposition \ref{representer-regularization}, we can represent any nonzero solution $f_0$ in the form \eqref{representer-regularization-subdifferential} for some $\hat{c}_j\in\mathbb{R}$, $j\in\mathbb{N}_m$. By setting $c_j:=\|f_0\|_{\mathcal{B}}\hat{c}_j,$ $j\in\mathbb{N}_m,$ and  arguments similar to those presented above, we get equation \eqref{representer-regularization-subdifferential-explicit-formula}.
\end{proof}

When the Banach space $\mathcal{B}$ has the smooth dual space $\mathcal{B}^*$, we have a special representer theorem for a solution of the regularization problem \eqref{regularization}.

\begin{thm}\label{representer-regularization-Gateaux-explicit}
Suppose that $\mathcal{B}$ is a Banach space having the smooth dual space $\mathcal{B}^{*}$, $\nu_j\in\mathcal{B}^*$, $j\in\mathbb{N}_m$ and $\mathcal{L}$ is defined by \eqref{functional-operator}. For a given  $\mathbf{y}_0\in\mathbb{R}^m$, let $\mathcal{Q}_{\mathbf{y}_0}:\mathbb{R}^m\rightarrow\mathbb{R}_{+}$ be a loss function, $\varphi:\mathbb{R}_{+}\rightarrow\mathbb{R}_{+}$ be a regularizer and $\lambda>0$.

(1) If $\varphi$ is increasing, then there exists a solution $f_0$ of the regularization problem \eqref{regularization} with $\mathbf{y}:=\mathbf{y}_0$ in the form
\begin{equation}\label{representer-regularization-Gateaux-explicit-formula}
f_0=\rho\mathcal{G}^*\left(\sum_{j\in\mathbb{N}_m}c_j\nu_j\right),
\end{equation}
for some $c_j\in\mathbb{R}$, $j\in\mathbb{N}_m,$ with $\rho:=\left\|\sum_{j\in\mathbb{N}_m}c_j\nu_j\right\|_{\mathcal{B}^*}$.

(2) If $\varphi$ is strictly increasing, then every solution $f_0$ of the regularization problem \eqref{regularization} with $\mathbf{y}:=\mathbf{y}_0$ has the form \eqref{representer-regularization-Gateaux-explicit-formula} for some $c_j\in\mathbb{R}$, $j\in\mathbb{N}_m.$
\end{thm}
\begin{proof}
We prove this theorem by applying Theorem \ref{representer-regularization-subdifferential-explicit}.
If the regularization problem \eqref{regularization} with $\mathbf{y}:=\mathbf{y}_0$ has a trivial solution $f_0=0$, we get equation \eqref{representer-regularization-Gateaux-explicit-formula} by choosing $c_j=0$, $j\in\mathbb{N}_m$. If the problem \eqref{regularization} does not have a trivial solution, Theorem \ref{representer-regularization-subdifferential-explicit} ensures that
there exists a nonzero solution $f_0$ of the regularization problem \eqref{regularization} satisfying equation \eqref{representer-regularization-subdifferential-explicit-formula} for some $c_j\in\mathbb{R}$, $j\in\mathbb{N}_m.$ Since $f_0\neq0,$ by equation \eqref{representer-regularization-subdifferential-explicit-formula} we get that the functional $\nu:=\sum_{j\in\mathbb{N}_m}c_j\nu_j$ is nonzero.
The smoothness of the dual space $\mathcal{B}^*$ guarantees that
\begin{equation}\label{singuton_*}
\partial\|\cdot\|_{\mathcal{B}^*}(\nu)=\{\mathcal{G}^*(\nu)\}.
\end{equation}
Upon substituting equation \eqref{singuton_*} into inclusion relation \eqref{representer-regularization-subdifferential-explicit-formula}, we may express $f_0$ in the form of equation \eqref{representer-regularization-Gateaux-explicit-formula}.

Suppose that $\varphi$ is strictly increasing. The trivial solution $f_0=0$, provided it exists, has the representation in \eqref{representer-regularization-Gateaux-explicit-formula} with $c_j=0,$ $j\in\mathbb{N}_m$. For any nontrivial solution $f_0$, Theorem \ref{representer-regularization-subdifferential-explicit} allows us to represent it in the form \eqref{representer-regularization-subdifferential-explicit-formula}. Again by using \eqref{singuton_*}, we get the explicit representation \eqref{representer-regularization-Gateaux-explicit-formula} for $f_0$.
\end{proof}

Below, we derive representer theorems for a solution of the regularization problem \eqref{regularization} in the case when the Banach space $\mathcal{B}$ has the pre-dual space $\mathcal{B}_{*}$ and $\nu_j\in\mathcal{B}_*$, $j\in\mathbb{N}_m$. We first obtain an explicit solution representation.

\begin{thm}\label{representer-regularization-subdifferential-predual}
Suppose that $\mathcal{B}$ is a Banach space having the pre-dual space $\mathcal{B}_{*}$, $\nu_j\in\mathcal{B}_*$, $j\in\mathbb{N}_m$ and $\mathcal{L}$ is defined by \eqref{functional-operator}. For a given  $\mathbf{y}_0\in\mathbb{R}^m$, let $\mathcal{Q}_{\mathbf{y}_0}:\mathbb{R}^m\rightarrow\mathbb{R}_{+}$ be a loss function, $\varphi:\mathbb{R}_{+}\rightarrow\mathbb{R}_{+}$ be a regularizer and $\lambda>0$.

(1) If $\varphi$ is increasing, then there exists a solution $f_0$ of the regularization problem \eqref{regularization} with $\mathbf{y}:=\mathbf{y}_0$ such that
\begin{equation}\label{representer-regularization-subdifferential-predual-formula}
f_0\in\gamma\partial\|\cdot\|_{\mathcal{B}_*}
\left(\sum_{j\in\mathbb{N}_m}c_j\nu_j\right),
\end{equation}
for some $c_j\in\mathbb{R}$, $j\in\mathbb{N}_m,$
with $\gamma:=\left\|\sum_{j\in\mathbb{N}_m}c_j\nu_j\right\|_{\mathcal{B}_*}$.

(2) If $\varphi$ is strictly increasing, then every solution $f_0$ of the regularization problem \eqref{regularization} with $\mathbf{y}:=\mathbf{y}_0$ satisfies \eqref{representer-regularization-subdifferential-predual-formula} for some $c_j\in\mathbb{R}$, $j\in\mathbb{N}_m$.
\end{thm}
\begin{proof}
If $\varphi$ is increasing, Theorem \ref{representer-regularization-subdifferential-explicit} ensures that there exists a solution $f_0$ of \eqref{regularization} with $\mathbf{y}:=\mathbf{y}_0$ satisfying \eqref{representer-regularization-subdifferential-explicit-formula}. As has been shown in the proof of Theorem \ref{representer-theorem-subdifferential-predual}, since $f_0\in\mathcal{B}$ and $\nu_j\in\mathcal{B}_*$, \eqref{representer-regularization-subdifferential-explicit-formula} holds if and only if \eqref{representer-regularization-subdifferential-predual-formula} holds. That is, there exists a solution $f_0$ of \eqref{regularization} with $\mathbf{y}:=\mathbf{y}_0$ satisfying \eqref{representer-regularization-subdifferential-predual-formula}. If $\varphi$ is strictly increasing, by Theorem \ref{representer-regularization-subdifferential-explicit} we get that every solution $f_0$ of \eqref{regularization} with $\mathbf{y}:=\mathbf{y}_0$ satisfying \eqref{representer-regularization-subdifferential-explicit-formula} and then \eqref{representer-regularization-subdifferential-predual-formula}.
\end{proof}

When the Banach space $\mathcal{B}$ has the smooth pre-dual space $\mathcal{B}_*$, we have a special representer theorem for a solution of the regularization problem \eqref{regularization}.

\begin{thm}\label{representer-regularization-Gateaux-predual}
Suppose that $\mathcal{B}$ is a Banach space having the smooth pre-dual space $\mathcal{B}_{*}$, $\nu_j\in\mathcal{B}_*$, $j\in\mathbb{N}_m$ and $\mathcal{L}$ is defined by \eqref{functional-operator}. For a given  $\mathbf{y}_0\in\mathbb{R}^m$, let $\mathcal{Q}_{\mathbf{y}_0}:\mathbb{R}^m\rightarrow\mathbb{R}_{+}$ be a loss function, $\varphi:\mathbb{R}_{+}\rightarrow\mathbb{R}_{+}$ be a regularizer and $\lambda>0$.

(1) If $\varphi$ is increasing, then there exists a solution $f_0$ of the regularization problem \eqref{regularization} with $\mathbf{y}:=\mathbf{y}_0$ in the form
\begin{equation}\label{representer-regularization-Gateaux-predual-formula}
f_0=\rho\mathcal{G}_*\left(\sum_{j\in\mathbb{N}_m}c_j\nu_j\right),
\end{equation}
for some $c_j\in\mathbb{R}$, $j\in\mathbb{N}_m,$ with $\rho:=\left\|\sum_{j\in\mathbb{N}_m}c_j\nu_j\right\|_{\mathcal{B}_*}$.

(2) If $\varphi$ is strictly increasing, then every solution $f_0$ of the regularization problem \eqref{regularization} with $\mathbf{y}:=\mathbf{y}_0$ has the form \eqref{representer-regularization-Gateaux-predual-formula} for some $c_j\in\mathbb{R}$, $j\in\mathbb{N}_m.$
\end{thm}
\begin{proof}
If the regularization problem \eqref{regularization} with $\mathbf{y}:=\mathbf{y}_0$ has a trivial solution $f_0=0$, equation \eqref{representer-regularization-Gateaux-predual-formula} holds for $c_j=0$, $j\in\mathbb{N}_m$. By employing Theorem \ref{representer-regularization-subdifferential-predual} and noting that
$$
\partial\|\cdot\|_{\mathcal{B}_*}
\left(\sum_{j\in\mathbb{N}_m}c_j\nu_j\right)
=\left\{\mathcal{G}_*\left(\sum_{j\in\mathbb{N}_m}c_j\nu_j\right)\right\},
$$
we get the desired conclusion for the case that the regularization problem \eqref{regularization} with $\mathbf{y}:=\mathbf{y}_0$ has no trivial solution.
\end{proof}

Observing from the above representer theorems, the essence of the representer theorems is that the original optimization problem in an infinite dimensional space can be converted to one in a finite dimensional space. This benefits from the fact that the number of data points, used in the regularization problem, is finite.

We consider below representer theorems for several special cases of Banach spaces and present special results. Moreover,
we identify the connection of our representer theorems obtained here with those that already exist in the literature.
Regularized learning was originally considered to learn a function in an RKHS from finite point-evaluation functional data, that is, $\nu_j:=\delta_{x_j},$ $j\in\mathbb{N}_m$, where $x_j, j\in\mathbb{N}_m,$ are finite points in an input set $X$. As a consequence of Theorem \ref{representer-regularization-Gateaux-explicit}, the representer theorem for the solution of the regularized learning problem in an RKHS can be obtained as follows.

\begin{cor}\label{representer-theorem-regularization-RKHS}
Suppose that $\mathcal{H}$ is an RKHS on $\mathcal{X}$ with the reproducing kernel $K$ and $x_j\in X$, $j\in\mathbb{N}_m$. If for a given $\mathbf{y}_0\in\mathbb{R}^m$, $\mathcal{Q}_{\mathbf{y}_0}$ and $\varphi$ are continuous and convex and moreover, $\varphi$ is strictly increasing and coercive, then there exists a unique solution $f_0$ of the regularization problem \eqref{regularization} with $\mathbf{y}:=\mathbf{y}_0$ and it has the form
\begin{equation}\label{representer-theorem-regularization-RKHS-formula}
f_0=\sum_{j\in\mathbb{N}_m}c_jK(x_j,\cdot),
\end{equation}
for some $c_j\in\mathbb{R},\ j\in\mathbb{N}_m.$
\end{cor}

\begin{proof}
Proposition \ref{existence1} ensures that the regularization problem \eqref{regularization} with $\mathbf{y}:=\mathbf{y}_0$ has at least one solution. The assumptions that $\mathcal{Q}_{\mathbf{y}}$ and $\varphi$ are convex and moreover, $\varphi$ is strictly increasing guarantees the uniqueness of the solution. 

We next show that $f_0$ has the form \eqref{representer-theorem-regularization-RKHS-formula}. Note that for each $j\in\mathbb{N}_m$, $K(x_j,\cdot)$ refers to a closed-form function representation of $\nu_j:=\delta_{x_j}$. Since $\varphi$ is strictly increasing, by Theorem \ref{representer-regularization-Gateaux-explicit} with $\mathcal{B}:=\mathcal{H}$ and $\nu_j:=K(x_j,\cdot),$ $j\in\mathbb{N},$ we express $f_0$ in the form
\eqref{representer-regularization-Gateaux-explicit-formula} for some $c_j\in\mathbb{R}$, $j\in\mathbb{N}_m$. Note that $\mathcal{H}_*=\mathcal{H}$. Substituting \eqref{GateauxDiff-H} with $f:=\sum_{j\in\mathbb{N}_m}c_jK(x_j,\cdot)$ into \eqref{representer-regularization-Gateaux-explicit-formula}, we get the representation \eqref{representer-theorem-regularization-RKHS-formula} of $f_0$.
\end{proof}

The well-known representer theorem for the regularization networks dates from \cite{KW70} and was generalized for non-quadratic loss functions and nondecreasing regularizers \cite{AMP,CO,SHS}. Theorem \ref{representer-regularization-Gateaux-explicit} is a generalization of the representer theorem in an RKHS.

We next consider regularization problems in a functional reproducing kernel Hilbert space (FRKHS). Motivated by learning a function from a finite number of non-point-evaluation functional data, we introduce in \cite{WX} the notion of FRKHSs. Let $\mathcal{H}$ be a Hilbert space and $\mathcal{F}$ a family of linear functionals on $\mathcal{H}$. Space $\mathcal{H}$ is called an FRKHS with respect to $\mathcal{F}$ if the norm of $\mathcal{H}$ is compatible with $\mathcal{F}$ and each linear functional in $\mathcal{F}$ is continuous on $\mathcal{H}$.  An FRKHS is expected to admit a reproducing kernel, which reproduces the linear functionals defining the space. Specifically, for an FRKHS $\mathcal{H}$ with respect to a family $\mathcal{F}:=\{\nu_{\alpha}:\alpha\in\Lambda\}$ of linear functionals, there exists a unique functional reproducing kernel $K:\Lambda\rightarrow \mathcal{H}$ such that
$K(\alpha)\in \mathcal{H}$, for all $\alpha\in \Lambda$, and
\begin{equation}\label{reproducing-property-FRKHS}
\nu_{\alpha}(f)=\langle f,K(\alpha)\rangle_{\mathcal{H}},
 \ \ \mbox{for all}\ \ f\in \mathcal{H}\ \ \mbox{and for all}\ \ \alpha\in\Lambda.
\end{equation}
The reproducing property \eqref{reproducing-property-FRKHS} shows that for each $\alpha\in\Lambda$, $K(\alpha)$ is an explicit representation for the functionals $\nu_{\alpha}$. The representer theorem for regularization problems in an FRKHS from a finite number of non-point-evaluation functional data $\nu_{\alpha_j}:=K(\alpha_j)$, $j\in\mathbb{N}_m$, can also be derived from Theorem \ref{representer-regularization-Gateaux-explicit}.

\begin{cor}\label{representer-theorem-regularization-FRKHS}
Suppose that $\mathcal{H}$ is an FRKHS with respect to the set $\mathcal{F}:=\{\nu_{\alpha}:\alpha\in\Lambda\}$ of linear functionals on $\mathcal{H}$, $K$ is the functional reproducing kernel for $\mathcal{H}$ and $\alpha_j\in\Lambda,$ $j\in\mathbb{N}_m.$ If for a given $\mathbf{y}_0\in\mathbb{R}^m$, $\mathcal{Q}_{\mathbf{y}_0}$ and $\varphi$ are continuous and convex and moreover, $\varphi$ is strictly increasing and coercive, then there exists a unique solution $f_0$ of the regularization problem \eqref{regularization} with $\mathbf{y}:=\mathbf{y}_0$ and it has the form
\begin{equation}\label{representer-theorem-regularization-FRKHS-formula}
f_0=\sum_{j\in\mathbb{N}_m}c_jK(\alpha_j).
\end{equation}
for some $c_j\in\mathbb{R},$ $j\in\mathbb{N}_m.$
\end{cor}
\begin{proof}
As pointed out in the proof of Corollary \ref{representer-theorem-regularization-RKHS}, the assumptions about $\mathcal{Q}_{\mathbf{y}_0}$ and $\varphi$ guarantee that the regularization problem \eqref{regularization} with $\mathbf{y}:=\mathbf{y}_0$ has a unique solution $f_0$. Theorem \ref{representer-regularization-Gateaux-explicit} with $\mathcal{B}:=\mathcal{H}$ and $\nu_{\alpha_j}:=K(\alpha_j), j\in\mathbb{N}_m$, allows to express $f_0$ in the form
\eqref{representer-regularization-Gateaux-explicit-formula} for some $c_j\in\mathbb{R}$, $j\in\mathbb{N}_m$. Substituting \eqref{GateauxDiff-H} with $f:=\sum_{j\in\mathbb{N}_m}c_jK(\alpha_j)$ into \eqref{representer-regularization-Gateaux-explicit-formula}, we get \eqref{representer-theorem-regularization-FRKHS-formula}.
\end{proof}

The representer theorem stated in Corollary \ref{representer-theorem-regularization-FRKHS} was also established in \cite{WX1}, where the result was proved by making use of ideas of the Tikhonov regularization.

We now turn to considering regularization problems in a uniformly Fr\'{e}chet smooth and uniformly convex Banach space $\mathcal{B}$. Recall that in such a Banach space, the semi-inner-product may be taken as a substitute of the inner product in a Hilbert space. Accordingly, each continuous linear functional on $\mathcal{B}$ can be represented by the dual element of a unique element in $\mathcal{B}$, which is defined via the semi-inner-product. In particular, the G\^{a}teaux derivative of the norm $\|\cdot\|_{\mathcal{B}}$ has the form \eqref{relation-gateaux-dual1}.

Applying Proposition \ref{representer-regularization-Gateaux} to the regularization problem \eqref{regularization} in a uniformly Fr\'{e}chet smooth and uniformly convex Banach space, we get the representer theorem as follows. Note that  in this case, the linearly independent functionals has the form $\nu_j:=g_j^{\sharp}$ for $g_j\in\mathcal{B}$, $j\in\mathbb{N}_m$.
\begin{thm}\label{representer-theorem-regularization-ufuc}
Suppose that $\mathcal{B}$ is a uniformly Fr\'{e}chet smooth and uniformly convex Banach space and $g_j\in\mathcal{B},$ $j\in\mathbb{N}$. For a given  $\mathbf{y}_0\in\mathbb{R}^m$, let $\mathcal{Q}_{\mathbf{y}_0}:\mathbb{R}^m\rightarrow\mathbb{R}_{+}$ be a loss function, $\varphi:\mathbb{R}_{+}\rightarrow\mathbb{R}_{+}$ be a regularizer and $\lambda>0$.

(1) If $\varphi$ is increasing, then there exists a solution $f_0$ of the regularization problem \eqref{regularization} with $\mathbf{y}:=\mathbf{y}_0$ such that
\begin{equation}\label{representer-theorem-regularization-formula-ufuc}
f_0^{\sharp}=\sum_{j\in\mathbb{N}_m}c_jg_j^{\sharp}
\end{equation}
for some $c_j\in\mathbb{R},\ j\in\mathbb{N}_m.$

(2) If $\varphi$ is strictly increasing, then every solution $f_0$ of the regularization problem \eqref{regularization} with $\mathbf{y}:=\mathbf{y}_0$ satisfies \eqref{representer-theorem-regularization-formula-ufuc} for some $c_j\in\mathbb{R}$, $j\in\mathbb{N}_m$.
\end{thm}
\begin{proof}
Note that the Banach space $\mathcal{B}$ is smooth. If $\varphi$ is increasing, by Proposition \ref{representer-regularization-Gateaux} with
$$
\nu_j:=g_j^{\sharp},\ \ j\in\mathbb{N}_m,
$$
we have that there exists a solution $f_0$ of the regularization problem \eqref{regularization} with $\mathbf{y}:=\mathbf{y}_0$ such that
\begin{equation}\label{representer-theorem-regularization-formula-ufuc0}
\mathcal{G}(f_0)=\sum_{j\in\mathbb{N}_m}\hat{c}_jg_j^{\sharp},
\end{equation}
for some $\hat{c}_j\in\mathbb{R}$, $j\in\mathbb{N}_m$.
Substituting \eqref{relation-gateaux-dual1} with $g:=f_0$ into \eqref{representer-theorem-regularization-formula-ufuc0} and choosing $c_j:=\|f_0\|_{\mathcal{B}}\hat{c}_j$, $j\in\mathbb{N}_m$, we obtain \eqref{representer-theorem-regularization-formula-ufuc}.

We now consider the case when $\varphi$ is strictly increasing. In this case, Proposition \ref{representer-regularization-Gateaux} ensures that every solution $f_0$ of \eqref{regularization} with $\mathbf{y}:=\mathbf{y}_0$ has the form \eqref{representer-theorem-regularization-formula-ufuc0}. Combining \eqref{relation-gateaux-dual1} with \eqref{representer-theorem-regularization-formula-ufuc0}, we obtain \eqref{representer-theorem-regularization-formula-ufuc} for $c_j:=\|f_0\|_{\mathcal{B}}\hat{c}_j$, $j\in\mathbb{N}_m$.
\end{proof}

It is desirable to have a representation for $f_0$ in addition to that for $f_0^{\sharp}$.
Since the uniformly Fr\'{e}chet smooth and uniformly convex Banach space $\mathcal{B}$ has the smooth dual space $\mathcal{B}^*$, Theorem \ref{representer-regularization-Gateaux-explicit} allows us to have a representation for $f_0$ in the uniformly Fr\'{e}chet smooth and uniformly convex Banach space $\mathcal{B}$.

\begin{thm}\label{representer-theorem-regularization-ufuc1}
Suppose that $\mathcal{B}$ is a uniformly Fr\'{e}chet smooth and uniformly convex Banach space and $g_j\in\mathcal{B},$ $j\in\mathbb{N}$. For a given  $\mathbf{y}_0\in\mathbb{R}^m$, let $\mathcal{Q}_{\mathbf{y}_0}:\mathbb{R}^m\rightarrow\mathbb{R}_{+}$ be a loss function, $\varphi:\mathbb{R}_{+}\rightarrow\mathbb{R}_{+}$ be a regularizer and $\lambda>0$.

(1) If $\varphi$ is increasing, then there exists a solution $f_0$ of the regularization problem \eqref{regularization} with $\mathbf{y}:=\mathbf{y}_0$ such that
\begin{equation}\label{representer-theorem-regularization-formula-ufuc1}
f_0=\left(\sum_{j\in\mathbb{N}_m}c_jg_j^{\sharp}\right)^{\sharp}
\end{equation}
for some $c_j\in\mathbb{R},\ j\in\mathbb{N}_m.$

(2) If $\varphi$ is strictly increasing, then every solution $f_0$ of the regularization problem \eqref{regularization} with $\mathbf{y}:=\mathbf{y}_0$ satisfies \eqref{representer-theorem-regularization-formula-ufuc1} for some $c_j\in\mathbb{R}$, $j\in\mathbb{N}_m$.
\end{thm}

\begin{proof}
Note that the uniformly Fr\'{e}chet smooth and uniformly convex Banach space $\mathcal{B}$ has the smooth dual space $\mathcal{B}^{*}$. If $\varphi$ is increasing, Theorem \ref{representer-regularization-Gateaux-explicit} with
$\nu_j:=g_j^{\sharp},\ \ j\in\mathbb{N}_m,$ shows that there exists a solution $f_0$ of the regularization problem \eqref{regularization} with $\mathbf{y}:=\mathbf{y}_0$ such that there exist $c_j\in\mathbb{R},\ j\in\mathbb{N}_m,$ satisfying
\begin{equation}\label{representer-theorem-regularization-formula-ufuc10}
f_0=\rho\mathcal{G}^*\left(\sum_{j\in\mathbb{N}_m}c_jg_j^{\sharp}\right).
\end{equation}
with $\rho:=\left\|\sum_{j\in\mathbb{N}_m}c_jg_j^{\sharp}\right\|_{\mathcal{B}^*}$. Substituting \eqref{relation-gateaux-dual1} with $\mathcal{B}$ being replaced by $\mathcal{B}^*$ and $g$ by $\sum_{j\in\mathbb{N}_m}c_jg_j^{\sharp}$ into \eqref{representer-theorem-regularization-formula-ufuc10} leads to \eqref{representer-theorem-regularization-formula-ufuc1}.

Moreover, if $\varphi$ is strictly increasing, Theorem \ref{representer-regularization-Gateaux-explicit} ensures that every solution $f_0$ of \eqref{regularization} with $\mathbf{y}:=\mathbf{y}_0$ has the form
\eqref{representer-theorem-regularization-formula-ufuc10}. Arguments similar to those presented above leads to \eqref{representer-theorem-regularization-formula-ufuc1}.
\end{proof}
A special uniformly Fr\'{e}chet smooth and uniformly convex Banach space is the space $\ell_p(\mathbb{N})$, with $1<p<+\infty$. We now apply Theorem \ref{representer-theorem-regularization-ufuc1} to the regularization problem in $\ell_p(\mathbb{N})$.

\begin{cor}\label{representer-theorem-regularization-lp}
Suppose that $\mathbf{u}_j\in \ell_q(\mathbb{N})$, $j\in\mathbb{N}_m$, with $1/p+1/q=1$, and operator $\mathcal{L}$ is defined by \eqref{functional-operator-lp}. If for a given $\mathbf{y}_0\in\mathbb{R}^m$, $\mathcal{Q}_{\mathbf{y}_0}$ and $\varphi$ are continuous and convex and moreover, $\varphi$ is strictly increasing and coercive, then there exists a unique solution $\hat{\mathbf{x}}:=(\hat x_j:j\in\mathbb{N})$ of the regularization problem \eqref{regularization} in $\ell_p(\mathbb{N})$ with $\mathbf{y}=\mathbf{y}_0$. If $\hat{\mathbf{x}}\neq0$, then
\begin{equation}\label{representer-theorem-regularization-formula-lp}
\hat{x}_j=\frac{u_j|u_j|^{q-2}}{\|\mathbf{u}\|_{q}^{q-2}},
\end{equation}
where $\mathbf{u}:=(u_j:j\in\mathbb{N})$ is defined by \eqref{definition-of-u} for some $c_j\in\mathbb{R}$, $j\in\mathbb{N}_m$.
\end{cor}
\begin{proof}
The assumptions on $\mathcal{Q}_{\mathbf{y}_0}$ and $\varphi$ with the fact that $\ell_p(\mathbb{N})$ is reflexive and strictly convex ensure that the regularization problem \eqref{regularization} in $\ell_p(\mathbb{N})$ with $\mathbf{y}=\mathbf{y}_0$ has a unique solution $\hat{\mathbf{x}}$. The space $\ell_p(\mathbb{N})$ satisfies the hypothesis of Theorem \ref{representer-theorem-regularization-ufuc1}, since it has $\ell_q(\mathbb{N})$ as its smooth pre-dual space. Hence, by Theorem \ref{representer-theorem-regularization-ufuc1} the solution $\hat{\mathbf{x}}:=(\hat{x}_j:j\in\mathbb{N})$ has the form
\begin{equation}\label{representer-theorem-regularization-formula-lp1}
\hat{\mathbf{x}}=\left(\sum_{j\in\mathbb{N}_m}c_j\mathbf{u}_j\right)^{\sharp},
\end{equation}
for some $c_j\in\mathbb{R}$, $j\in\mathbb{N}_m$. Let $\mathbf{u}:=(u_j:j\in\mathbb{N})$ be defined by \eqref{definition-of-u}. Equation \eqref{representer-theorem-regularization-formula-lp1} can be represented as $\hat{\mathbf{x}}=\mathbf{u}^{\sharp}$. If $\hat{\mathbf{x}}\neq0$, then
$\mathbf{u}\neq0$. Substituting
$$
\mathbf{u}^{\sharp}
=\left(\frac{u_j|u_j|^{q-2}}{\|\mathbf{u}\|_{q}^{q-2}}:j\in\mathbb{N}\right)
$$
into \eqref{representer-theorem-regularization-formula-lp1} leads to the representation  \eqref{representer-theorem-regularization-formula-lp} of $\hat{\mathbf{x}}:=(\hat x_j:j\in\mathbb{N})$.
\end{proof}

The semi-inner-product RKBS is uniformly Fr\'{e}chet smooth and uniformly convex. A regularization problem in such a space is considered to learn a function from the sample data produced by point-evaluation functionals $\nu_j:=\delta_{x_j},$ $j\in\mathbb{N}_m$, where $x_j, j\in\mathbb{N}_m,$ are a finite number of points in an input set $X$. By the reproducing property \eqref{reproducing-property-ZXZ}, the dual element $G(x_j,\cdot)^{\sharp}$ of $G(x_j,\cdot)$ coincides exactly with the point-evaluation functional $\delta_{x_j}$ for $j\in\mathbb{N}_m$.
Applying Theorem \ref{representer-theorem-regularization-ufuc} to the regularization problem \eqref{regularization} in a semi-inner-product RKBS, we get the representer theorem as follows.

\begin{cor}\label{representer-theorem-regularization-ZXZ}
Suppose that $\mathcal{B}$ is the semi-inner-product RKBS with the semi-inner-product reproducing kernel $G$ and $x_j\in X$, $j\in\mathbb{N}_m$. For a given  $\mathbf{y}_0\in\mathbb{R}^m$, let $\mathcal{Q}_{\mathbf{y}_0}:\mathbb{R}^m\rightarrow\mathbb{R}_{+}$ be a loss function, $\varphi:\mathbb{R}_{+}\rightarrow\mathbb{R}_{+}$ be a regularizer and $\lambda>0$.

(1) If $\varphi$ is increasing, then there exists a solution $f_0$ of the regularization problem \eqref{regularization} with $\mathbf{y}:=\mathbf{y}_0$ such that
\begin{equation}\label{representer-theorem-regularization-formula-ZXZ}
f_0^{\sharp}=\sum_{j\in\mathbb{N}_m}c_jG(x_j,\cdot)^{\sharp}
\end{equation}
for some $c_j\in\mathbb{R},\ j\in\mathbb{N}_m.$

(2) If $\varphi$ is strictly increasing, then every solution $f_0$ of the regularization problem \eqref{regularization} with $\mathbf{y}:=\mathbf{y}_0$ satisfies \eqref{representer-theorem-regularization-formula-ZXZ} for some $c_j\in\mathbb{R}$, $j\in\mathbb{N}_m$.
\end{cor}
\begin{proof}
Note that the semi-inner-product RKBS $\mathcal{B}$ satisfies the hypothesis of Theorem \ref{representer-theorem-regularization-ufuc}. Thus, by employing Theorem \ref{representer-theorem-regularization-ufuc} with $g_j:=G(x_j,\cdot)$, $j\in\mathbb{N}_m$, we get the desired result.
\end{proof}

The representer theorem stated in Corollary \ref{representer-theorem-regularization-ZXZ} was observed in \cite{ZXZ,ZZ}. An explicit representation for the solution of the regularization problem in the semi-inner-product RKBS $\mathcal{B}$ can also be obtained.

\begin{thm}\label{representer-theorem-regularization-ZXZ1}
Suppose that $\mathcal{B}$ is the semi-inner-product RKBS with the semi-inner-product reproducing kernel $G$ and $x_j\in X$, $j\in\mathbb{N}_m$. For a given  $\mathbf{y}_0\in\mathbb{R}^m$, let $\mathcal{Q}_{\mathbf{y}_0}:\mathbb{R}^m\rightarrow\mathbb{R}_{+}$ be a loss function, $\varphi:\mathbb{R}_{+}\rightarrow\mathbb{R}_{+}$ be a regularizer and $\lambda>0$.

(1) If $\varphi$ is increasing, then there exists a solution $f_0$ of the regularization problem \eqref{regularization} with $\mathbf{y}:=\mathbf{y}_0$ such that
\begin{equation}\label{representer-theorem-regularization-formula-ZXZ1}
f_0=\left(\sum_{j\in\mathbb{N}_m}c_jG(x_j,\cdot)^{\sharp}\right)^{\sharp}
\end{equation}
for some $c_j\in\mathbb{R},\ j\in\mathbb{N}_m.$

(2) If $\varphi$ is strictly increasing, then every solution $f_0$ of the regularization problem \eqref{regularization} with $\mathbf{y}:=\mathbf{y}_0$ satisfies \eqref{representer-theorem-regularization-formula-ZXZ1} for some $c_j\in\mathbb{R}$, $j\in\mathbb{N}_m$.
\end{thm}

\begin{proof}
The desired result is an immediate consequence of Theorem \ref{representer-theorem-regularization-ufuc1} with $g_j:=G(x_j,\cdot)$, for $j\in\mathbb{N}_m$.
\end{proof}

We turn to the regularization problem in a right-sided RKBS $\mathcal{B}$ with the right-sided reproducing kernel $K$. As a consequence of Proposition \ref{representer-regularization-Gateaux}, we have the following representer theorem.

\begin{cor}\label{representer-theorem-regularization-XY}
Let $\mathcal{B}$ be a right-sided RKBS with a right-sided reproducing kernel $K$ and $x_j\in X$, $j\in\mathbb{N}_m$. Suppose that $\mathcal{B}$ is reflexive, strictly convex and smooth. If for a given $\mathbf{y}_0\in\mathbb{R}^m$, $\mathcal{Q}_{\mathbf{y}_0}$ and $\varphi$ are continuous and convex and moreover, $\varphi$ is strictly increasing and coercive, then the regularization problem \eqref{regularization} with $\mathbf{y}:=\mathbf{y}_0$ has a unique solution $f_0$ and it has the form
\begin{equation}\label{representer-theorem-regularization-formula-XY}
\mathcal{G}(f_0)=\sum_{j\in\mathbb{N}_m}c_jK(x_j,\cdot),
\end{equation}
for some $c_j\in\mathbb{R},\ j\in\mathbb{N}_m.$
\end{cor}
\begin{proof}
Since $\mathcal{B}$ is reflexive, both $\mathcal{Q}_{\mathbf{y}_0}$ and $\varphi$ are continuous and $\varphi$ is strictly increasing and coercive, Corollary \ref{existence2} guarantees the existence of the solutions of \eqref{regularization} with $\mathbf{y}:=\mathbf{y}_0$. Moreover, the assumptions that $\mathcal{B}$ is strictly convex, $\mathcal{Q}_{\mathbf{y}_0}$ and $\varphi$ are convex and $\varphi$ is strictly increasing ensure the uniqueness of the solution \eqref{regularization} with $\mathbf{y}:=\mathbf{y}_0$. Hence, we conclude that the regularization problem \eqref{regularization} with $\mathbf{y}:=\mathbf{y}_0$ has a unique solution $f_0$.

By Proposition \ref{representer-regularization-Gateaux}, we may express $f_0$ as in  \eqref{representer-regularization-Gateaux-formula}. Note that the right-sided reproducing kernel $K$ provides the closed-from function representation for the point-evaluation functionals. Substituting $\nu_j:=K(x_j,\cdot),$ $j\in\mathbb{N}_m$, into \eqref{representer-regularization-Gateaux-formula}, we get \eqref{representer-theorem-regularization-formula-XY}.
\end{proof}

We note that the representer theorem described in Corollary \ref{representer-theorem-regularization-XY} was established in \cite{XY} by a different method from ours. In fact, an explicit reprsentation for $f_0$ can be obtained as follows.

\begin{thm}\label{representer-theorem-regularization-XY1}
Let $\mathcal{B}$ be a right-sided RKBS with the right-sided reproducing kernel $K$ and $x_j\in X$, $j\in\mathbb{N}_m$. Suppose that $\mathcal{B}$ is reflexive and strictly convex. If for a given $\mathbf{y}_0\in\mathbb{R}^m$, $\mathcal{Q}_{\mathbf{y}_0}$ and $\varphi$ are continuous and convex and moreover, $\varphi$ is strictly increasing and coercive, then the regularization problem \eqref{regularization} with $\mathbf{y}:=\mathbf{y}_0$ has a unique solution $f_0$ and it has the form
\begin{equation*}\label{representer-theorem-regularization-formula-XY1}
f_0=\sigma\mathcal{G}^*
\left(\sum_{j\in\mathbb{N}_m}c_jK(x_j,\cdot)\right),
\end{equation*}
for some $c_j\in\mathbb{R},\ j\in\mathbb{N}_m,$ where $\sigma:=\left\|\sum_{j\in\mathbb{N}_m}c_jK(x_j,\cdot)\right\|_{\mathcal{B}^*}$.
\end{thm}
\begin{proof}
As pointed out in the proof of Corollary \ref{representer-theorem-regularization-XY}, the regularization problem \eqref{regularization} with $\mathbf{y}:=\mathbf{y}_0$ has a unique solution $f_0$. It is known that the reflexive Banach space $\mathcal{B}$ is strict convexity if and only if the dual space $\mathcal{B}^*$ is smooth. Thus, the hypotheses of Theorem \ref{representer-regularization-Gateaux-explicit} are satisfied. Hence, by Theorem \ref{representer-regularization-Gateaux-explicit} with $\nu_j:=K(x_j,\cdot),$ $j\in\mathbb{N}_m,$ we get the desired conclusion.
\end{proof}

Finally, we specialize Theorem \ref{representer-regularization-subdifferential-predual} to the regularization problem in the space $\ell_1(\mathbb{N})$.

\begin{thm}\label{representer-theorem-regularization-l1}
Suppose that $\mathbf{u}_j\in c_0$, $j\in\mathbb{N}_m$, and operator $\mathcal{L}$ is defined by \eqref{functional-operator-l1}. For a given $\mathbf{y}_0\in\mathbb{R}^m$, let $\mathcal{Q}_{\mathbf{y}_0}:\mathbb{R}^m\rightarrow\mathbb{R}_{+}$ be a loss function, $\varphi:\mathbb{R}_{+}\rightarrow\mathbb{R}_{+}$ be a regularizer and $\lambda>0$.

(1) If $\varphi$ is increasing, then there exists a solution $\hat{\mathbf{x}}\in\ell_1$ of the regularization problem \eqref{regularization} in $\ell_1(\mathbb{N})$ with $\mathbf{y}:=\mathbf{y}_0$ such that
\begin{equation}\label{representer-theorem-regularization-formula-l1}
\hat{\mathbf{x}}\in\|\mathbf{u}\|_{\infty}\mathrm{co}(\mathcal{V}
(\mathbf{u})),
\end{equation}
for some $c_j\in\mathbb{R}$, $j\in\mathbb{N}_m,$ where $\mathbf{u}:=\sum_{j\in\mathbb{N}_m}c_j\mathbf{u}_j.$

(2) If $\varphi$ is strictly increasing, then every solution $\hat{\mathbf{x}}$ of the regularization problem \eqref{regularization} in $\ell_1(\mathbb{N})$ with $\mathbf{y}:=\mathbf{y}_0$ satisfies \eqref{representer-theorem-regularization-formula-l1} for some $c_j\in\mathbb{R}$, $j\in\mathbb{N}_m.$
\end{thm}

\begin{proof}
Suppose that $\varphi$ is increasing. If the regularization problem \eqref{regularization} in $\ell_1(\mathbb{N})$ with $\mathbf{y}:=\mathbf{y}_0$ has $\hat{\mathbf{x}}=0$ as a solution, then the trival solution has the form \eqref{representer-theorem-regularization-formula-l1} for some $c_j=0$, $j\in\mathbb{N}_m.$ We consider the case that the regularization problem \eqref{regularization} in $\ell_1(\mathbb{N})$ with $\mathbf{y}:=\mathbf{y}_0$ has no trival solution.
Since $\ell_1(\mathbb{N})$ has $c_0$ as its pre-dual space, by Theorem \ref{representer-regularization-subdifferential-predual} with $\nu_j:=\mathbf{u}_j$, $j\in\mathbb{N}_m,$ there exists a nonzero solution $\hat{\mathbf{x}}$ of \eqref{regularization} such that
\begin{equation}\label{subdifferential-l1}
\hat{\mathbf{x}}\in\|\mathbf{u}\|_{\infty}\partial\|\cdot\|_{\infty}(\mathbf{u}),
\end{equation}
for some $c_j\in\mathbb{R}$, $j\in\mathbb{N}_m,$ where $\mathbf{u}:=\sum_{j\in\mathbb{N}_m}c_j\mathbf{u}_j$. Substituting equation \eqref{subdifferentials-formula-c0} of Lemma \ref{subdifferentials-c0} into above equation we get \eqref{representer-theorem-regularization-formula-l1}.

If $\varphi$ is strictly increasing, the trival solution $\hat{\mathbf{x}}=0$, provided its existence, has the form \eqref{representer-theorem-regularization-formula-l1} for some $c_j=0$, $j\in\mathbb{N}_m.$
Moreover, Theorem \ref{representer-regularization-subdifferential-predual} ensures that any nontrival solution satisfies \eqref{subdifferential-l1}. This together with
\eqref{subdifferentials-formula-c0} completes the proof of this theorem.
\end{proof}

The representer theorems established in this subsection will serve as a theoretical foundation for solution methods to be developed in the next subsection for solving the regularization problem.

\subsection{Solution Methods for Regularization Problems}

The representer theorems presented in the last subsection for the regularization problem give only forms of the solutions for the problem, the same as those for the minimum morn interpolation developed in section 3. To provide solution methods for the regularization problem, one has to determine the coefficients $c_j$ involved in the solution representations. We develop in this subsection approaches to determine these coefficients, leading to solution methods for solving the regularization problem. We also consider in the next subsection solving directly the regularization problem \eqref{regularization} by the fixed-point approach, which has been discussed in section 6 for the minimum norm interpolation problem.

We derive in this subsection solution methods based on the representer theorems for the regularization problem.

We begin with considering the case that the Banach space $\mathcal{B}$ has the smooth pre-dual space $\mathcal{B}_{*}$ and $\nu_j\in\mathcal{B}_*$, for $j\in\mathbb{N}_m$. In this case, Theorem \ref{representer-regularization-Gateaux-predual} provides a simple and explicit representation for the solutions of the regularization problem \eqref{regularization}. By employing this solution representation, the regularization problem \eqref{regularization} can be converted to a finite dimensional minimization problem about the coefficients appearing in the representation.

\begin{thm}\label{solution-regu-smooth-B_*}
Suppose that $\mathcal{B}$ is a Banach space having the smooth pre-dual space $\mathcal{B}_{*}$, and $\nu_j\in\mathcal{B}_*$, $j\in\mathbb{N}_m$. Let $\mathcal{L}$ be the linear operator defined by \eqref{functional-operator} and $\mathcal{L}^*$ be the adjoint operator. For a given  $\mathbf{y}_0\in\mathbb{R}^m$, let $\mathcal{Q}_{\mathbf{y}_0}:\mathbb{R}^m\rightarrow\mathbb{R}_{+}$ be a loss function, $\varphi:\mathbb{R}_{+}\rightarrow\mathbb{R}_{+}$ be a strictly increasing regularizer and $\lambda>0$. Then
\begin{equation}\label{rt-Gateaux-predual-solve-regularization}
f_0:=\|\mathcal{L}^*(\hat{\mathbf{c}})\|_{\mathcal{B}_*}
\mathcal{G}_*(\mathcal{L}^*(\hat{\mathbf{c}})), \ \ \mathbf{\hat c}\in\mathbb{R}^m,
\end{equation}
is a solution of the regularization problem \eqref{regularization} with $\mathbf{y}:=\mathbf{y}_0$ if and only if $\hat{\mathbf{c}}\in\mathbb{R}^m$ is a solution of the minimization problem
\begin{equation}\label{optimization-c-smooth-B_*}
\inf\{\mathcal{Q}_{\mathbf{y}_0}(\|\mathcal{L}^*(\mathbf{c})\|_{\mathcal{B}_*}
\mathcal{L}(\mathcal{G}_*(\mathcal{L}^*(\mathbf{c}))))
+\lambda\varphi(\|\mathcal{L}^*(\mathbf{c})\|_{\mathcal{B}_*})
:\mathbf{c}\in\mathbb{R}^m\}.
\end{equation}
\end{thm}
\begin{proof}
Since $\varphi$ is strictly increasing, Theorem \ref{representer-regularization-Gateaux-predual} ensures that every solution $f_0$ of the regularization problem \eqref{regularization} with $\mathbf{y}:=\mathbf{y}_0$ has the form \eqref{rt-Gateaux-predual-solve-regularization}, for some $\mathbf{\hat c}\in\mathbb{R}^m$. It suffices to show that $f_0$ in the form \eqref{rt-Gateaux-predual-solve-regularization} is a solution of \eqref{regularization} if and only if $\mathbf{\hat c}$ is a solution of the minimization problem \eqref{optimization-c-smooth-B_*}. To this end, we define a subset $\mathcal{A}$ of $\mathcal{B}$ by
$$
\mathcal{A}:=\{f\in\mathcal{B}: f=\|\mathcal{L}^*(\mathbf{c})\|_{\mathcal{B}_*}
\mathcal{G}_*(\mathcal{L}^*(\mathbf{c})), \mathbf{c}\in\mathbb{R}^m\}.
$$
Clearly, the regularization problem \eqref{regularization} with $\mathbf{y}:=\mathbf{y}_0$ is equivalent to
\begin{equation}\label{optimization-c-smooth-B_*-1}
\inf\{\mathcal{Q}_{\mathbf{y}_0}(\mathcal{L}(f))+\lambda\varphi(\|f\|_{\mathcal{B}}):f\in\mathcal{A}\}.
\end{equation}
Note that each $f\in\mathcal{A}$ has the form
\begin{equation}\label{special-representation-of-f}
f:=\|\mathcal{L}^*(\mathbf{c})\|_{\mathcal{B}_*}
\mathcal{G}_*(\mathcal{L}^*(\mathbf{c})).
\end{equation}
By equation \eqref{norm-GateauxDiff} we have that $\|\mathcal{G}_*(\mathcal{L}^*(\mathbf{c}))\|_{\mathcal{B}}=1$ and then
\begin{equation}\label{B-norm-of-f}
\|f\|_{\mathcal{B}}=\|\mathcal{L}^*(\mathbf{c})\|_{\mathcal{B}_*}.
\end{equation}
Substituting the representation \eqref{special-representation-of-f} of $f\in\mathcal{A}$ and the norm representation \eqref{B-norm-of-f} into the first term and second term of the objective function of the minimization problem \eqref{optimization-c-smooth-B_*-1}, respectively, we observe that the minimization problem \eqref{optimization-c-smooth-B_*-1} is equivalent to \eqref{optimization-c-smooth-B_*}, proving the desired result.
\end{proof}

As a special case, we consider the regularization problem \eqref{regularization} in a Hilbert space $\mathcal{H}$, that is  $\mathcal{B}:=\mathcal{H}$. In this case, the regularizer $\varphi$ has the form $\varphi(t):=t^2,$ $t\in\mathbb{R}_+$, which is strictly increasing on $\mathbb{R}_+$, $\mathcal{B}_*=\mathcal{H}$ and the linearly independent functionals $\nu_j$, $j\in\mathbb{N}_m$, are identified with $g_j\in\mathcal{H}$.

\begin{cor}\label{solution-regu-H}
Suppose that $\mathcal{H}$ is a Hilbert space and $g_j\in\mathcal{H}$, $j\in\mathbb{N}_m$. Let $\mathcal{L}$ be the linear operator defined by \eqref{functional-operator} with $\nu_j:=g_j$, $j\in\mathbb{N}_m$, $\mathcal{L}^*$ be the adjoint operator and $\mathbf{G}$ be the Gram matrix defined by \eqref{Gram-matrix}. For a given $\mathbf{y}_0\in\mathbb{R}^m$, let $\mathcal{Q}_{\mathbf{y}_0}:\mathbb{R}^m\rightarrow\mathbb{R}_{+}$ be a loss function and $\lambda>0$. Then
\begin{equation}\label{rt-regu-H-solve}
f_0:=\sum_{j\in\mathbb{N}_m}\hat{c}_jg_j
\end{equation}
is a solution of the regularization problem \eqref{regularization} with $\mathbf{y}:=\mathbf{y}_0$
if and only if $\hat{\mathbf{c}}:=[\hat{c}_j:j\in\mathbb{N}_m]\in\mathbb{R}^m$ is a solution of the minimization problem
\begin{equation}\label{optimization-c-Hilbert}
\inf\{\mathcal{Q}_{\mathbf{y}_0}(\mathbf{G}\mathbf{c})+
\lambda\mathbf{c}^{\top}\mathbf{G}\mathbf{c}
:\mathbf{c}\in\mathbb{R}^m\}.
\end{equation}
\end{cor}
\begin{proof}
We conclude from Theorem \ref{solution-regu-smooth-B_*} that $f_0$ in the form \eqref{rt-Gateaux-predual-solve-regularization} is a solution of the regularization problem \eqref{regularization} with $\mathbf{y}:=\mathbf{y}_0$ if and only if $\hat{\mathbf{c}}$ is a solution of the minimization problem \eqref{optimization-c-smooth-B_*}. It surfaces to represent $f_0$ in the form \eqref{rt-regu-H-solve} and to reformulate the minimization problem \eqref{optimization-c-smooth-B_*} in the form \eqref{optimization-c-Hilbert}. Note that $\mathcal{H}_*=\mathcal{H}$. According to equation \eqref{GateauxDiff-H} we get that
\begin{equation}\label{optimization-c-Hilbert1}
\|\mathcal{L}^*(\mathbf{c})\|_{\mathcal{H}}
\mathcal{G}(\mathcal{L}^*(\mathbf{c}))=\mathcal{L}^*(\mathbf{c}),
\ \ \mbox{for all}\ \ \mathbf{c}\in\mathbb{R}^m.
\end{equation}
Substituting equation \eqref{optimization-c-Hilbert1} with $\mathbf{c}:=\hat{\mathbf{c}}$ and the representation \eqref{adjoint-operator} of $\mathcal{L}^*$ into \eqref{rt-Gateaux-predual-solve-regularization}, we get the form \eqref{rt-regu-H-solve} of $f_0$. Again by equation \eqref{optimization-c-Hilbert1}, we rewrite \eqref{optimization-c-smooth-B_*} as
$$
\inf\{\mathcal{Q}_{\mathbf{y}_0}(\mathcal{L}\mathcal{L}^*(\mathbf{c}))
+\lambda\|\mathcal{L}^*(\mathbf{c})\|_{\mathcal{H}}^2
:\mathbf{c}\in\mathbb{R}^m\}.
$$
According to the definition of the Gram matrix $\mathbf{G}$, we have that
$$
\mathcal{L}\mathcal{L}^*(\mathbf{c})=\mathbf{G}\mathbf{c}
\ \ \mbox{and}\ \
\|\mathcal{L}^*(\mathbf{c})\|_{\mathcal{H}}^2=
\mathbf{c}^{\top}\mathbf{G}\mathbf{c},
\ \ \mbox{for all}\ \ \mathbf{c}\in\mathbb{R}^m.
$$
Substituting these equations into the above minimization problem leads to \eqref{optimization-c-Hilbert}.
\end{proof}

Below, we discuss how to solve the finite dimensional minimization problem \eqref{optimization-c-Hilbert}.
Approaches that may be adopted to solve the problem \eqref{optimization-c-Hilbert} depend on the smoothness of the function $\mathcal{Q}_{\mathbf{y}_0}$ appearing in the fidelity term of \eqref{optimization-c-Hilbert}. When $\mathcal{Q}_{\mathbf{y}_0}$ is differentiable, the minimization problem \eqref{optimization-c-Hilbert} may be solved by using standard methods such as the gradient descent method. When $\mathcal{Q}_{\mathbf{y}_0}$ is not differentiable, standard methods for solving minimization problems are not applicable to problem \eqref{optimization-c-Hilbert} and it requires special treatment. We will pay a special attention to the case when $\mathcal{Q}_{\mathbf{y}_0}$ is not differentiable.

We now consider solving the finite dimensional minimization problem \eqref{optimization-c-Hilbert}.

\begin{rem}
For a given $\mathbf{y}_0\in\mathbb{R}^m$, let $\mathcal{Q}_{\mathbf{y}_0}$ be a convex loss function. If $\mathcal{Q}_{\mathbf{y}_0}$ is non-differentiable, then $\hat{\mathbf{c}}\in\mathbb{R}^m$ is the unique solution of the minimization problem
\eqref{optimization-c-Hilbert} if and only if $\hat{\mathbf{c}}$ satisfies
\begin{equation}\label{fixed-point-H}
\hat{\mathbf{c}}=\frac{1}{-2\lambda}\mathrm{prox}_{\mathcal{Q}_{\mathbf{y}_0}^*}
(-2\lambda\hat{\mathbf{c}}+\mathbf{G}\hat{\mathbf{c}}).
\end{equation}
\end{rem}
\begin{proof}
Note that the assumption that $\mathcal{Q}_{\mathbf{y}_0}$ is convex ensures the uniqueness of the solution of the minimization problem \eqref{optimization-c-Hilbert}. Because of the linear independence of $g_j\in\mathcal{H}$, $j\in\mathbb{N}_m$, the Gram matrix $\mathbf{G}$ is symmetric and positive definite. Then by the Fermat rule and the chain rule \eqref{chain-rule}, $\hat{\mathbf{c}}$ is the solution of \eqref{optimization-c-Hilbert} if and only if
$$
0\in\mathbf{G}\partial\mathcal{Q}_{\mathbf{y}_0}(\mathbf{G}\hat{\mathbf{c}})
+2\lambda\mathbf{G}\hat{\mathbf{c}},
$$
which is equivalent to
\begin{equation}\label{optimization-c-Hilbert-inclusion}
-2\lambda\hat{\mathbf{c}}\in\partial
\mathcal{Q}_{\mathbf{y}_0}(\mathbf{G}\hat{\mathbf{c}}).
\end{equation}
If $\mathcal{Q}_{\mathbf{y}_0}$ is non-differentiable, we can characterize the solution of \eqref{optimization-c-Hilbert} via a fixed-point equation. According to \eqref{conjugate}, the inclusion relation \eqref{optimization-c-Hilbert-inclusion} holds if and only if
$$
\mathbf{G}\hat{\mathbf{c}}\in\partial
\mathcal{Q}_{\mathbf{y}_0}^*(-2\lambda\hat{\mathbf{c}}).
$$
Hence, by Lemma \ref{subdiff-prox-Rm} we obtain the equivalence between this inclusion relation and the fixed-point equation \eqref{fixed-point-H}.
\end{proof}

In the case that the loss function is differentiable, the solution of the finite dimensional minimization problem \eqref{optimization-c-Hilbert} satisfies a system which usually is nonlinear. For the loss function with a special form, the nonlinear system reduces to a linear one.

\begin{rem}
If the convex function $\mathcal{Q}_{\mathbf{y}_0}$ is differentiable, then $\hat{\mathbf{c}}$ is the unique solution of the minimization problem \eqref{optimization-c-Hilbert} if and only if $\hat{\mathbf{c}}$ satisfies the system
\begin{equation}\label{system-regu-H}
-2\lambda\hat{\mathbf{c}}=
\nabla\mathcal{Q}_{\mathbf{y}_0}(\mathbf{G}\hat{\mathbf{c}}).
\end{equation}
Particularly, if $\mathcal{Q}_{\mathbf{y}_0}$ has the form \eqref{RN}, then the system \eqref{system-regu-H} reduces to the linear system
\begin{equation}\label{system-regu-RN}
(\mathbf{G}+\lambda\mathbf{I})\hat{\mathbf{c}}=\mathbf{y}_0.
\end{equation}
\end{rem}
\begin{proof}
Note that $\hat{\mathbf{c}}$ is the solution of \eqref{optimization-c-Hilbert} if and only if it satisfies \eqref{optimization-c-Hilbert-inclusion}. If $\mathcal{Q}_{\mathbf{y}_0}$ is differentiable, then we have that
$$
\mathcal{Q}_{\mathbf{y}_0}(\mathbf{G}\hat{\mathbf{c}})
=\{\nabla\mathcal{Q}_{\mathbf{y}_0}(\mathbf{G}\hat{\mathbf{c}})\}.
$$
Substituting the above equation into \eqref{optimization-c-Hilbert-inclusion}, we obtain the system \eqref{system-regu-H}. If $\mathcal{Q}_{\mathbf{y}_0}$ has the form \eqref{RN}, then there holds
$$
\nabla\mathcal{Q}_{\mathbf{y}_0}(\mathbf{G}\hat{\mathbf{c}})
=2(\mathbf{G}\hat{\mathbf{c}}-\mathbf{y}_0),
$$
which together with \eqref{system-regu-H} leads to the linear system \eqref{system-regu-RN}.
\end{proof}

Our second example concerns the regularization problem in a uniformly Fr\'{e}chet smooth and uniformly convex Banach space $\mathcal{B}$. In such a space, there exists a unique semi-inner-product $[\cdot,\cdot]_{\mathcal{B}}$ that induces the norm $\|\cdot\|_{\mathcal{B}}$. Moreover, for each $\nu\in\mathcal{B}^*$, there exists a unique $g\in\mathcal{B}$ such that
$\nu=g^{\sharp}$. Thus, in this case, the linear functional $\nu_j$ appearing in  \eqref{regularization} is identified with $g_j^{\sharp}$, for $g_j\in\mathcal{B}$, $j\in\mathbb{N}_m$. With respect to the sequence $g_j\in\mathcal{B}$, $j\in\mathbb{N}_m$, we introduce a nonlinear operator $\mathbf{G}_{\mathrm{s.i.p}}$ from $\mathbb{R}^m$ to itself. Specifically, set
\begin{equation}\label{Gram-semi-inner}
\mathbf{G}_{\mathrm{s.i.p}}(\mathbf{c})
:=\Bigg[\bigg[g_j^{\sharp},\sum_{k\in\mathbb{N}_m}
c_kg_k^{\sharp}\bigg]_{\mathcal{B}^*}:j\in\mathbb{N}_m\Bigg],
\ \ \mbox{for all}\ \ \mathbf{c}:=[c_k:k\in\mathbb{N}_m]\in\mathbb{R}^m.
\end{equation}
Here, $[\cdot,\cdot]_{\mathcal{B}^*}$ is the semi-inner-product of $\mathcal{B}^*$ defined by \eqref{semi-inner-product-B^*}.

\begin{cor}\label{solution-regu-ufuc}
Suppose that $\mathcal{B}$ is a uniformly Fr\'{e}chet smooth and uniformly convex Banach space
and $g_j\in\mathcal{B}$, $j\in\mathbb{N}_m$. Let $\mathcal{L}$ be the linear operator defined by \eqref{functional-operator} with $\nu_j:=g_j^{\sharp}$, $j\in\mathbb{N}_m$, $\mathcal{L}^*$ be the adjoint operator of $\mathcal{L}$ and $\mathbf{G}_{\mathrm{s.i.p}}$ be the operator defined by \eqref{Gram-semi-inner}. For a given  $\mathbf{y}_0\in\mathbb{R}^m$, let $\mathcal{Q}_{\mathbf{y}_0}:\mathbb{R}^m\rightarrow\mathbb{R}_{+}$ be a loss function, $\varphi:\mathbb{R}_{+}\rightarrow\mathbb{R}_{+}$ be a strictly increasing regularizer and $\lambda>0$. Then
\begin{equation}\label{rt-regu-ufuc-solve}
f_0:=\left(\sum_{j\in\mathbb{N}_m}\hat{c}_jg_j^{\sharp}\right)^{\sharp},
\end{equation}
is a solution of the regularization problem \eqref{regularization} with $\mathbf{y}:=\mathbf{y}_0$
if and only if $\hat{\mathbf{c}}:=[\hat{c}_j:j\in\mathbb{N}_m]\in\mathbb{R}^m$ is a solution of the minimization problem
\begin{equation}\label{optimization-c-ufuc}
\inf\left\{\mathcal{Q}_{\mathbf{y}_0}
(\mathbf{G}_{\mathrm{s.i.p}}(\mathbf{c}))
+\lambda\varphi\left((\mathbf{c}^{\top}\mathbf{G}_{\mathrm{s.i.p}}(\mathbf{c}))^{1/2}\right)
:\mathbf{c}\in\mathbb{R}^m\right\}.
\end{equation}
\end{cor}
\begin{proof}
Theorem \ref{solution-regu-smooth-B_*} ensures that $f_0$ with the form \eqref{rt-Gateaux-predual-solve-regularization} is a solution of the regularization problem \eqref{regularization} with $\mathbf{y}:=\mathbf{y}_0$ if and only if $\hat{\mathbf{c}}$ is a solution of the minimization problem \eqref{optimization-c-smooth-B_*}. By making use of the semi-inner-product, we will represent $f_0$ in the form \eqref{rt-regu-ufuc-solve} and the minimization problem \eqref{optimization-c-smooth-B_*} in the form \eqref{optimization-c-ufuc}.
Note that for the uniformly Fr\'{e}chet smooth and uniformly convex Banach space $\mathcal{B}$,  its dual space $\mathcal{B}^*$ is identified with the pre-dual space $\mathcal{B}_*$. Substituting \eqref{relation-gateaux-dual1} with $\mathcal{B}$ being replaced by $\mathcal{B}^*$ and $g$ by $\sum_{j\in\mathbb{N}_m}\hat{c}_jg_j^{\sharp}$ into \eqref{rt-Gateaux-predual-solve-regularization}, $f_0$ may be rewritten as \eqref{rt-regu-ufuc-solve}. Again by \eqref{relation-gateaux-dual1},
we rewrite \eqref{optimization-c-smooth-B_*} as
\begin{equation}\label{optimization-c-ufuc1}
\inf\{\mathcal{Q}_{\mathbf{y}_0}(\mathcal{L}(\mathcal{L}^*(\mathbf{c}))^{\sharp})
+\lambda\varphi(\|\mathcal{L}^*(\mathbf{c})\|_{\mathcal{B}^*})
:\mathbf{c}\in\mathbb{R}^m\}.
\end{equation}
It follows from \eqref{dual-element-B^*} that
$$
\left\langle g_j^{\sharp},(\mathcal{L}^*(\mathbf{c}))^{\sharp}
\right\rangle_{\mathcal{B}}
=\left[g_j^{\sharp},\mathcal{L}^*(\mathbf{c})\right]_{\mathcal{B}^*},
\ \ \mbox{for all}\ \ j\in\mathbb{N}_m,
$$
which together with the representations of $\mathcal{L}^*$ and $\mathbf{G}_{\mathrm{s.i.p}}$ leads to
\begin{equation}\label{optimization-c-ufuc2}
\mathcal{L}(\mathcal{L}^*(\mathbf{c}))^{\sharp}
=\mathbf{G}_{\mathrm{s.i.p}}(\mathbf{c}).
\end{equation}
There holds for all $\mathbf{c}\in\mathbb{R}^m$ that
\begin{equation*}
\|\mathcal{L}^*(\mathbf{c})\|_{\mathcal{B}^*}^2
=[\mathcal{L}^*(\mathbf{c}),\mathcal{L}^*(\mathbf{c})]_{\mathcal{B}^*}
=\sum_{j\in\mathbb{N}_m}c_j\left[g_j^{\sharp},\mathcal{L}^*(\mathbf{c})\right]_{\mathcal{B}^*}.
\end{equation*}
By the definition of the nonlinear operator $\mathbf{G}_{\mathrm{s.i.p}}$, the above equation leads to
\begin{equation}\label{optimization-c-ufuc3}
\|\mathcal{L}^*(\mathbf{c})\|_{\mathcal{B}^*}^2
=\mathbf{c}^{\top}\mathbf{G}_{\mathrm{s.i.p}}(\mathbf{c}).
\end{equation}
Substituting equations \eqref{optimization-c-ufuc2} and \eqref{optimization-c-ufuc3} into the minimization problem \eqref{optimization-c-ufuc1}, we get the equivalent form \eqref{optimization-c-ufuc}.
\end{proof}

In the following, we show that the finite dimensional minimization problem \eqref{optimization-c-ufuc} reduces to a nonlinear system in a special case that both $\mathcal{Q}_{\mathbf{y}_0}$ and $\varphi$ are convex and differentiable.

\begin{rem}
For a given  $\mathbf{y}_0\in\mathbb{R}^m$, let $\mathcal{Q}_{\mathbf{y}_0}$ be a convex loss function and $\varphi$ be a strictly increasing and convex regularizer. If $\mathcal{Q}_{\mathbf{y}_0}$ and $\varphi$ are both differentiable, then $\hat{\mathbf{c}}\neq 0$ is the unique solution of the minimization problem \eqref{optimization-c-ufuc} if and only if $\hat{\mathbf{c}}$ is the solution of the nonlinear system
\begin{equation}\label{regu-solve-nonlinear-c}
\nabla\mathcal{Q}_{\mathbf{y}_0}
(\mathbf{G}_{\mathrm{s.i.p}}(\hat{\mathbf{c}}))
+\lambda\frac{\varphi^{'}\left((\hat{\mathbf{c}}^{\top}
\mathbf{G}_{\mathrm{s.i.p}}(\hat{\mathbf{c}}))^{1/2}\right)}
{(\hat{\mathbf{c}}^{\top}\mathbf{G}_{\mathrm{s.i.p}}(\hat{\mathbf{c}}))^{1/2}}
\hat{\mathbf{c}}=0.
\end{equation}
\end{rem}
\begin{proof}
The assumptions about $\mathcal{Q}_{\mathbf{y}_0}$ and $\varphi$ ensure the uniqueness of the solution of the minimization problem \eqref{optimization-c-ufuc}. Note that $\hat{\mathbf{c}}\neq0$ is the solution of the minimization problem \eqref{optimization-c-ufuc} if and only if
\begin{equation}\label{regu-solve-inclusion-relation-c}
0\in\mathcal{L}^*\partial\mathcal{Q}_{\mathbf{y}_0}
\left(\mathcal{L}(\mathcal{L}^*(\hat{\mathbf{c}}))^{\sharp}\right)
+\lambda\partial(\varphi\circ\|\cdot\|_{\mathcal{B}})
\left((\mathcal{L}^*(\hat{\mathbf{c}}))^{\sharp}\right).
\end{equation}
Since $\mathcal{Q}_{\mathbf{y}_0}$ is differentiable, there holds
\begin{equation}\label{regu-solve-nonlinear-c1}
\partial\mathcal{Q}_{\mathbf{y}_0}\left(\mathcal{L}
(\mathcal{L}^*(\hat{\mathbf{c}}))^{\sharp}\right)
=\left\{\nabla\mathcal{Q}_{\mathbf{y}_0}\left(\mathcal{L}
(\mathcal{L}^*(\hat{\mathbf{c}}))^{\sharp}\right)\right\}.
\end{equation}
The linear independence of $g_j^{\sharp}$, $j\in\mathbb{N}_m,$ leads to  $\mathcal{L}^*(\hat{\mathbf{c}})\neq0$. Then by the differentiability of $\varphi$ and equation \eqref{relation-gateaux-dual1} with $g:=(\mathcal{L}^*(\hat{\mathbf{c}}))^{\sharp}$, we have that
\begin{equation}\label{regu-solve-nonlinear-c2}
\partial(\varphi\circ\|\cdot\|_{\mathcal{B}})
\left((\mathcal{L}^*(\hat{\mathbf{c}}))^{\sharp}\right)
=\left\{\frac{\varphi^{'}(\|(\mathcal{L}^*
(\hat{\mathbf{c}}))^{\sharp}\|_{\mathcal{B}})}
{\|(\mathcal{L}^*(\hat{\mathbf{c}}))^{\sharp}\|_{\mathcal{B}}}
\mathcal{L}^*(\hat{\mathbf{c}})\right\}.
\end{equation}
Note that the two sets in the right hand side of \eqref{regu-solve-nonlinear-c1} and \eqref{regu-solve-nonlinear-c2} are singleton. Substituting \eqref{regu-solve-nonlinear-c1} and \eqref{regu-solve-nonlinear-c2} into the \eqref{regu-solve-inclusion-relation-c}, with noticing that
$$
\|(\mathcal{L}^*(\hat{\mathbf{c}}))^{\sharp}\|_{\mathcal{B}}
=\|\mathcal{L}^*(\hat{\mathbf{c}})\|_{\mathcal{B}^*},
$$
we get that $\hat{\mathbf{c}}\neq 0$ is the solution of the minimization problem \eqref{optimization-c-ufuc} if and only if $\hat{\mathbf{c}}$ is the solution of the nonlinear system
$$
\mathcal{L}^*\nabla\mathcal{Q}_{\mathbf{y}_0}
(\mathcal{L}(\mathcal{L}^*(\hat{\mathbf{c}}))^{\sharp})
+\lambda\frac{\varphi^{'}(\|\mathcal{L}^*(\hat{\mathbf{c}})
\|_{\mathcal{B}^*})}{\|\mathcal{L}^*(\hat{\mathbf{c}})
\|_{\mathcal{B}^*}}\mathcal{L}^*(\hat{\mathbf{c}})=0.
$$
Combining \eqref{optimization-c-ufuc2} with \eqref{optimization-c-ufuc3}, we rewrite the above system as
$$
\mathcal{L}^*\nabla\mathcal{Q}_{\mathbf{y}_0}
(\mathbf{G}_{\mathrm{s.i.p}}(\hat{\mathbf{c}}))
+\lambda\frac{\varphi^{'}\left((\hat{\mathbf{c}}^{\top}
\mathbf{G}_{\mathrm{s.i.p}}(\hat{\mathbf{c}}))^{1/2}\right)}
{(\hat{\mathbf{c}}^{\top}\mathbf{G}_{\mathrm{s.i.p}}(\hat{\mathbf{c}}))^{1/2}}
\mathcal{L}^*(\hat{\mathbf{c}})=0,
$$
which together with the linearity of $\mathcal{L}^*$ leads to
$$
\mathcal{L}^*\left(\nabla\mathcal{Q}_{\mathbf{y}_0}
(\mathbf{G}_{\mathrm{s.i.p}}(\hat{\mathbf{c}}))
+\lambda\frac{\varphi^{'}\left((\hat{\mathbf{c}}^{\top}
\mathbf{G}_{\mathrm{s.i.p}}(\hat{\mathbf{c}}))^{1/2}\right)}
{(\hat{\mathbf{c}}^{\top}\mathbf{G}_{\mathrm{s.i.p}}(\hat{\mathbf{c}}))^{1/2}}
\hat{\mathbf{c}}\right)=0.
$$
By the linear independence of $g_j^{\sharp}$, $j\in\mathbb{N}_m$, the above system is equivalent to \eqref{regu-solve-nonlinear-c}.
\end{proof}

The nonlinear system was established in \cite{ZZ} in the case that $\mathcal{B}$ is a semi-inner-product RKBS and for $\mathbf{y}_0:=[y_j:j\in\mathbb{N}_m]\in\mathbb{R}^m$, the loss function $\mathcal{Q}_{\mathbf{y}_0}$ has the form
$$
\mathcal{Q}_{\mathbf{y}_0}(\mathbf{z}):=\sum_{j\in\mathbb{N}_m}\mathcal{Q}_j(z_j,y_i),
\ \ \mbox{for all}\ \ \mathbf{z}:=[z_j:j\in\mathbb{N}_m]\in\mathbb{R}^m,
$$
where $\mathcal{Q}_j:\mathbb{R}\times\mathbb{R}\rightarrow\mathbb{R}_+$, $j\in\mathbb{N}_m,$ are a finite number of bivariate loss functions.

\subsection{Fixed-Point Approach for Regularization Problems}
We develop in this subsection a fixed-point approach for solving the regularization problem in a Banach space.

Following the idea in section 6, we now consider solving directly the regularization problem \eqref{regularization} in a Banach space $\mathcal{B}$. We will characterize the solutions of the problem via fixed-point equations. Again, we need to consider both cases when the loss function $\mathcal{Q}_{\mathbf{y}_0}$ is differentiable and when it is not differentiable.

\begin{thm}\label{characterize-prox-regularization}
Suppose that $\mathcal{B}$ is a Banach space with the dual space $\mathcal{B}^*$, $\nu_j\in\mathcal{B}^*$, $j\in\mathbb{N}_m$ and that $\mathcal{L}$ is defined by \eqref{functional-operator}, $\mathcal{L}^*$ is the adjoint operator of $\mathcal{L}$ and $\mathcal{V}$ is defined by \eqref{linear-span-nu-j}. Let $\mathcal{H}$ be a Hilbert space and $\mathcal{T}$ a bounded linear operator from $\mathcal{B}$ to $\mathcal{H}$ such that $\mathcal{T}^*\mathcal{T}$ is a one-to-one mapping from the inverse image of $\mathcal{V}$ onto $\mathcal{V}$. For a given $\mathbf{y}_0\in\mathbb{R}^m$, let $\mathcal{Q}_{\mathbf{y}_0}:\mathbb{R}^m\rightarrow\mathbb{R}_{+}$ be a convex loss function, $\varphi:\mathbb{R}_{+}\rightarrow\mathbb{R}_{+}$ be a convex regularizer and $\lambda>0$. Then $f_0\in\mathcal{B}$ is a solution of the regularization problem \eqref{regularization} with $\mathbf{y}:=\mathbf{y}_0$ if and only if there exists $\hat{\mathbf{c}}\in\mathbb{R}^m$ such that
\begin{equation}\label{characterize-prox-regularization-formula1}
\hat{\mathbf{c}}=\mathrm{prox}_{\mathcal{Q}_{\mathbf{y}_0}^*}
(\hat{\mathbf{c}}+\mathcal{L}(f_0))
\end{equation}
and
\begin{equation}\label{characterize-prox-regularization-formula2}
f_0=\mathrm{prox}_{\varphi\circ\|\cdot\|_{\mathcal{B}},
\mathcal{H},\mathcal{T}}\left(f_0
-\frac{1}{\lambda}(\mathcal{T}^*\mathcal{T})^{-1}
\mathcal{L}^*(\hat{\mathbf{c}})\right).
\end{equation}
\end{thm}
\begin{proof}
By employing the Fermat rule, we have that $f_0\in\mathcal{B}$ is a solution of the regularization problem \eqref{regularization} with $\mathbf{y}:=\mathbf{y}_0$ if and only if
$$
0\in\partial(\mathcal{Q}_{\mathbf{y}_0}(\mathcal{L}(\cdot))
+\lambda\varphi\circ\|\cdot\|_{\mathcal{B}})(f_0).
$$
According to the chain rule \eqref{chain-rule} of the subdifferential, the above inclusion relation can be rewritten as
\begin{equation*}
0\in\mathcal{L}^*\partial\mathcal{Q}_{\mathbf{y}_0}(\mathcal{L}(f_0))
+\lambda\partial(\varphi\circ\|\cdot\|_{\mathcal{B}})(f_0).
\end{equation*}
This is equivalent to that there exists $\hat{\mathbf{c}}\in\mathbb{R}^m$ such that
\begin{equation}\label{characterize-prox-regularization-1}
\hat{\mathbf{c}}\in\partial\mathcal{Q}_{\mathbf{y}_0}(\mathcal{L}(f_0))
\end{equation}
and
\begin{equation}\label{characterize-prox-regularization-2}
-\frac{1}{\lambda}\mathcal{L}^*(\hat{\mathbf{c}})
\in\partial(\varphi\circ\|\cdot\|_{\mathcal{B}})(f_0).
\end{equation}
Relation \eqref{conjugate} ensures that the inclusion relation \eqref{characterize-prox-regularization-1} holds if and only if
$$
\mathcal{L}(f_0)\in\partial\mathcal{Q}_{\mathbf{y}_0}^*(\hat{\mathbf{c}}),
$$
which is equivalent to \eqref{characterize-prox-regularization-formula1}. Since
$-\frac{1}{\lambda}\mathcal{L}^*(\hat{\mathbf{c}})\in\mathcal{V}$, we represent the inclusion relation \eqref{characterize-prox-regularization-2} as
\begin{equation*}
(\mathcal{T}^*\mathcal{T})(\mathcal{T}^*\mathcal{T})^{-1}
\left(-\frac{1}{\lambda}\mathcal{L}^*(\hat{\mathbf{c}})\right)
\in\partial(\varphi\circ\|\cdot\|_{\mathcal{B}})(f_0).
\end{equation*}
By Proposition \ref{subdiff-prox-Banach}, we conclude that the above relation is equivalent to \eqref{characterize-prox-regularization-formula2}. Therefore, $f_0\in\mathcal{B}$ is a solution of the regularization problem \eqref{regularization} with $\mathbf{y}:=\mathbf{y}_0$ if and only if there exists $\hat{\mathbf{c}}\in\mathbb{R}^m$ satisfying the fixed-point equations \eqref{characterize-prox-regularization-formula1} and \eqref{characterize-prox-regularization-formula2}.
\end{proof}

In the special case when the loss function $\mathcal{Q}_{\mathbf{y}_0}$ is differentiable, the fixed-point equations \eqref{characterize-prox-regularization-formula1} and \eqref{characterize-prox-regularization-formula2} can reduce to only one fixed-point equation.

\begin{cor}\label{characterize-prox-regularization-RN}
Suppose that the hypotheses of Theorem \ref{characterize-prox-regularization} hold. If in addition the loss function $\mathcal{Q}_{\mathbf{y}_0}:\mathbb{R}^m\rightarrow\mathbb{R}_{+}$ is differentiable, then $f_0\in\mathcal{B}$ is a solution of the regularization problem \eqref{regularization} with $\mathbf{y}:=\mathbf{y}_0$ if and only if it satisfies the fixed-point equation
\begin{equation}\label{characterize-prox-regularization-formula-smooth}
f_0=\mathrm{prox}_{\varphi\circ\|\cdot\|_{\mathcal{B}},
\mathcal{H},\mathcal{T}}\left(f_0
-\frac{1}{\lambda}(\mathcal{T}^*\mathcal{T})^{-1}
\mathcal{L}^*\nabla\mathcal{Q}_{\mathbf{y}_0}\mathcal{L}(f_0)\right).
\end{equation}
\end{cor}
\begin{proof}
Theorem \ref{characterize-prox-regularization} ensures that $f_0\in\mathcal{B}$ is a solution of the regularization problem \eqref{regularization} with $\mathbf{y}:=\mathbf{y}_0$ if and only if  there exists $\hat{\mathbf{c}}\in\mathbb{R}^m$ satisfying \eqref{characterize-prox-regularization-formula1} and \eqref{characterize-prox-regularization-formula2}. Note that equation \eqref{characterize-prox-regularization-formula1} is equivalent to the inclusion relation \eqref{characterize-prox-regularization-1}. Since $\mathcal{Q}_{\mathbf{y}_0}$ is  differentiable, the subdifferential of $\mathcal{Q}_{\mathbf{y}_0}$ at $\mathcal{L}(f_0)$ is the singleton $\nabla\mathcal{Q}_{\mathbf{y}_0}(\mathcal{L}(f_0))$. That is, there holds
$$
\hat{\mathbf{c}}=\nabla\mathcal{Q}_{\mathbf{y}_0}(\mathcal{L}(f_0)).
$$
Substituting this equation into \eqref{characterize-prox-regularization-formula2} leads to \eqref{characterize-prox-regularization-formula-smooth}.
\end{proof}

In particular, for the learning network problem, in which the loss function $\mathcal{Q}_{\mathbf{y}_0}$ has the form \eqref{RN}, we have the following special result.

\begin{rem}
If for a given $\mathbf{y}_0\in\mathbb{R}^m$, the loss function $\mathcal{Q}_{\mathbf{y}_0}$ has the form \eqref{RN}, the fixed-point equation \eqref{characterize-prox-regularization-formula-smooth} reduces to
\begin{equation}\label{characterize-prox-regularization-formula-RN}
f_0=\mathrm{prox}_{\varphi\circ\|\cdot\|_{\mathcal{B}},
\mathcal{H},\mathcal{T}}\left(f_0
-\frac{2}{\lambda}(\mathcal{T}^*\mathcal{T})^{-1}
\mathcal{L}^*(\mathcal{L}(f_0)-\mathbf{y}_0)\right).
\end{equation}
\end{rem}
\begin{proof}
Since $\mathcal{Q}_{\mathbf{y}_0}$ has the form \eqref{RN}, there holds
$$
\nabla\mathcal{Q}_{\mathbf{y}_0}(\mathcal{L}(f_0))=2(\mathcal{L}(f_0)-\mathbf{y}_0).
$$
Substituting the equation above into \eqref{characterize-prox-regularization-formula-smooth} leads to \eqref{characterize-prox-regularization-formula-RN}.
\end{proof}

According to Theorem \ref{characterize-prox-regularization} the solution of the regularization problem \eqref{regularization} can be obtained by solving fixed-point equations  \eqref{characterize-prox-regularization-formula1} and \eqref{characterize-prox-regularization-formula2} or \eqref{characterize-prox-regularization-formula-smooth}. Note that either equations \eqref{characterize-prox-regularization-formula2} or  \eqref{characterize-prox-regularization-formula-smooth} is of infinite dimension. In section 6, we have demonstrated that a solution of the minimum norm interpolation \eqref{mni} with $\mathcal{B}=\ell_1(\mathbb{N})$ can be formulated as finite dimensional fixed-point equations. We next show that a solution of the regularization problem \eqref{regularization} with $\mathcal{B}=\ell_1(\mathbb{N})$ can also be formulated as finite dimensional fixed-point equations. The regularization problem in the space $\ell_1(\mathbb{N})$ has the form
\begin{equation}\label{regularized-learning-l1}
\inf\{\mathcal{Q}_{\mathbf{y}}(\mathcal{L}(\mathbf{x}))
+\lambda\|\mathbf{x}\|_1:\mathbf{x}\in\ell_1(\mathbb{N})\}.
\end{equation}

With the help of Lemma \ref{improtant-character2}, a solution of the regularization problem \eqref{regularized-learning-l1} can be characterized via finite dimensional fixed-point equations as follows.

\begin{thm}\label{characterize-prox-regularization-l1}
Suppose that $\mathbf{u}_j\in c_0$, $j\in\mathbb{N}_m,$, $\mathcal{L}$ is defined by \eqref{functional-operator-l1} and $\mathcal{L}^*$ is the adjoint operator. Let $\mathcal{T}_0$ be defined by \eqref{T-l1} and $\mathcal{S}$ be defined by \eqref{truncation} and \eqref{truncation-vector}. For a given $\mathbf{y}_0\in\mathbb{R}^m$, let $\mathcal{Q}_{\mathbf{y}_0}:\mathbb{R}^m\rightarrow\mathbb{R}_{+}$ be a convex loss function and $\lambda>0$. Then $\mathbf{x}_0\in\ell_1(\mathbb{N})$ is a solution of the regularization problem \eqref{regularized-learning-l1} with $\mathbf{y}:=\mathbf{y}_0$ if and only if there exists $\hat{\mathbf{c}}\in\mathbb{R}^m$ such that
\begin{equation}\label{characterize-prox-regularization-formula-l1-1}
\hat{\mathbf{c}}=\mathrm{prox}_{\mathcal{Q}_{\mathbf{y}_0}^*}
(\hat{\mathbf{c}}+\mathcal{L}(\mathbf{x}_0))
\end{equation}
and
\begin{equation}\label{characterize-prox-regularization-formula-l1-2}
\mathbf{x}_0=\mathrm{prox}_{\|\cdot\|_1,
\ell_2(\mathbb{N}),\mathcal{T}_0}\left(\mathbf{x}_0
-\frac{1}{\lambda}\mathcal{S}\mathcal{L}^*(\hat{\mathbf{c}})\right).
\end{equation}
\end{thm}
\begin{proof}
As has been shown in the proof of Theorem \ref{characterize-prox-regularization}, $\mathbf{x}_0\in\ell_1(\mathbb{N})$ is a solution of the regularization problem \eqref{regularized-learning-l1} with $\mathbf{y}:=\mathbf{y}_0$ if and only if
there exists $\hat{\mathbf{c}}\in\mathbb{R}^m$ such that
\begin{equation}\label{characterize-prox-regularization-l1-1}
\hat{\mathbf{c}}\in\partial\mathcal{Q}_{\mathbf{y}_0}(\mathcal{L}(\mathbf{x}_0))
\end{equation}
and
\begin{equation}\label{characterize-prox-regularization-l1-2}
-\frac{1}{\lambda}\mathcal{L}^*(\hat{\mathbf{c}})\in\partial\|\cdot\|_1(\mathbf{x}_0).
\end{equation}
By relation \eqref{conjugate} between the subdifferentials of $\mathcal{Q}_{\mathbf{y}_0}$ and its conjugate $\mathcal{Q}_{\mathbf{y}_0}^*$, we rewrite the inclusion relation \eqref{characterize-prox-regularization-l1-1} as
$$
\mathcal{L}(\mathbf{x}_0)\in\partial\mathcal{Q}_{\mathbf{y}_0}^*(\hat{\mathbf{c}}),
$$
which is equivalent to \eqref{characterize-prox-regularization-formula-l1-1}. Lemma \ref{improtant-character2} ensures that the inclusion relation \eqref{characterize-prox-regularization-l1-2} is equivalent to
\begin{equation*}
-\frac{1}{\lambda}\mathcal{S}(\mathcal{L}^*(\hat{\mathbf{c}}))
\in\partial\|\cdot\|_1(\mathbf{x}_0).
\end{equation*}
Relation \eqref{subdiff-prox-l1} with $\psi:=\|\cdot\|_1$ leads to the equivalence between the above relation and  equation \eqref{characterize-prox-regularization-formula-l1-2}.
\end{proof}

If the loss function $\mathcal{Q}_{\mathbf{y}_0}$ is differentiable, the solution of the regularization problem \eqref{regularized-learning-l1} with $\mathbf{y}:=\mathbf{y}_0$ can be characterized via a single fixed-point equation. We present this result in the next corollary.

\begin{cor}\label{characterize-prox-regularization-l1-RN}
Suppose that the hypotheses of Theorem \ref{characterize-prox-regularization-l1} hold. If in addition the loss function $\mathcal{Q}_{\mathbf{y}_0}:\mathbb{R}^m\rightarrow\mathbb{R}_{+}$ is differentiable, then $\mathbf{x}_0\in\ell_1(\mathbb{N})$ is a solution of the regularization problem \eqref{regularized-learning-l1} with $\mathbf{y}:=\mathbf{y}_0$ if and only if
\begin{equation}\label{characterize-prox-regularization-formula-l1-smooth}
\mathbf{x}_0=\mathrm{prox}_{\|\cdot\|_1,
\ell_2(\mathbb{N}),\mathcal{T}_0}\left(\mathbf{x}_0
-\frac{1}{\lambda}\mathcal{S}\mathcal{L}^*
\nabla\mathcal{Q}_{\mathbf{y}_0}\mathcal{L}(\mathbf{x}_0)\right).
\end{equation}
\end{cor}
\begin{proof}
By Theorem \ref{characterize-prox-regularization-l1}, $\mathbf{x}_0\in\ell_1(\mathbb{N})$ is a solution of the regularization problem \eqref{regularized-learning-l1} with $\mathbf{y}:=\mathbf{y}_0$ if and only if  there exists $\hat{\mathbf{c}}\in\mathbb{R}^m$ satisfying \eqref{characterize-prox-regularization-formula-l1-1} and \eqref{characterize-prox-regularization-formula-l1-2}. As has been shown in the proof of Theorem \ref{characterize-prox-regularization-l1}, equation \eqref{characterize-prox-regularization-formula-l1-1} is equivalent to the inclusion relation \eqref{characterize-prox-regularization-l1-1}. Since $\mathcal{Q}_{\mathbf{y}_0}$ is differentiable, we have that
$$
\partial\mathcal{Q}_{\mathbf{y}_0}(\mathcal{L}(\mathbf{x}_0))
=\{\nabla\mathcal{Q}_{\mathbf{y}_0}(\mathcal{L}(\mathbf{x}_0))\}.
$$
Substituting the above equation into \eqref{characterize-prox-regularization-l1-1} leads to
$$
\hat{\mathbf{c}}=\nabla\mathcal{Q}_{\mathbf{y}_0}(\mathcal{L}(\mathbf{x}_0)),
$$
which together with \eqref{characterize-prox-regularization-formula-l1-2} leads to \eqref{characterize-prox-regularization-formula-l1-smooth}.
\end{proof}

Once again, for the learning network problem, in which the loss function $\mathcal{Q}_{\mathbf{y}_0}$ has the form \eqref{RN}, we have the following special result.

\begin{rem}
If for a given $\mathbf{y}_0\in\mathbb{R}^m$, the loss function $\mathcal{Q}_{\mathbf{y}_0}$ has the form \eqref{RN}, the fixed-point equation  \eqref{characterize-prox-regularization-formula-l1-smooth} reduces to
\begin{equation}\label{characterize-prox-regularization-formula-l1-RN}
\mathbf{x}_0=\mathrm{prox}_{\|\cdot\|_1,
\ell_2(\mathbb{N}),\mathcal{T}_0}\left(\mathbf{x}_0
-\frac{2}{\lambda}\mathcal{S}\mathcal{L}^*(\mathcal{L}(\mathbf{x}_0)-\mathbf{y}_0)\right).
\end{equation}
\end{rem}

Below, we point out the finite dimensional component of the fixed-points equations \eqref{characterize-prox-regularization-formula-l1-1} and \eqref{characterize-prox-regularization-formula-l1-2}, the same as those for the minimum morn interpolation stated in Theorem \ref{finite-dimension-fixed-point}. To see this, we rewrite the fixed-point equations \eqref{characterize-prox-regularization-formula-l1-1} and \eqref{characterize-prox-regularization-formula-l1-2} in the following compact form
\begin{equation}\label{fixed-point-l1-regularization}
\mathbf{s}_r=(\mathcal{P}_r\circ\mathcal{R}_r)(\mathbf{s}_r),
\end{equation}
where $\mathbf{s}_r$ denotes the vector
\begin{equation*}
\mathbf{s}_r:=\left[\begin{array}{c}
\hat{\mathbf{c}}\\ \mathbf{x}_0
\end{array}
\right].
\end{equation*}
and two matrices $\mathcal{P}_r$ and $\mathcal{R}_r$ of operators have the form
\begin{equation}\label{definitionofP-regularization}
\mathcal{P}_r:=\left[\begin{array}{c}
\mathrm{prox}_{\mathcal{Q}_{\mathbf{y}_0}^*}\\
\mathrm{prox}_{\|\cdot\|_1,\ell_2(\mathbb{N}),\mathcal{T}_0}
\end{array}
\right]
\end{equation}
and
\begin{equation}\label{definitionofR-regularization}
\mathcal{R}_r:=\left[\begin{array}{cc}
\mathcal{I}& \mathcal{L}\\
-\frac{1}{\lambda}\mathcal{S}\mathcal{L}^* &\mathcal{I}
\end{array}
\right].
\end{equation}

We show in the following theorem that the fixed-point equation \eqref{fixed-point-l1-regularization} (or equivalently the system of the fixed-points equations \eqref{characterize-prox-regularization-formula-l1-1} and \eqref{characterize-prox-regularization-formula-l1-2} is of finite dimension.

\begin{thm}\label{finite-dimension-fixed-point-regularization}
If operators $\mathcal{P}_r$ and $\mathcal{R}_r$ are defined respectively by \eqref{definitionofP-regularization} and \eqref{definitionofR-regularization}, then $\mathcal{P}_r\circ\mathcal{R}_r$ is an operator from $(\mathbb{R}^m,\ell_1(\mathbb{N}))$ to $(\mathbb{R}^m, c_c)$ and its fixed-point $\mathbf{s}_r=\left[\begin{array}{c}
\hat{\mathbf{c}}\\ \mathbf{x}_0\end{array}\right]
\in(\mathbb{R}^m,\ell_1(\mathbb{N}))$  satisfies
\begin{equation}\label{support1-regularization}
\mathbf{x}_0\in c_c\ \ \mbox{and}\ \ \mathrm{supp}(\mathbf{x}_0)\subseteq
\mathrm{supp}(\mathcal{S}(\mathcal{L}^*(\hat{\mathbf{c}})).
\end{equation}
\end{thm}
\begin{proof}
Note that similar to the proximity operator $\mathrm{prox}_{\iota_{\mathbf{y}}^*}$, the proximity  operator $\mathrm{prox}_{\mathcal{Q}_{\mathbf{y}_0}^*}$ is a mapping from $\mathbb{R}^m$ to itself. Then by similar arguments in the proof of Theorem \ref{finite-dimension-fixed-point}, we get the desired conclusion for the operator $\mathcal{P}_r\circ\mathcal{R}_r$.
\end{proof}

A solution of the regularization problem \eqref{regularized-learning-l1} with $\mathcal{B}:=\ell_1(\mathbb{N})$ guaranteed by Theorems \ref{characterize-prox-regularization-l1} and \ref{finite-dimension-fixed-point-regularization} has an additional property.

\begin{rem}\label{Remark-on-support-regularization}
Each solution $\mathbf{x}_0\in\ell_1(\mathbb{N})$ of the regularization problem \eqref{regularized-learning-l1} together with $\hat{\mathbf{c}}\in\mathbb{R}^m$ satisfying the fixed-point equations \eqref{characterize-prox-regularization-formula-l1-1} and \eqref{characterize-prox-regularization-formula-l1-2} is of finite dimension, that is, it satisfies \eqref{support1-regularization}.
\end{rem}

Theorem \ref{characterize-prox-regularization-l1} provides a theoretical foundation for algorithmic development for solving the regularization problem \eqref{regularization} with $\mathcal{B}:=\ell_1(\mathbb{N})$. Specifically, the fixed-point equations \eqref{characterize-prox-regularization-formula-l1-1} and \eqref{characterize-prox-regularization-formula-l1-2} on the finite dimensional space will serve as a starting point to design efficient fixed-point iterative algorithms. We postpone further algorithmic development for a future project.

Below, we comment on closed-form formulas for the proximity operator of loss functions, required to find a fixed-point according to equations \eqref{characterize-prox-regularization-formula-l1-1} and \eqref{characterize-prox-regularization-formula-l1-2}. The closed-form of $\mathrm{prox}_{\|\cdot\|_1,\ell_2(\mathbb{N}),\mathcal{T}_0}$ has been given in \eqref{prox-1norm}. When the loss function is not differentiable, we also need a closed-form formula for its proximity operator.  The proximity operator of certain commonly used loss functions can also be computed explicitly \cite{Li-Song-Xu2019}. For example, if $\mathcal{Q}_{\mathbf{y}}$ is defined as in \eqref{SVMC}, the proximity operator $\mathrm{prox}_{\mathcal{Q}_{\mathbf{y}}}$ at $\mathbf{a}:=[a_j:j\in\mathbb{N}_m]$ has the form
$$
\mathrm{prox}_{\mathcal{Q}_{\mathbf{y}}}(\mathbf{a}):=[b_j:j\in\mathbb{N}_m],
$$
where for all $j\in\mathbb{N}_m$
$$
b_j:=\left\{\begin{array}{ll}
y_j^2a_j,& \mbox{if}\ \ 1\leq y_ja_j,\\
y_j,& \mbox{if}\ \ 0\leq y_ja_j\leq 1,\\
y_j(y_ja_j+1),& \mbox{if}\ \ y_ja_j< 0.
\end{array}
\right.
$$
If $\mathcal{Q}_{\mathbf{y}}$ is defined as in \eqref{SVMR}, we present the proximity operator $\mathrm{prox}_{\mathcal{Q}_{\mathbf{y}}}$ at $\mathbf{a}:=[a_j:j\in\mathbb{N}_m]$ as follows.
If $\epsilon\geq1/2$, then for all $j\in\mathbb{N}_m$
$$
b_j:=\left\{\begin{array}{ll}
a_j-1,& \mbox{if}\ \ \epsilon+1+y_j\leq a_j,\\
\epsilon+y_j,& \mbox{if}\ \ \epsilon+y_j\leq a_j<\epsilon+1+y_j,\\
a_j,& \mbox{if}\ \ \epsilon-1+y_j\leq a_j<\epsilon+y_j,\\
a_j+1& \mbox{if}\ \ -\epsilon+y_j\leq a_j<\epsilon-1+y_j,\\
-\epsilon+y_j,& \mbox{if}\ \ -\epsilon-1+y_j\leq a_j<-\epsilon+y_j,\\
a_j+1,&\mbox{if}\ \ a_j< -\epsilon-1+y_j.
\end{array}
\right.
$$
If $\epsilon<1/2$, then for all $j\in\mathbb{N}_m$
$$
b_j:=\left\{\begin{array}{ll}
a_j-1,& \mbox{if}\ \ \epsilon+1+y_j\leq a_j,\\
\epsilon+y_j,& \mbox{if}\ \ \epsilon+y_j\leq a_j<\epsilon+1+y_j,\\
a_j,& \mbox{if}\ \ -\epsilon+y_j\leq a_j<\epsilon+y_j,\\
-\epsilon+y_j,& \mbox{if}\ \ -\epsilon-1+y_j\leq a_j<-\epsilon+y_j,\\
a_j+1,&\mbox{if}\ \ a_j< -\epsilon-1+y_j.
\end{array}
\right.
$$

As we have explained earlier, whether the fixed-point equations defined via the proximity operator will lead to efficient iterative algorithms pretty much depends on whether the proximity operators involved have closed-form formulas. However, it is not realistic to expect that the proximity operator of an arbitrary function has a closed-form formula. Although we have closed-form formulas for the proximity operators of a class of simple functions, it requires further research to establish such formulas for proximity operators of various functions which appear in practical applications. Along this line, calculus of the proximity operator is interesting and useful. Furthermore, numerical computation of the proximity operators of functions which appear in important applications but have no closed-form formulas deserves investigation.


\section{Conclusions}

To conclude this paper, we briefly summarize the major mathematical contributions made in this paper toward solutions of the minimum norm interpolation problem and the related regularization problem in a Banach space.
The main contributions of this paper include the following five aspects:
\begin{itemize}
\item We have provided a systematic study of the representer theorems for a solution of the minimum norm interpolation problem, and the regularization problem in a Banach space. Both functional analytic and convex analytic approaches are used to achieve this goal.
\item We have established explicit solution representations for the minimum norm interpolation problem and the regularization problem in a Banach space which has a dual space in both cases smooth or non-smooth.
\item We have developed approaches to determine the coefficients appearing in solution representations of these problems, leading to solution methods for solving these problems. Specifically, the coefficients appearing in the solution representations can be determined by solving a linear system, a nonlinear system, or a finite dimensional optimization problem.
\item We have expressed the infimum of the minimum norm interpolation in a Banach space in terms of the interpolation functionals. This is established by using its solution representations and properties of the subdifferential of a norm function of the Banach space.
\item We have observed that although the representer theorems for a solution of these problems in a Banach space convert the originally infinite dimensional problems to a finite dimensional problem, unlike in a Hilbert space where the resulting linear system is truly finite dimensional, the resulting finite dimensional problem has certain hidden infinite dimensional components. We have demonstrated a way to overcome this challenge in the special case when the Banach space is $\ell_1(\mathbb{N})$.
\end{itemize}

Developing efficient computational algorithms based on the solution representations provided by this paper requires further investigation. Nevertheless, the theory established here furnishes a solid mathematical foundation for this practical goal.

\bigskip

\noindent{\bf Acknowledgment:}
R. Wang is supported by the Fundamental Research Funds for the Central Universities and by the Opening Project of Guangdong Province Key Laboratory of Computational Science at the Sun Yat-sen University under grant 2018005.
Y. Xu is supported by US National Science Foundation under grant DMS-1912958 and by Natural Science Foundation of China under grant 11771464. He is a Professor Emeritus of Mathematics at Syracuse University, Syracuse, New York 13244, USA.

\end{document}